\documentclass[11pt,oneside,reqno]{amsart}
\usepackage{amsmath,amssymb,amsthm}
\usepackage{thmtools,mathtools}
\usepackage{hyperref} 
\usepackage{mathtools}

\usepackage[english]{babel}
\usepackage{color}

\theoremstyle{plain}
\newtheorem{theorem}{Theorem}
\newtheorem{lemma}{Lemma}
\newtheorem{proposition}{Proposition}
\newtheorem{corollary}{Corollary}

\theoremstyle{definition}
\newtheorem{definition}{Definition}
\newtheorem{example}{Example}

\theoremstyle{remark}
\newtheorem{remark}{Remark}
\newtheorem{assumption}{Assumption}

\usepackage{extarrows}
\usepackage[page,toc,titletoc,title]{appendix}
\usepackage{graphicx}
\usepackage{xcolor}
\usepackage{hyperref}
\usepackage{comment}

\DeclareRobustCommand{\gobblefive}[5]{}

\definecolor{darktaupe}{rgb}{0.24, 0.08, 0.08}
\definecolor{blueN}{rgb}{0.00, 0.75, 1.00}
\hypersetup{colorlinks=true,
            breaklinks,
            linkcolor=blueN,
            citecolor=green,
            urlcolor=blueN}

\makeatletter
\newcommand{\proofpart}[2]{%
  \par
  \addvspace{\medskipamount}%
  \noindent\emph{Part #1: #2}\par\nobreak
  \addvspace{\smallskipamount}%
  \@afterheading
}
\makeatother

\usepackage[nameinlink, noabbrev, capitalise]{cleveref}
\usepackage{algorithm,algorithmic,matlab-prettifier}
\usepackage{enumitem}
\usepackage[font=small,skip=2pt]{caption,subcaption}
\usepackage{multirow,enumitem}
\usepackage[group-digits=false]{siunitx}
\usepackage{booktabs}
\usepackage{soul}
\usepackage[colorinlistoftodos,prependcaption,textsize=tiny]{todonotes}

\newcommand{\R}{\mathbb R}

\newcommand{\cF}{\mathcal F}
\newcommand{\cS}{\mathcal S}

\newcommand{\eps}{\varepsilon}

\title[Random Spectral Radius for a Matrix Family]{Probabilistic Analysis of the Random Spectral Radius for a Matrix Family}

\author[F.P. Maiale, A. Trofimova, N. Guglielmi]{Francesco Paolo Maiale, Anastasiia Trofimova, Nicola Guglielmi}

\address {Francesco Paolo Maiale \newline \indent
	Gran Sasso Science Institute \newline \indent
	Viale Rendina 26--28, L'Aquila, 67100, Italy}
\email{francescopaolo.maiale@gssi.it}

\address {Anastasiia Trofimova \newline \indent
	Gran Sasso Science Institute \newline \indent
	Viale Rendina 26--28, L'Aquila, 67100, Italy}
\email{anastasiia.trofimova@gssi.it}

\address {Nicola Guglielmi \newline \indent
	Gran Sasso Science Institute \newline \indent
	Viale Rendina 26--28, L'Aquila, 67100, Italy}
\email{nicola.guglielmi@gssi.it}

\begin{document}


\begin{abstract}
We investigate joint spectral characteristics of a family of matrices $\mathcal F $, associated with products in the semigroup generated by $\mathcal F$. In the literature, extremal measures such as the well-known joint spectral radius and the lower spectral radius have been extensively studied. However, these measures fail to capture the typical growth rate of matrix products, focusing instead on the worst and best-case scenarios. Nevertheless, when examining, for instance, a switching dynamical system, a probabilistic rate of growth, which characterizes typical trajectories, emerges as a highly intriguing and significant measure.

In this article, we study the \textbf{random spectral radius}, defined as the spectral radius of a length-$n$ product sampled at random from the semigroup according to a given probability measure. We establish asymptotic results, namely a Law of Large Numbers and a Central Limit Theorem, for diagonal (equivalently, commuting), upper- or lower-triangular, and small perturbations of diagonal matrices. Beyond recovering the correct scaling, we obtain exact closed-form expressions for the limiting value and variance, which may be useful in numerical applications requiring precise constants. In the coalescence regime, where the leading eigenvalues merge, the limiting distribution is non-Gaussian: it is given by the maximum of a correlated Gaussian vector with explicit covariance structure. This phenomenon governs phase transitions between distinct growth regimes in switching systems.
\end{abstract}

\maketitle
\setcounter{tocdepth}{1}

\section{Introduction} 

The study of the \emph{joint spectral characteristics} of a bounded family of matrices 
\[
\cF = \bigl\{ A(\vartheta) \bigr\}_{\vartheta \in \Theta}
\]
(with $\Theta$ a $p$-dimensional compact set) concerns the spectral radii of all finite products
\[
P = A_{\vartheta_k} A_{\vartheta_{k-1}} \cdots A_{\vartheta_1}, \qquad \vartheta_1, \ldots, \vartheta_k \in \Theta, \quad k \ge 1
\]
which generate the \emph{product semigroup} associated with $\mathcal{F}$. In most cases $\mathcal{F}$ is assumed to be finite, in which case one can identify $\Theta$ as $\{ 1,2,\ldots,m\}$.

Among the various spectral characteristics of $\mathcal{F}$, the most classical one is the \emph{joint spectral radius} (JSR), first introduced by Rota and Strang~\cite{RotaStrang1960}. Intuitively, the JSR measures the \emph{maximum possible rate of growth} of products generated by $\mathcal{F}$. For a single matrix $A$, in which case the product semigroup is simply given by all powers of $A$, this quantity coincides with its spectral radius $\rho(A)$; however, for a family of matrices, the situation is considerably more complex, since one must consider all unordered and repeated products in the semigroup. The crucial feature of the JSR, which was proved by Berger and Wang~\cite{BergerWang1992}, is that a necessary and sufficient condition to guarantee that all products of degree $k$ (that is, formed by the products of $k$ matrices taken from $\mathcal{F}$), converge to zero as $k \to \infty$ is that the JSR is smaller than $1$, in complete analogy to the property, for a single matrix $A$, $\rho(A)<1$.

The JSR plays a central role in a variety of applications, including discrete switched systems (see, e.g., Shorten et al.~\cite{MaSh07}), convergence analysis of subdivision schemes and wavelets (see, e.g., Daubechies and Lagarias~\cite{DauLag91}, Hechler, M\"o\ss ner and Reif~\cite{HMR09}, Villemoes~\cite{Vill}, and~\cite{GugPro16}), functional equations, coding theory (see, e.g., Moision, Orlitsky and Siegel~\cite{MOS01}), and combinatorics (see, e.g., Blondel, Jungers, and Protasov~\cite{JPB06}).

A well-known application of the JSR arises in the stability analysis of variable-step numerical methods for differential equations, where the problem often reduces to studying linear difference equations with variable coefficients. In this context, one considers a parameterized family $\cF = \{A(\vartheta)\}_{\vartheta \in \Theta}$, where $\vartheta$ represents parameters such as stepsizes. The asymptotic behavior of solutions depends on the convergence properties of products generated by $\cF$. This is expressed by the discrete system
\[
x_k = A(\vartheta_k) x_{k-1}, \qquad k \ge 1
\]
with an initial vector $x_0 \in \R^d$. Here $\vartheta_k \in \Theta$ denotes the parameter value at step $k$, often determined by an adaptive stepsize strategy. The solution satisfies $x_k = P_k x_0$ where $P_k = A(\vartheta_k) A(\vartheta_{k-1}) \cdots A(\vartheta_1)$. A sufficient condition for asymptotic stability (i.e., convergence of solutions to zero) is the existence of a norm such that $\| A(\vartheta) \| < 1$ for all $\vartheta \in \Theta$. However, finding a suitable norm in which all matrices are contractive can lead to overly restrictive stability conditions. A more general and powerful criterion is obtained through the JSR $\rho(\mathcal{F})$, since all products of matrices in $\mathcal{F}$ vanish asymptotically if and only if $\rho(\mathcal{F}) < 1$ (see Jungers~\cite{J2009} for a comprehensive discussion). The JSR-based stability condition is often \textbf{too conservative} for probabilistic questions: even when $\rho(\mathcal{F}) > 1$, typical randomly generated trajectories may still decay, or remain stable with high probability, because the JSR captures only the \emph{worst-case} growth scenario, which is associated with a very specific parameter selection in the generation of an infinite product.

In summary, both criteria are not necessary (but only sufficient) to guarantee the vanishing behavior of a product sequence $\{ P_k \}_{k \ge 1}$ as $k \to \infty$ and it would be interesting to characterize in a probabilistic way the asymptotic stability behavior of typical sequences (instead of worst-case ones, as the JSR analysis would provide).

\smallskip

A complementary quantity is the \emph{lower spectral radius} (LSR), introduced by Gurvits~\cite{Gu}, which measures the minimum possible rate of growth of matrix products in $\cF$. The LSR serves as a crucial stability measure with applications in control theory~\cite{GuLaPr17,JuMa17}, where it reveals the most stable achievable trajectories and provides sharp bounds in stabilization problems. It also appears in the computation of lower and upper bounds for Euler partition functions (see, e.g.,~\cite{Eul45,Kru50}), and in the analysis of refinement equations (see, e.g.,~\cite{Pro17}). Again, the crucial property of the LSR is that--if this is smaller than $1$--there exist infinite products converging to zero.

While the JSR and LSR capture extremal growth behaviors, they offer limited information on the typical or average dynamics of a system. To bridge this gap, we introduce the concept of the \emph{random spectral radius} (RSR) of a family of matrices, defined as the probabilistic asymptotic growth rate of products generated by random sampling according to a prescribed probability distribution.

The RSR thus complements the JSR and LSR by providing a probabilistic characterization of the system's growth: it describes what happens along typical random switching sequences, rather than the extreme best- or worst-case scenarios.

When the LSR is smaller than $1$ and the JSR is larger than $1$, the product semigroup contains both vanishing and unbounded sequences. In such cases, the RSR provides a probabilistic lens for addressing fundamental questions:
\begin{itemize}
\item How likely is a randomly switched system to exhibit the worst-case growth rate (JSR), corresponding to unstable behavior?
\item Do most trajectories remain bounded, even when the JSR exceeds $1$ and unstable trajectories exist? And do most trajectories stay near the most stable regime (LSR), indicating robust stability?
\item How concentrated or dispersed are the growth rates within these bounds?
\end{itemize}

Understanding this is essential in probabilistic stability analysis, where one seeks not only worst-case guarantees but also likelihood estimates of various behaviors under random or uncertain switching. In control design, such insights inform the construction of switching strategies: a narrow growth-rate distribution suggests robustness, while a wide and asymmetric one indicates sensitivity to specific switching patterns.

Moreover, the distribution of spectral radii provides a refined measure of system robustness beyond the JSR--LSR gap, revealing whether small perturbations in switching induce gradual or abrupt changes in asymptotic growth. This probabilistic perspective also enriches our understanding of average convergence rates in subdivision schemes and the regularity of refinement equations, where typical rather than extremal behavior governs smoothness and stability.

In summary, the RSR provides a new tool for probabilistic stability assessment, switching strategy optimization, and average-case performance evaluation, bridging deterministic extremal analysis and stochastic system behavior.

We remark that an alternative approach, although conceptually different from ours, to defining a probabilistic joint spectral radius $\rho(\nu, P, \mathcal{F})$ for Markov random switching signals with transition matrix $P$ and invariant probability measure $\nu$ has been proposed in~\cite{CMS21}. In that work, the authors consider probability measures on the space of sequences generated by discrete-time shift-invariant Markov chains.

\subsection{Main results of the paper}

The asymptotic analysis of random matrix products has classical roots in the work of Furstenberg and Kesten~\cite{FurstenbergKesten1960}, who established almost sure convergence of the normalized logarithmic norm to the top Lyapunov exponent. For the spectral radius specifically, Aoun and Sert~\cite{AounSert2021} recently proved a law of large numbers in full generality. Moreover,  \cite{Aoun2021} established the central limit theorem around the limiting Lyapunov exponent for eigenvalues of random matrix products under a strong irreducibility condition in the general non-commutative cases.

We define the \emph{random spectral radius} (RSR) $\rho_n(\cF,\mathbb P_{\cF})$ as the normalized spectral radius of products of length $n$, namely the random variable
\[
\rho_n(\sigma) \coloneqq \rho \left(A_{i_n}\cdots A_{i_1}\right)^{1/n}, \qquad \sigma=(i_1,\dots,i_n)\in [m]^n,
\]
where the indices are sampled independently according to $\mathbb P_{\cF}$. We present theoretical results for families of commuting and upper- or lower-triangular matrices, together with numerical investigations for general complex matrices.

Our first result establishes a Law of Large Numbers (LLN) for $\rho_n$ when $\cF$ consists of commuting or upper- or lower-triangular matrices. As $n \to +\infty$, the spectral radius of typical products converges almost surely to a deterministic limit, namely
\[
\mathbb{P}\left(\lim_{n\to\infty}\rho_n\bigl(\mathcal F,\mathbb P_{\mathcal F}\bigr) = \rho_\infty\bigl(\mathcal F,\mathbb P_{\mathcal F}\bigr)\right)=1.
\]
Our second contribution is a Central Limit Theorem (CLT) describing fluctuations of the RSR around $\rho_{\infty} \bigl(\mathcal F,\mathbb P_{\mathcal F}\bigr)$. In the generic finite-family case, when a unique active coordinate determines the spectral radius and the associated asymptotic variance is positive, the normalized fluctuations converge in distribution to a standard normal:
\[
\frac{\rho_n(\mathcal F,\mathbb P_{\mathcal F})-\rho_\infty(\mathcal F,\mathbb P_{\mathcal F})}{n^{-1/2} \, \sigma_\infty(\mathcal F,\mathbb P_{\mathcal F})}
\Longrightarrow \mathcal N(0,1).
\]
In special commuting or triangular configurations, several diagonal coordinates may contribute equally to the deterministic limit. In that coalescence regime, the limiting fluctuation is no longer Gaussian in general: if $\boldsymbol G=(G_j)_{j\in J}$ is the centered Gaussian vector with covariance matrix $\Sigma_{\mathcal F,\mathbb P_{\mathcal F}}$ associated with the active coordinates, then
\[
\sqrt n\bigl(\rho_n(\mathcal F,\mathbb P_{\mathcal F})-\rho_\infty(\mathcal F,\mathbb P_{\mathcal F})\bigr)
\Longrightarrow
\rho_\infty(\mathcal F,\mathbb P_{\mathcal F})\,\max_{j\in J} G_j.
\]
To capture finite-size effects, we then analyze corrections beyond the CLT. Because finite matrix families lead to discrete laws for the underlying logarithmic variables, we derive exact multinomial formulas in the aligned unique-maximizer regime and finite-size asymptotic expansions for the mean and variance of $\rho_n$; in the coalescence regime we identify the universal leading terms and, under an additional weak Edgeworth hypothesis on the active law, the corresponding next-order corrections. These formulas show that the mean typically converges to $\rho_\infty$ with error of order $n^{-1}$ in the unique-maximizer case and of order $n^{-1/2}$ in the coalescence case, while the variance is of order $n^{-1}$ with finer corrections depending on the regime.

Definitions of $\rho_{\infty}$, $\sigma_{\infty}$, and $\Sigma_{\cF, \mathbb{P}_{\cF}}$ in terms of the initial data, along with detailed proofs, are provided in~\cref{sec:expected_spectral_radius}. 

Moreover, we extend our results to families $\cF_{\eps}$ obtained via small perturbations (of norm $\eps$) of diagonal matrices $D(\vartheta)$ as
\[
\tilde{A}(\vartheta) = D(\vartheta) + \eps \Delta_\vartheta,
\]
where $\Delta_{\vartheta}$ is normalized (in Frobenius norm), that is $\| \Delta _\vartheta \| = 1$. Under a nondegenerate unique-maximizer assumption, our perturbation analysis yields a uniform first-order logarithmic approximation for the perturbed random spectral radius. When the perturbed family is assumed to remain simultaneously triangularizable, this approximation upgrades to exact fixed-$\eps$ LLN and CLT statements and yields continuous dependence of the spectral parameters $\rho_{\infty}$ and $\sigma_{\infty}$ on $\eps$ for sufficiently small $|\eps|$. Although multi-directional fluctuation patterns can arise in the unperturbed diagonal setting, small generic perturbations separate the formerly coalescent branches at first order, so the perturbative approximation selects a preferred dominant branch.

Beyond the LLN and CLT, we establish a \textbf{large deviation principle} (LDP) for the RSR. The LDP provides exponential decay rates for the probability that $\rho_n$ deviates from $\rho_\infty$ by a fixed amount, directly answering how likely a randomly switched system is to exhibit worst-case (JSR) or best-case (LSR) growth. In the nondegenerate unique-maximizer case, the rate function is computed explicitly via a Legendre--Fenchel transform.

While prior work established the existence of limiting values and Gaussian fluctuations for random matrix products, our contributions provide several new results. First, we derive \emph{explicit formulas} for $\rho_\infty(\mathcal{F}, \mathbb{P}_{\mathcal{F}})$ and $\sigma_\infty^2(\mathcal{F}, \mathbb{P}_{\mathcal{F}})$ for diagonal, triangular, and perturbed families. Second, we characterize the \emph{coalescence phenomenon}: when multiple eigenvalues achieve the same maximal growth rate, fluctuations are governed by the maximum of a correlated Gaussian vector, yielding a non-Gaussian phase transition. Third, we develop finite-size asymptotic expansions of Edgeworth type, large deviation principles with explicit rate functions in the scalar regime, and a perturbation theory with uniform bounds.

\subsection{Outline of the article}

The paper is organized as follows. In~\cref{sec:expected_spectral_radius}, we introduce the random spectral radius and present the main theoretical results. Our analysis begins with the simplest illustrative case where each matrix in the diagonal family has a strictly dominant eigenvalue (\cref{subsec:ordered_diagonal}), in which we establish a LLN and a CLT with explicit formulas for the limit values and their fluctuations. We then extend these results to upper- and lower-triangular families, which include the generic diagonal case without the aligned dominance assumption (\cref{subsec:general_diagonal_matrices}). Finally, in~\cref{subsec:approx_errors} we provide explicit finite-size error bounds, exact finite-family expectation formulas, and Edgeworth-type corrections adapted to the discrete matrix setting; in~\cref{subsec:large_deviations} we establish a large deviation principle quantifying the exponential decay of tail probabilities and linking the RSR to the JSR and LSR; and in~\cref{subsec:rho_infty_prob} we analyze the dependence of $\rho_\infty$ on the probability distribution, establishing log-convexity and explicit sensitivity formulas. Following this, in~\cref{sec:perturbation_theory} a detailed perturbation analysis explores the behavior of small deviations from diagonal and triangular structures and shows the robustness of the one-dimensional fluctuation regime under small perturbations. Finally,~\cref{sec:numerics} and~\cref{app:numerics} present numerical experiments for generic dense matrix families.

\section{Random spectral radius} \label{sec:expected_spectral_radius}

In this section, we study the \emph{random spectral radius} (RSR) of a finite family of complex-valued $d \times d$ matrices, which we denote by
\[
\cF \coloneqq \{A_1,\ldots,A_m\} \subset \mathbb C^{d \times d}.
\]

\begin{assumption}
Whenever logarithms of eigenvalue moduli are used in this section, we assume that the
corresponding eigenvalues are nonzero. In particular, in the diagonal and triangular
settings considered below, we assume that
\[
|\lambda_j^{(i)}|>0, \qquad i=1,\dots,m,\quad j=1,\dots,d.
\]
\end{assumption}

We equip the family $\cF$ with a probability distribution $\mathbb{P}_{\cF}: \cF \rightarrow [0,1]$, assigning to each matrix $A_i$ the mass
\[
p_i \coloneqq \mathbb{P}_{\cF}(A_i), \qquad i = 1,\ldots,m.
\]
Denoting the set of all sequences of length $n$ with elements in $\{1,\ldots,m\}$ by $[m]^n \coloneqq \{1,\ldots,m\}^n,$ we define, for $\sigma=(i_1,\ldots,i_n)\in [m]^n$, the product map
\[
\Pi(\sigma) \coloneqq \prod_{j=1}^n A_{i_j} = A_{i_n} \cdots A_{i_1},
\]
and the \emph{random spectral radius of length $n$} as the random variable
\[
\rho_n : [m]^n\longrightarrow [0,\infty), \qquad \rho_n(\sigma) \coloneqq \rho\bigl(\Pi(\sigma)\bigr)^{1/n}.
\]
We equip $[m]^n$ with the product probability measure $\mathbb P(\sigma)=\prod_{j=1}^n p_{i_j}$. If one wishes to regard $\rho_n$ as a random variable on the set of products 
\[
\Sigma_n(\cF) \coloneqq \{\Pi(\sigma)\ : \ \sigma\in [m]^n\},
\]
then the corresponding law is the push-forward $\mathbb P_n(P) \coloneqq \sum_{\sigma: \, \Pi(\sigma)=P}\mathbb P(\sigma)$.

\begin{remark}
Matrix multiplication is generally non-commutative, so the order of factors affects both $\Pi(\sigma)$ and $\rho\bigl(\Pi(\sigma)\bigr)$. Exceptions include diagonal families (which commute) and triangular ones, for which $\rho\bigl(\Pi(\sigma)\bigr)$ is order-independent. More generally, since $\rho(XY)=\rho(YX)$ for any square matrices $X,Y$, the spectral radius is invariant under cyclic permutations.
\end{remark}

\begin{remark}
The definition of the RSR extends to infinite families: if matrices are sampled independently from a probability space via a measurable map, one can form products of arbitrary length and define the RSR as the normalized spectral radius of these products.
\end{remark}

To describe the Gaussian fluctuations that arise in the limit, we use the following standard notation for the cumulative distribution function (CDF) and probability density function (PDF) of the \textbf{standard normal distribution} $\mathcal{N}(0,1)$:
\begin{equation}\label{eq:cdf_pdf}
\Phi(x) = \frac{1}{\sqrt{2 \pi}} \int_{-\infty}^{x} \mathrm{e}^{-\frac{u^2}{2}}\,\mathrm{d}u, \qquad
\phi(x) = \frac{1}{\sqrt{2\pi}}\, \mathrm{e}^{-\frac{x^2}{2}}.
\end{equation}
Here, $\Phi(x)$ gives the probability that a standard normal random variable is less than or equal to $x$, and $\phi(x)$ is the corresponding density function. 

We start with diagonal families. This also covers commuting families that are simultaneously diagonalizable, after a change of basis. The second case involves families of upper or lower triangular matrices, which, while not necessarily commuting, share a common triangular structure that enables explicit analysis. In both settings, this structure allows for the derivation of explicit probabilistic limit laws for the RSR.

\subsection{RSR of commuting families with aligned dominance} \label{subsec:ordered_diagonal}

We now specialize to the case where the family consists of diagonal matrices 
\[
\mathcal{S}= \bigl\{ D_1,\ldots,D_m \bigr\} \subset \mathbb C^{d \times d}, \qquad D_i = \operatorname{diag} \bigl(\lambda^{(i)}_1, \lambda^{(i)}_2,\ldots,\lambda^{(i)}_d\bigr).
\]
Throughout this section, we assume that the family $\cS$ satisfies the \emph{aligned dominance assumption}: there exists  $r_\star\in\{1,\dots,d\}$ such that for each matrix $D_i$, the corresponding eigenvalue $\lambda^{(i)}_{r_\star}$ is strictly dominant in modulus,
\begin{equation} \label{def:aligned dom convergence}
\bigl|\lambda^{(i)}_{r_\star}\bigr| > \max_{r\neq r_\star} \, \bigl|\lambda^{(i)}_{r}\bigr|.
\end{equation}
By simultaneously permuting coordinates (i.e., conjugating all $D_i$ by the same permutation matrix), we may assume without loss of generality that $r_\star=1$.

Since diagonal matrices commute, the spectral radius of any product is simply the maximum modulus among the diagonal entries of that product. By~\eqref{def:aligned dom convergence}, this maximum is always achieved by the first diagonal entry. Consequently,
\[
\rho \bigl( \Pi(\sigma) \bigr) = \max_{1\le r\le d}\ \prod_{k=1}^n \bigl|\lambda^{(i_k)}_r\bigr| =\prod_{k=1}^n \bigl|\lambda^{(i_k)}_1\bigr|
\]
for any $\sigma = (i_1,\ldots,i_n) \in [m]^n$. Equipping $\mathcal{S}$ with a probability distribution $\mathbb{P}_{\mathcal{S}}$ with weights $p_i:=\mathbb{P}_{\mathcal{S}}(D_i)$, we define an i.i.d. sequence of random variables $\{Y_k\}_{k \geq 1}$, where each $Y_k$ takes the absolute value of a dominant eigenvalue according to
\[
\mathbb{P}_Y\Bigl(Y_k=|\lambda^{(i)}_1|\Bigr)=p_i, \qquad i = 1, \dots, m.
\]
In other words, the probability $\mathbb{P}_{\mathcal{S}}$ on matrices induces a corresponding probability distribution $\mathbb{P}_Y$ on the dominant eigenvalues, transferring the sampling mechanism from the matrix level to the scalar level. Under this framework, the (length-$n$) RSR reduces to the geometric mean of $n$ i.i.d. samples:
\[
\rho_n \bigl(\mathcal{S},\mathbb{P}_{\mathcal{S}} \bigr) = \left(\prod_{k=1}^n Y_k\right)^{1/n}.
\]
This i.i.d. structure is crucial for establishing asymptotic results, starting with a LLN that characterizes the spectral radius of a typical matrix product.

\begin{theorem}[LLN]
The random spectral radius of the family $\cS$ converges almost surely to a deterministic limit. Specifically, it satisfies
\begin{equation} \label{def:rho_infty}
\rho_n(\mathcal S,\mathbb P_{\mathcal S}) \xrightarrow{\ \text{a.s.}\ } \rho_\infty\bigl(\mathcal S,\mathbb P_{\mathcal S}\bigr) \coloneqq \prod_{i=1}^m \bigl|\lambda^{(i)}_1\bigr|^{p_i} \qquad \text{as } \, n\to\infty.
\end{equation}
Equivalently, the limit holds with probability one:
\[
\mathbb{P}\left(\lim_{n\to\infty}\rho_n\bigl(\mathcal S,\mathbb P_{\mathcal S}\bigr) = \rho_\infty\bigl(\mathcal S,\mathbb P_{\mathcal S}\bigr)\right)=1.
\]
\end{theorem}

\begin{proof}
Define $Z_k \coloneqq \log Y_k$ for $k\geq 1$, the i.i.d. random variables with finite moments. The random spectral radius can be expressed as the exponential:
\begin{equation} \label{eq:rho=exp(sum)}
\rho_n(\mathcal S,\mathbb P_{\mathcal S})=\exp\left(\frac1n\sum_{k=1}^n Z_k\right),
\end{equation}
where $Z_k$ takes the value $\log|\lambda^{(i)}_1|$ with probability $p_i$, i.e. $\mathbb{P}_Z \bigl(Z_k=\log|\lambda^{(i)}_1|\bigr) = p_i$. By the strong law of large numbers,
\[
\frac1n\sum_{k=1}^n Z_k \xrightarrow{\ \text{a.s.}\ } \mathbb E[Z_1] = \sum_{i=1}^m p_i \log\bigl|\lambda^{(i)}_1 \bigr| =\log\left(\prod_{i=1}^m \bigl|\lambda^{(i)}_1\bigr|^{p_i}\right).
\]
Therefore, by continuity of the exponential, the random spectral radius $\rho_n(\mathcal S, \mathbb{P}_{\mathcal S})$ converges almost surely to the exponential of the mean value of $Z_1$, that is,
\[
\rho_n(\mathcal S,\mathbb P_{\mathcal S})  \xrightarrow{\ \text{a.s.}\ }
\exp\Bigl(\mathbb{E}[Z_1]\Bigr) = \prod_{i=1}^m \bigl|\lambda_1^{(i)}\bigr|^{p_i}.
\]
\end{proof}

\begin{remark} 
For the matrix family $\cS$, under the aligned dominance assumption~\eqref{def:aligned dom convergence}, the moments of the RSR have compact exact formulas. For every $q>0$,
\[
\mathbb E\bigl[\rho_n(\mathcal S,\mathbb P_{\mathcal S})^q\bigr]
=
\left(\sum_{i=1}^m p_i |\lambda_1^{(i)}|^{q/n}\right)^n.
\]
In particular,
\[
\mathbb{E}\bigl[\rho_n(\mathcal{S},\mathbb{P}_{\mathcal{S}})\bigr]
=
\left(\sum_{i=1}^m p_i |\lambda_1^{(i)}|^{1/n}\right)^n,
\]
and
\[
\operatorname{Var}\bigl(\rho_n(\mathcal{S},\mathbb{P}_{\mathcal{S}})\bigr)
=
\left(\sum_{i=1}^m p_i |\lambda_1^{(i)}|^{2/n}\right)^n
-
\left(\sum_{i=1}^m p_i |\lambda_1^{(i)}|^{1/n}\right)^{2n}.
\]
As $n \to \infty$,
\[
\mathbb E[\rho_n(\mathcal{S},\mathbb{P}_{\mathcal{S}})] = \rho_\infty(\mathcal S, \mathbb P_{\mathcal S}) \bigl(1+O(n^{-1})\bigr),
\]
and
\[
\operatorname{Var}\left(\rho_n(\mathcal{S},\mathbb{P}_{\mathcal{S}})\right)
=  \frac{\sigma_{\infty}^2(\mathcal S, \mathbb P_{\mathcal S})}{n}+O(n^{-2}),
\]
where
\begin{equation} \label{def:sigma_infty}
\sigma_{\infty}^2 (\mathcal S, \mathbb P_{\mathcal S}) := \rho_\infty^2(\mathcal S, \mathbb P_{\mathcal S}) \left(\sum_{i=1}^m p_i\bigl(\log \bigl|\lambda^{(i)}_1 \bigr|\bigr)^2 - \left(\sum_{i=1}^m p_i\log \bigl|\lambda^{(i)}_1\bigr|\right)^2 \right).
\end{equation}
\end{remark}

Having characterized the limiting behavior of $\rho_n$ via the law of large numbers, we now turn to the fluctuations around this limit.

\begin{theorem}[CLT] \label{thm:CLT}
Let $Z_1=\log Y_1$, and set
\[
\mu_Z\coloneqq \mathbb E[Z_1]=\log\rho_\infty(\mathcal S,\mathbb P_{\mathcal S}),
\qquad
\sigma_Z^2\coloneqq \operatorname{Var}(Z_1).
\]
Then
\[
\sqrt{n}\,\bigl(\rho_n(\mathcal S,\mathbb P_{\mathcal S})-\rho_\infty(\mathcal S,\mathbb P_{\mathcal S})\bigr)  \xRightarrow{\ d\ } \mathcal N\left(0, \sigma_\infty^2(\mathcal S,\mathbb P_{\mathcal S})\right),
\]
where $\sigma_\infty^2(\mathcal S,\mathbb P_{\mathcal S})=\rho_\infty^2(\mathcal S,\mathbb P_{\mathcal S})\sigma_Z^2$. If $\sigma_Z^2>0$, the standardized fluctuations satisfy
\begin{equation}\label{eq:normal_diag}
\lim_{n\to\infty}\mathbb P\left( \frac{\rho_n(\mathcal S,\mathbb P_{\mathcal S})-\rho_\infty(\mathcal S,\mathbb P_{\mathcal S})}{n^{-1/2} \, \sigma_\infty(\mathcal S,\mathbb P_{\mathcal S})} < x\right)=\Phi(x),\qquad x\in\mathbb R.
\end{equation}
If $\sigma_Z^2=0$, then $\rho_n=\rho_\infty$ almost surely for every $n$ and the limiting fluctuation is degenerate.
\end{theorem}

\begin{proof}
When $\sigma_Z^2=0$, the random variable $Z_1$ is almost surely constant, and the assertion is immediate. Assume $\sigma_Z^2>0$. By the classical CLT,
\[
\sqrt n\left(\frac1n\sum_{k=1}^n Z_k-\mu_Z\right)\xRightarrow{\ d\ }\mathcal N(0,\sigma_Z^2).
\]
Since
\[
\rho_n=\exp\left(\frac1n\sum_{k=1}^n Z_k\right),
\]
the delta method applied to the exponential map at $\mu_Z$ gives
\[
\sqrt n(\rho_n-\rho_\infty)\xRightarrow{\ d\ }\rho_\infty\,\mathcal N(0,\sigma_Z^2)=\mathcal N(0,\rho_\infty^2\sigma_Z^2).
\]
The standardized convergence follows because $\sigma_\infty=\rho_\infty\sigma_Z>0$.
\end{proof}

The CLT provides an asymptotic approximation, but it is natural to ask the rate of convergence. The following proposition gives a non-asymptotic bound.

\begin{proposition}[Berry-Esseen bound]\label{prop:berry_esseen}
Under the hypotheses of~\cref{thm:CLT}, assume $\sigma_Z^2>0$, and let $\beta_3 := \mathbb{E}\bigl[|Z_1 - \mu_Z|^3\bigr]$. Define
\[
\overline Z_n\coloneqq \frac{\frac1n\sum_{k=1}^n Z_k-\mu_Z}{n^{-1/2}\sigma_Z}.
\]
Then, for every $n \geq 1$,
\begin{equation}\label{eq:berry_esseen_Z}
\sup_{x \in \mathbb{R}} \left| \mathbb{P}\bigl(\overline{Z}_n \leq x\bigr) - \Phi(x) \right| \leq \frac{C_{\mathrm{BE}} \, \beta_3}{\sigma_Z^3 \, \sqrt{n}},
\end{equation}
where $C_{\mathrm{BE}} \leq 0.4748$ is the Berry--Esseen constant. Consequently, for every $A>0$ there exists a constant $C_A>0$ such that
\begin{equation}\label{eq:berry_esseen_rho}
\sup_{|x|\le A} \left| \mathbb P \left( \frac{\rho_n(\mathcal S,\mathbb P_{\mathcal S})-\rho_\infty(\mathcal S,\mathbb P_{\mathcal S})}{n^{-1/2}\sigma_\infty(\mathcal S,\mathbb P_{\mathcal S})} \le x \right) -\Phi(x) \right| \le \frac{C_A}{\sqrt n}.
\end{equation}
\end{proposition}

\begin{proof}
The bound~\eqref{eq:berry_esseen_Z} follows from the Berry--Esseen theorem (see, e.g., \cite[Theorem 3.4.17]{durrett2019probability}) applied to the i.i.d. sequence $\{(Z_k-\mu_Z)/\sigma_Z\}_{k\ge1}$; the constant $C_{\mathrm{BE}} \le 0.4748$ is due to~\cite{shevtsova2011improvement}. For the transfer to $\rho_n$, use $\sigma_\infty=\rho_\infty\sigma_Z$ and write, uniformly for $|x|\le A$,
\[
\mathbb P\left(\frac{\rho_n-\rho_\infty}{n^{-1/2}\sigma_\infty}\le x\right)
=
\mathbb P\left(\overline Z_n\le a_{n,x}\right),
\qquad
 a_{n,x}\coloneqq \frac{\sqrt n}{\sigma_Z}\log\left(1+\frac{x\sigma_Z}{\sqrt n}\right).
\]
The Taylor expansion of the logarithm gives $a_{n,x}=x+O_A(n^{-1/2})$. Combining~\eqref{eq:berry_esseen_Z} with the Lipschitz continuity of $\Phi$ gives~\eqref{eq:berry_esseen_rho}.
\end{proof}

\subsection{RSR of triangular matrices} \label{subsec:general_diagonal_matrices}

We now consider a more general setting where the aligned dominance assumption does \textbf{not} hold. Let $\mathcal{T}$ be a finite collection of upper(lower)-triangular matrices with
\[
\mathcal{T}=\{T_1,\ldots, T_m\},\qquad \left(T_i\right)_{jj} = \lambda^{(i)}_j, \quad j = 1, \dots, d.
\]                                
The distribution $\mathbb{P}_{\mathcal{T}}$ on $\mathcal{T}$ is defined by weights $p_i:=\mathbb{P}_{\mathcal{T}}(T_i)$. We emphasize that the family $\mathcal{T}$ consists exclusively of either upper or lower triangular matrices, \textbf{not both}. 

The spectrum of the product of upper (or lower) triangular matrices is independent of the multiplication order, hence we find the spectral radius in the following form:
\[
\rho\bigl(\Pi(\sigma)\bigr)=\rho\!\left(\prod_{k=1}^n T_{i_k}\right)=\max_{1\le j\le d}\ \prod_{k=1}^n \bigl|\lambda^{(i_k)}_j\bigr|.
\]
The key difference from the aligned dominance case is the presence of the maximum function, which chooses a value among \emph{all} diagonal positions. To capture this structure, we introduce $d$-dimensional i.i.d. random vector variables
\[
\boldsymbol{Y}_k = \bigl(Y_k^{(1)}, \dots, Y_k^{(d)}\bigr), \qquad k \geq 1
\]
representing the absolute eigenvalues of each sampled matrix. In particular, each random vector $\boldsymbol{Y}_k$ takes the value $(|\lambda^{(i)}_1|,\ldots, |\lambda_d^{(i)}|)$ with probability
\[
\mathbb{P}_{\boldsymbol{Y}} \Bigl(\boldsymbol{Y}= \left( \bigl|\lambda^{(i)}_1\bigr|,\dots, \bigl|\lambda_d^{(i)}\bigr| \right) \Bigr) = p_i,\qquad i=1,\ldots,m.
\]
The random spectral radius is then the maximum component of the componentwise geometric mean of these vectors. Writing $\odot$ for componentwise multiplication,
\begin{equation}\label{eq:rho_n_max_prodY}
\rho_n(\mathcal{T},\mathbb P_{\mathcal{T}}) = \max_{1\le j\le d} \left(\boldsymbol{Y}_1\odot\cdots\odot \boldsymbol{Y}_n\right)_j^{1/n}
\end{equation}
where the subscript $j$ denotes the $j$-th component. Unlike the aligned dominance case, where we tracked a single scalar sequence, here the RSR value is formed from all the $d$ components of a resulting random vector.

\begin{theorem}[LLN]\label{thm:LLN gen diag}
The RSR of the family $\mathcal{T}$ converges almost surely, i.e.
\[
\rho_n(\mathcal{T},\mathbb P_{\mathcal{T}}) \xrightarrow{\ \text{a.s.}\ } \rho_\infty\bigl(\mathcal{T},\mathbb P_{\mathcal{T}}\bigr)  \qquad \text{as } \, n\to\infty,
\]
where 
\begin{equation}\label{def:rho_infty2}
\rho_\infty(\mathcal{T},\mathbb P_{\mathcal{T}}) \coloneqq \max_{1\le j\le d} \rho_\infty^{(j)}(\mathcal{T},\mathbb P_{\mathcal{T}}), \qquad \rho_\infty^{(j)}(\mathcal{T},\mathbb P_{\mathcal{T}}) \coloneqq \prod_{i=1}^m \bigl|\lambda^{(i)}_j\bigr|^{p_i}.
\end{equation}
Equivalently, the corresponding limit holds with probability one:
\[
\mathbb P \left(\lim_{n\to\infty}\rho_n(\mathcal{T},\mathbb P_{\mathcal{T}})=\rho_\infty(\mathcal{T},\mathbb P_{\mathcal{T}})\right)=1.
\]
\end{theorem}

\begin{proof}
Define the log-transformed vectors by
\begin{equation}\label{def:Zvector}
\boldsymbol{Z}_k \coloneqq (Z_k^{(1)}, \dots, Z_k^{(d)}), \qquad Z_k^{(j)} = \log Y_k^{(j)}.
\end{equation}
The random vectors $\{\boldsymbol{Z}_k\}_{k\geq 1}$ are i.i.d. with finite moments, hence by the strong law of large numbers,
\[
\frac1n\sum_{k=1}^n \boldsymbol{Z}_k \xrightarrow{\, \text{a.s.}\, } \mathbb E[\boldsymbol{Z}_1]
= \left( \sum_{i=1}^m p_i \log|\lambda_1^{(i)}|, \ldots, \sum_{i=1}^m p_i \log|\lambda_d^{(i)}| \right).
\]
Since the exponential function and the maximum of finitely many quantities are continuous operations, passing to the limit from~\eqref{eq:rho_n_max_prodY} we get
\begin{equation}\label{eq:rho_n_via_maxZ}
\rho_n(\mathcal{T},\mathbb P_{\mathcal{T}}) = \max_{1\le j\le d} \exp\left(\frac1n\sum_{k=1}^n Z_k^{(j)} \right) \xrightarrow{\, \text{a.s.} \,} \max_{1\le j\le d}\ \exp \bigl(\mathbb E\bigl[Z_1^{(j)}\bigr]\bigr).
\end{equation}
\end{proof}

Before stating the central limit theorem for the general case, we observe that the maximum in the law of large numbers~\eqref{eq:rho_n_via_maxZ} may be attained at several indices. To keep track of this information, we introduce the set of maximizing indices:
\[
J \coloneqq \arg\max_{1\le j\le d} \rho_\infty^{(j)}(\mathcal{T},\mathbb{P}_{\mathcal{T}}), \qquad s \coloneqq |J|.
\]
If $J=\{j^\star\}$, then the maximum is uniquely attained at index $j^{\star}$, and the asymptotic fluctuations for $n$ large enough are governed solely by the $j^\star$-coordinate, recovering the result of~\cref{thm:CLT}. However, if multiple coordinates share the maximal asymptotic growth rate, they jointly influence the fluctuations. With the following theorem, we show how fluctuations for typical matrix products change in this case.

\begin{theorem}[CLT] \label{thm:clt_general}
Let $\rho_\infty(\mathcal{T}, \mathbb P_{\mathcal{T}})$ be defined by~\eqref{def:rho_infty2}. Then
\begin{equation}\label{CLT_general}
\sqrt n\Bigl(\rho_n(\mathcal{T},\mathbb P_{\mathcal{T}})-\rho_\infty(\mathcal{T},\mathbb P_{\mathcal{T}})\Bigr) \ \xRightarrow{\ d\ }\  \rho_\infty(\mathcal{T},\mathbb P_{\mathcal{T}})\cdot \max_{j\in J} G_j,
\end{equation}
where $\boldsymbol{G} = (G_j)_{j \in J}$ is a zero-mean Gaussian vector with covariance matrix
\begin{equation} \label{def:Sigma}
\Sigma_{\mathcal{T}, \mathbb{P}_{\mathcal{T}}} = \{\Sigma(\ell,h)\}_{\ell,h \in J}, \qquad 
 \Sigma(\ell,h) \coloneqq \sum_{i=1}^m p_i \log \bigl|\lambda_\ell^{(i)}\bigr| \log \bigl|\lambda_h^{(i)}\bigr| - \mu_\ell\mu_h,
\end{equation}
and
\begin{equation} \label{def:mu_j}
\mu_\ell \coloneqq \log \rho_\infty^{(\ell)}(\mathcal{T},\mathbb P_{\mathcal{T}})=\sum_{i=1}^m p_i\log\bigl|\lambda_\ell^{(i)}\bigr|, \qquad \mu^\star \coloneqq \max_{1\le j\le d} \mu_j=\log\rho_\infty(\mathcal{T},\mathbb P_{\mathcal{T}}).
\end{equation}
Equivalently, for every continuity point $x$ of the limiting distribution,
\[
\lim_{n\to\infty}\mathbb P\left(\sqrt n\bigl(\rho_n-\rho_\infty\bigr)<x\right) = \mathbb P \left(\max_{j\in J}G_j<\frac{x}{\rho_\infty}\right) \eqqcolon M_s \left(\frac{x}{\rho_\infty}; \; \Sigma_{\mathcal{T}, \mathbb{P}_{\mathcal{T}}}\right).
\]
If $\Sigma_{\mathcal{T}, \mathbb{P}_{\mathcal{T}}}$ is positive definite, then the limiting distribution has no atoms and the equality holds for every $x\in\mathbb R$.
\end{theorem}

\begin{proof}
Recall the log-transformed random vectors $\boldsymbol{Z}_k = \log \boldsymbol{Y}_k$ from~\cref{thm:LLN gen diag}. Their law is
\[ 
\mathbb{P}_{ \boldsymbol{Z}} \left( \boldsymbol{Z} = \bigl(\log |\lambda_1^{(i)}|, \ldots, \log |\lambda_d^{(i)} |\bigr)\right) = p_i, \qquad i=1,\ldots,m.
\]
Let $\boldsymbol\mu=(\mu_1,\ldots,\mu_d)$. By the multivariate CLT,
\[
\boldsymbol W_n\coloneqq \sqrt{n} \left( \frac{1}{n}\sum_{k=1}^n \boldsymbol{Z}_k-\boldsymbol{\mu}\right)
\xRightarrow{\ d\ }
\boldsymbol W\sim\mathcal{N}(\boldsymbol{0}, \Gamma),
\]
where $\Gamma$ is the full covariance matrix of $\boldsymbol Z_1$ and the restriction of $\Gamma$ to the coordinates in $J$ is~\eqref{def:Sigma}. Consider the map $F:\mathbb R^d\to(0,\infty)$ given by
\[
F(\boldsymbol z)=\exp\left(\max_{1\le j\le d}z_j\right).
\]
Then $\rho_n=F(\boldsymbol\mu+n^{-1/2}\boldsymbol W_n)$ and $\rho_\infty=F(\boldsymbol\mu)$. The map $F$ is directionally differentiable at $\boldsymbol\mu$, with derivative
\[
F_{\boldsymbol\mu}'(\boldsymbol h)=\rho_\infty\max_{j\in J}h_j.
\]
Indeed, for bounded $\boldsymbol h$ and sufficiently small $t>0$, the inactive coordinates $j\notin J$ remain separated from the maximum by the positive gaps $\mu^\star-\mu_j$, so
\[
\max_j(\mu_j+t h_j)=\mu^\star+t\max_{j\in J}h_j.
\]
Therefore
\[
\frac{F(\boldsymbol\mu+t\boldsymbol h)-F(\boldsymbol\mu)}{t}
=\rho_\infty\frac{\exp(t\max_{j\in J}h_j)-1}{t}
\longrightarrow \rho_\infty\max_{j\in J}h_j.
\]
Applying this expansion with $t=n^{-1/2}$ and $\boldsymbol h=\boldsymbol W_n=O_p(1)$ gives
\[
\sqrt n(\rho_n-\rho_\infty)
=
\rho_\infty\max_{j\in J}W_{n,j}+o_p(1).
\]
By the continuous mapping theorem and Slutsky's lemma \cite[Thm.~2.3, Lem.~2.8]{vandervaart1998}, together with the delta method for directionally differentiable functions \cite{fang2019delta}, we obtain~\eqref{CLT_general}.
\end{proof}

\begin{corollary} \label{Cor:genericCLT}
Assume $J = \{j^\star\}$ and let
\[
\sigma_{\infty}^{(j)}(\mathcal{T}, \mathbb{P}_{\mathcal{T}})
:= \rho_\infty^{(j)} (\mathcal{T}, \mathbb{P}_{\mathcal{T}}) \left(\sum_{i=1}^m p_i (\log|\lambda_{j}^{(i)}|)^2 - \left(\sum_{i=1}^m p_i\log|\lambda_{j}^{(i)}|\right)^2\right)^{1/2}.
\]
If $\sigma_{\infty}^{(j^\star)}(\mathcal T,\mathbb P_{\mathcal T})>0$, then
\[
\lim_{n \to \infty}\mathbb{P}\biggl(\frac{\rho_n(\mathcal{T}, \mathbb{P}_{\mathcal{T}}) - \rho_{\infty}(\mathcal{T}, \mathbb{P}_{\mathcal{T}})}{n^{-1/2}\sigma_{\infty}^{(j^\star)}(\mathcal{T}, \mathbb{P}_{\mathcal{T}})}  < x \biggr) = \Phi(x),\qquad x\in\mathbb R.
\]
If $\sigma_{\infty}^{(j^\star)}(\mathcal T,\mathbb P_{\mathcal T})=0$, then the limiting fluctuation in~\cref{thm:clt_general} is degenerate.
\end{corollary}

\begin{remark}
The case $|J|>1$ in~\cref{thm:clt_general} corresponds to families where multiple diagonal positions share the same maximal asymptotic growth rate, i.e., $\rho_\infty^{(j_1)} = \rho_\infty^{(j_2)}$ for some indices $j_1 \neq j_2 \in J$. Since this requires the equality 
\[
\sum_{i=1}^m p_i \log|\lambda_{j_1}^{(i)}| = \sum_{i=1}^m p_i \log|\lambda_{j_2}^{(i)}|,
\]
which defines a codimension-one submanifold in the parameter space of eigenvalues and probabilities, the condition $|J|>1$ is \textbf{non-generic}: for a ``random'' choice of matrices and probabilities, one expects $|J|=1$ and hence a Gaussian CLT. The multivariate-maximum limit law is thus a boundary phenomenon governing transitions between different dominant eigenvalue regimes. As we demonstrate in~\cref{sec:perturbation_theory}, infinitesimal perturbations of diagonal matrices generically separate the formerly coalescent branches at first order. Under the additional assumption that the perturbed family remains simultaneously triangularizable, this separation yields an exact fixed-$\eps$ single-maximizer LLN/CLT regime for sufficiently small $|\eps|$, reinforcing the view that the $|J|>1$ case is structurally unstable.
\end{remark}

\begin{figure}[ht!]
  \centering
  \includegraphics[width=.5\textwidth]{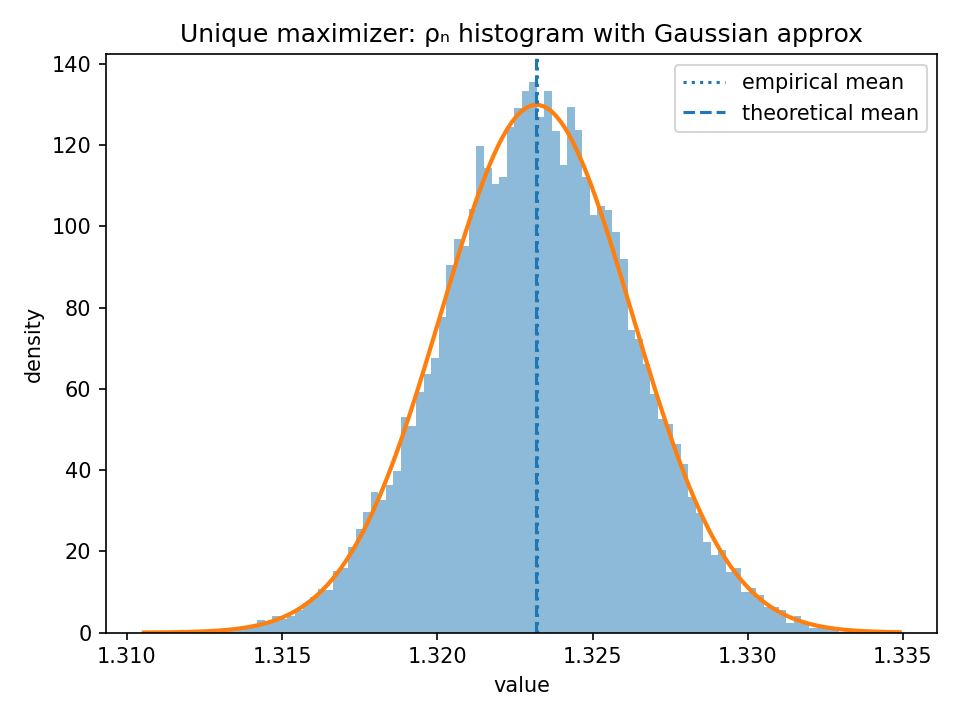}
  \caption{Histogram of $\rho_{800}(\mathcal{T}, \mathbb P_{\mathcal T})$ compared with the Gaussian pdf. Here, the family consists of two $3 \times 3$ real-valued matrices with $J = \{1\}, $ $\rho_\infty \approx 1.323194$ and $\sigma_\infty \approx 0.086861$.}
  \label{fig:unique}
\end{figure}

\begin{remark}
When the covariance matrix $\Sigma_{\mathcal{T}, \mathbb{P}_{\mathcal{T}}}$ is diagonal and the active variances are positive, the components of the Gaussian vector $\boldsymbol{G} = (G_j)_{j \in J}$ are independent. In this case, writing $J=\{j_1,\ldots,j_s\}$,
\begin{equation*}
\lim_{n\to\infty}\mathbb{P}\left(\sqrt{n}(\rho_n-\rho_\infty)<x\right) =\prod_{k=1}^s \Phi \left(\frac{x}{\sigma_{\infty}^{(j_k)} (\mathcal{T}, \mathbb{P}_{\mathcal{T}})}\right)
\end{equation*}
at every continuity point of the limiting distribution.
\end{remark}

\begin{remark}
When $J=\{j_1,\ldots,j_s\}$ with $s>1$, the limiting distribution in~\eqref{CLT_general} is generally non-Gaussian. The limiting random variable
\[
\xi^{(s)} = \rho_\infty \cdot \max_{j\in J} G_j
\]
has nonnegative mean, because $\max_j G_j\ge s^{-1}\sum_jG_j$ and the right-hand side has expectation zero. The mean is strictly positive unless the active Gaussian components coincide almost surely. In that exceptional case, $\max_{j\in J}G_j$ reduces to a single Gaussian random variable and the limit is Gaussian. At continuity points of the distribution function,
\[
\mathbb{P}\bigl(\xi^{(s)}\le x\bigr)  = M_s \left(\frac{x}{\rho_\infty};\,\Sigma_{\mathcal{T}, \mathbb{P}_{\mathcal{T}}}\right),\qquad x\in\mathbb{R}.
\]
When the law is continuous, the corresponding density is
\[
\frac{\mathrm{d}}{\mathrm{d}x} \mathbb{P}\bigl(\xi^{(s)}\le x\bigr) = \frac{1}{\rho_\infty}\frac{\mathrm{d}}{\mathrm{d}t} M_s \left(t;\,\Sigma_{\mathcal{T}, \mathbb{P}_{\mathcal{T}}}\right)\bigg|_{t=x/\rho_\infty}.
\]
The expected value admits the integral representation
\[
\mathbb{E}\bigl[\xi^{(s)}\bigr] = \rho_\infty \int_0^\infty \Bigl(1- M_s(x;\, \Sigma_{\mathcal{T}, \mathbb{P}_{\mathcal{T}}}) - M_s(-x; \, \Sigma_{\mathcal{T}, \mathbb{P}_{\mathcal{T}}})\Bigr) \,\mathrm{d}x,
\]
which follows from the identity $\mathbb{E}[X] = \int_0^\infty \bigl(1-F_X(x)-F_X(-x)\bigr)\,\mathrm{d}x$ for any random variable $X$ with $\mathbb{E}|X| < \infty$.
\end{remark}

In some simple cases, closed-form expressions for both $\mathbb{E}[\xi^{(s)}]$ and $\operatorname{Var}(\xi^{(s)})$ can be derived. We provide these formulas below, with numerical illustrations. All symbolic computations were performed using \texttt{Mathematica}.

\begin{example}
Assume $J=\{j_1,j_2\}$ and let $s_\ell^2:=\Sigma_{\mathcal{T}, \mathbb P_{\mathcal{T}}}(j_\ell,j_\ell)$ denote the variance of $G_{j_\ell}$ for $\ell=1,2$. The correlation coefficient between $G_{j_1}$ and $G_{j_2}$ is
\[
\rho := \frac{\Sigma_{\mathcal{T}, \mathbb P_{\mathcal{T}}}(j_1,j_2)}{s_1 s_2}\in[-1,1].
\]
Then
\[
\mathbb E \left[\xi^{(2)}\right] = \rho_\infty\, \sqrt{\frac{{s_1^{2}+s_2^{2}-2\rho\,s_1 s_2}}{2\pi}}.
\]
If the components of $\boldsymbol{G}$ are independent ($\rho=0$), this simplifies to
\[
\mathbb E \left[\xi^{(2)}\right] = \rho_\infty\,\frac{\sqrt{s_1^{2}+s_2^{2}}}{\sqrt{2\pi}}.
\]
When $s_1=s_2=:s$, the variance has the closed formula
\[
\operatorname{Var}\left(\xi^{(2)}\right) = \rho_\infty^{2}\,s^{2}\left(1-\frac{1-\rho}{\pi}\right).
\]
For unequal variances and $\rho = 0$,
\[
\operatorname{Var}\left(\xi^{(2)}\right) = \rho_\infty^{2}\, \bigl(s_1^{2}+s_2^{2}\bigr)\left(\frac12-\frac{1}{2\pi}\right).
\]
\end{example}

\begin{example}
Assume $J=\{j_1,j_2,j_3\}$ and let the components of $\boldsymbol{G}$ have equal variances $s_1=s_2=s_3=:s$ and equal pairwise correlation $\rho\in[-1/2,1]$. Then
\[
\mathbb E \left[\xi^{(3)}\right] = \rho_\infty\,s\sqrt{1-\rho}\, \frac{3}{2\sqrt{\pi}},\qquad \operatorname{Var}\left(\xi^{(3)}\right) = \rho_\infty^{2}\,s^{2}\left(\rho+(1-\rho)\frac{2\sqrt{3}+4\pi-9}{4\pi}\right).
\]
For $0\le\rho\le1$, these formulas follow from the representation $G_j = s\bigl(\sqrt{\rho}\,Z+\sqrt{1-\rho}\,\eps_j\bigr)$ with $Z,\eps_1,\eps_2,\eps_3$ i.i.d. standard normal. For the admissible negative range $-1/2\le\rho<0$, the same formulas follow by decomposing the equicorrelated Gaussian vector into its sample mean and its orthogonal centered component. A simple computation gives
\[
\mathbb E \left[\max_{1\le k\le 3}\eps_k\right]=\frac{3}{2\sqrt{\pi}},\qquad \operatorname{Var} \left(\max_{1\le k\le 3}\eps_k\right)=\frac{2\sqrt{3}+4\pi-9}{4\pi}.
\]
If the components are uncorrelated ($\rho=0$) with potentially unequal variances, i.e. $\Sigma_{\mathcal{T}, \mathbb{P}_{\mathcal{T}}} = \operatorname{diag}\bigl(s_1^{2},s_2^{2},s_3^{2}\bigr)$, then
\[
\mathbb E \left[\xi^{(3)}\right] =\frac{\rho_\infty}{2\sqrt{2\pi}} \sum_{1\le i<h\le 3} \sqrt{s_i^{2}+s_h^{2}}.
\]
A closed formula for the variance, even when $\rho = 0$, is unavailable assuming that $s_1\neq s_2 \neq s_3$. See \Cref{fig:s3}, which illustrates that the presence of multiple maximizing indices results in a net positive average fluctuation, leading to a skewed limiting distribution.

\begin{figure}[ht!]
  \centering
  \includegraphics[width=.48\textwidth]{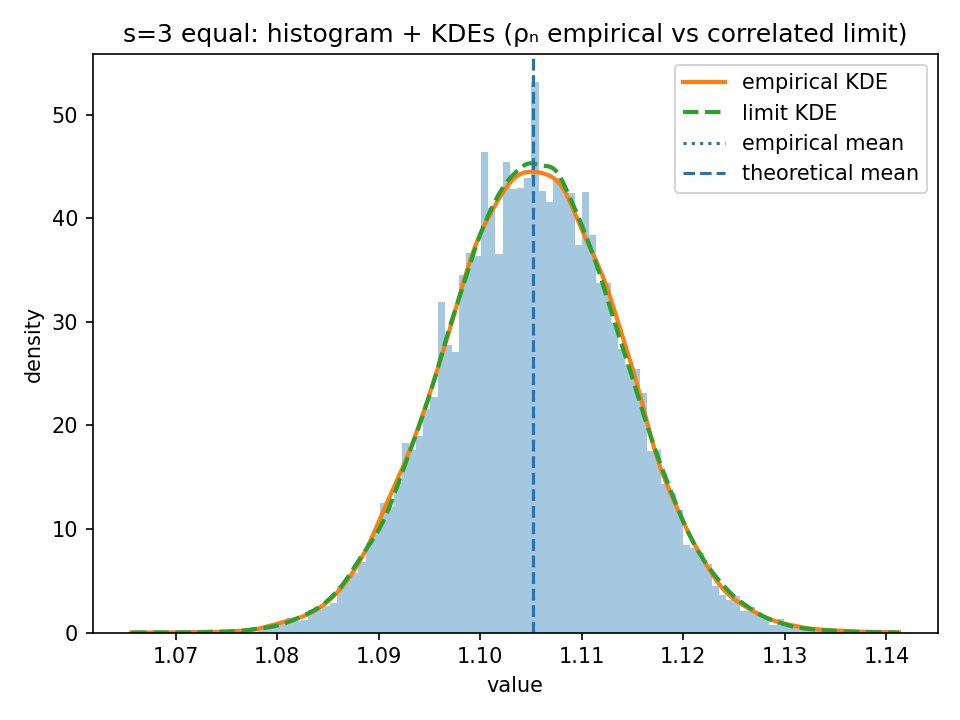}
  \includegraphics[width=.48\textwidth]{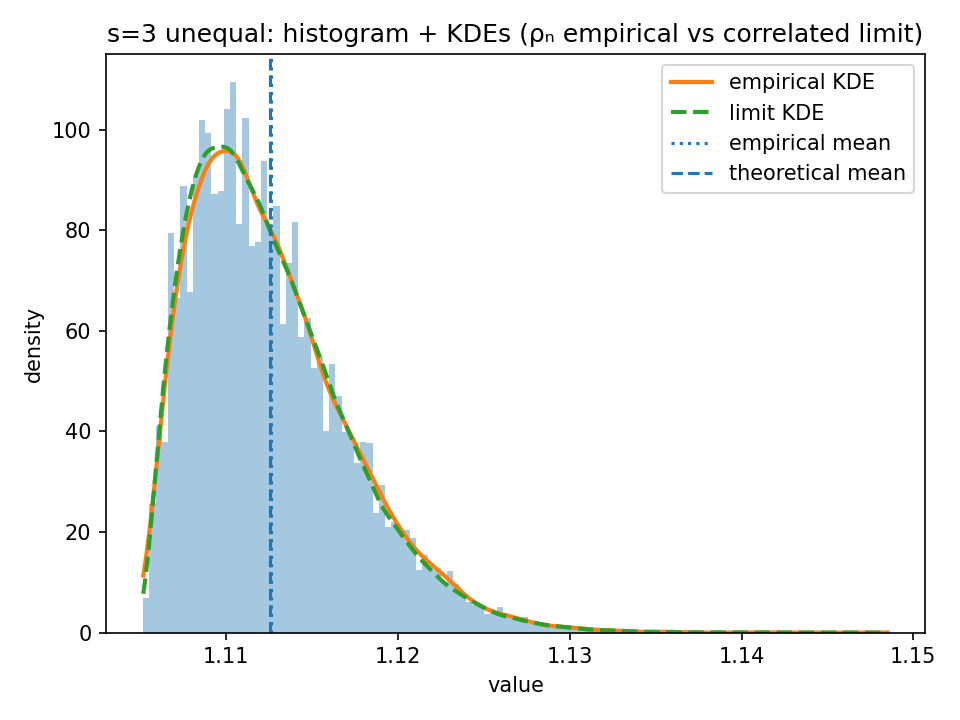}
  
  \caption{The histogram represents the empirical distribution, and its kernel density estimate (KDE) is shown as the green curve. These are compared against the theoretical limiting pdf (orange curve) for $\xi^{(3)} = \rho_\infty \cdot \max_{j \in \{j_1, j_2, j_3\}} G_j$.  The components of $\boldsymbol{G}$ have \textbf{(left)} equal variances and correlations; \textbf{(right)} unequal variances and are uncorrelated.}
  \label{fig:s3}
\end{figure}
\end{example}

\begin{example}
Assume $J=\{j_1,j_2,j_3,j_4\}$, equal variances $s_j=s$ for all $j\in J$, and equal pairwise correlation coefficient $\rho\in[-1/3,1]$. Then
\[
\mathbb E \left[\xi^{(4)}\right] = \rho_\infty\, c_4\,s \sqrt{1-\rho}, \qquad \operatorname{Var} \left(\xi^{(4)}\right) = \rho_\infty^{2}\,s^{2}\bigl(\rho+(1-\rho)\,v_4\bigr),
\]
where $c_4$ and $v_4$ are the mean and variance of the maximum of four independent standard normals, i.e. taking $Z_1,\ldots,Z_4 \stackrel{\text{i.i.d.}}{\sim} \mathcal{N}(0,1)$,
\[
c_4=\mathbb E \left[\max_{1\le k\le 4}Z_k\right]\approx 1.029375,\qquad v_4=\operatorname{Var} \left(\max_{1\le k\le 4}Z_k\right)\approx 0.491715.
\]
For $0\le\rho\le1$ this follows from the common-factor representation; for $-1/3\le\rho<0$ it follows from the mean/orthogonal-component decomposition of the equicorrelated Gaussian vector. The constants $c_4$ and $v_4$ do not have simple closed forms but can be computed numerically to arbitrary precision.
\end{example}

\subsection{Approximation error analysis} \label{subsec:approx_errors}

The CLT established above shows that for diagonal families, and more generally for families of upper(lower)-triangular matrices, the centered and scaled distribution of $\rho_n$ exhibits fluctuations at the $n^{-1/2}$ scale as $n\to\infty$. In numerical applications one works with finite $n$, so it is useful to quantify the finite-sample errors.

For a finite matrix family, the logarithmic variables are finitely supported. Such laws are not absolutely continuous, and a finitely supported law is not necessarily lattice-valued in the one-dimensional sense. Consequently, density-level Edgeworth expansions are not the right object for the matrix model. In this subsection we use finite-family expansions for expectations. The unique-maximizer case follows directly from the cumulant generating function. In the multi-maximizer case we state the leading finite-family expansion and, when a first-order weak Edgeworth expansion is available for the active finitely supported law, the corresponding next-order correction.

For later reference, note the exact multinomial identity: for every $q>0$,
\begin{equation}\label{eq:finite_multinomial_moment}
\mathbb E[\rho_n^q] = \sum_{\nu_1+\cdots+\nu_m=n} \binom{n}{\nu_1,\ldots,\nu_m} \prod_{i=1}^m p_i^{\nu_i} \left[ \max_{1\le j\le d} \exp\left(\sum_{i=1}^m \frac{\nu_i}{n}\log |\lambda_j^{(i)}|\right) \right]^q .
\end{equation}
This formula is exact for finite diagonal or triangular families and may be used directly when $m$ and $n$ are moderate.

\subsubsection{Unique maximizer}

\begin{theorem}\label{thm:edgeworth_exp_univariate}
Let $\rho_n(\mathcal{T},\mathbb P_{\mathcal{T}})$ be as in~\eqref{eq:rho_n_via_maxZ} and assume $J=\{j^\star\}$. Set
\[
\rho_\infty(\mathcal{T},\mathbb P_{\mathcal{T}}) = \rho_\infty^{(j^\star)}(\mathcal{T},\mathbb P_{\mathcal{T}}),
\]
and let $\kappa_r$ be the $r$-th cumulant of $Z^{(j^\star)}$ defined by~\eqref{def:Zvector}. Then, as $n\to\infty$,
\begin{align}
\mathbb E\big[\rho_n(\mathcal{T},\mathbb P_{\mathcal{T}})\big] &=\rho_\infty(\mathcal{T},\mathbb P_{\mathcal{T}}) \left(1+\frac{\kappa_2}{2n} + \frac{\kappa_3/6+\kappa_2^2/8}{n^2} + O(n^{-3})\right)+O(e^{-cn}),\label{eq:Avrho}\\
n \operatorname{Var}\big(\rho_n(\mathcal{T},\mathbb P_{\mathcal{T}})\big) &=\rho_\infty^2(\mathcal{T},\mathbb P_{\mathcal{T}}) \left(\kappa_2 + \frac{2\kappa_3+3\kappa_2^2}{2n} + O(n^{-2})\right)+O(ne^{-cn}),\label{eq:Varrho}
\end{align}
for some constant $c>0$. In particular, the exponentially small terms are negligible compared with all powers of $n^{-1}$.
\end{theorem}

\begin{proof}
Let
\[
Z \coloneqq Z^{(j^\star)}, \qquad \mu_Z \coloneqq \mathbb E[Z]=\log \rho_\infty(\mathcal T,\mathbb P_{\mathcal T}),
\]
and define the auxiliary variable
\[
\widetilde\rho_n\coloneqq \exp\left(\frac1n\sum_{k=1}^n Z_k\right).
\]
Since $J=\{j^\star\}$, for every $j\ne j^\star$ the gap
\[
\Delta_j\coloneqq \mu_Z-\mu_j>0
\]
is positive. The variables $Z_k^{(j)}-Z_k^{(j^\star)}$ are bounded and have mean $-\Delta_j$. Hence, by Hoeffding's inequality (see, e.g., \cite[Lemma 2.2]{boucheron2013concentration} and \cite{hoeffding1963probability}), there exist constants $C,c>0$ such that
\[
\mathbb P\left(\rho_n\ne \widetilde\rho_n\right)
\le
\sum_{j\ne j^\star}\mathbb P\left(\frac1n\sum_{k=1}^n\bigl(Z_k^{(j)}-Z_k^{(j^\star)}\bigr)\ge0\right)
\le Ce^{-cn}.
\]
Since all eigenvalue moduli in the finite family are bounded, $|\rho_n^q-\widetilde\rho_n^{\,q}|\le C_q\mathbf 1_{\{\rho_n\ne\widetilde\rho_n\}}$ for $q=1,2$. Thus the first two moments of $\rho_n$ and $\widetilde\rho_n$ differ by $O(e^{-cn})$.

Let $K_Z(t) \coloneqq \log \mathbb E[\mathrm e^{tZ}]$ be the cumulant generating function. Since $Z$ is finitely supported, $K_Z$ is analytic and
\[
K_Z(t)=\mu_Z t + \sum_{r\ge 2}\frac{\kappa_r}{r!}\,t^r.
\]
Therefore
\[
\mathbb E[\widetilde\rho_n]=\left(\mathbb E[\mathrm e^{Z/n}]\right)^n=\exp\bigl(nK_Z(1/n)\bigr),
\]
with
\[
nK_Z(1/n)=\mu_Z+\frac{\kappa_2}{2n}+\frac{\kappa_3}{6n^2}+\frac{\kappa_4}{24n^3}+O(n^{-4}).
\]
Expanding the exponential gives the formula for the expectation. Similarly,
\[
\mathbb E[\widetilde\rho_n^{\,2}]=\left(\mathbb E[\mathrm e^{2Z/n}]\right)^n=\exp\bigl(nK_Z(2/n)\bigr),
\]
and
\[
nK_Z(2/n)=2\mu_Z+\frac{2\kappa_2}{n}+\frac{4\kappa_3}{3n^2}+\frac{2\kappa_4}{3n^3}+O(n^{-4}).
\]
Subtracting $\mathbb E[\widetilde\rho_n]^2$ from $\mathbb E[\widetilde\rho_n^{\,2}]$ and multiplying by $n$ yields~\eqref{eq:Varrho}. The exponentially small error estimated above completes the proof for $\rho_n$.
\end{proof}

\begin{remark}
Assume $J=\{j^\star\}$ and $\kappa_2>0$. By the preceding proof and the Borel--Cantelli lemma (see, e.g., \cite[Theorem 2.3.1]{durrett2019probability}), the inactive coordinates overtake the active coordinate only finitely often almost surely. Hence the law of the iterated logarithm for the scalar sums gives the sharp bounds
\[
\liminf_{n\to\infty}
\frac{\rho_n(\mathcal{T},\mathbb P_{\mathcal{T}})-\rho_\infty(\mathcal{T},\mathbb P_{\mathcal{T}})}
{\sigma_\infty^{(j^\star)}\sqrt{2n^{-1}\log\log n}}
=-1, \qquad \limsup_{n\to\infty}
\frac{\rho_n(\mathcal{T},\mathbb P_{\mathcal{T}})-\rho_\infty(\mathcal{T},\mathbb P_{\mathcal{T}})}
{\sigma_\infty^{(j^\star)}\sqrt{2n^{-1}\log\log n}}
=1.
\]
Thus the largest almost-sure fluctuations scale as $n^{-1/2}\sqrt{\log\log n}$, whereas the typical Gaussian fluctuations scale as $n^{-1/2}$.
\end{remark}

\subsubsection{Multiple maximizers}

When $|J|>1$, the function determining the fluctuation is no longer linear but the maximum of the active coordinates. The following finite-family statement separates the always-valid leading terms from the optional first Edgeworth correction.

\begin{theorem}\label{thm:edgeworth_exp_multivariate}
Let $J=\{j_1,\ldots,j_s\}$ with $s>1$, and let $\Sigma_J =\Sigma_{\mathcal{T}, \mathbb{P}_{\mathcal{T}}}$ be the active covariance matrix defined in~\eqref{def:Sigma}--\eqref{def:mu_j}. Assume $\Sigma_J$ is positive definite. Define
\[
S_n \coloneqq \frac{1}{\sqrt n}\sum_{k=1}^n(\boldsymbol Z_{k,J}-\boldsymbol\mu_J),
\qquad
M_n\coloneqq \max_{j\in J} (S_n)_j,
\]
and let $\boldsymbol U\sim\mathcal N(\boldsymbol 0,\Sigma_J)$ with
\[
M(\boldsymbol U)\coloneqq \max_{j\in J}U_j.
\]
Set
\[
m_k\coloneqq \mathbb E[M(\boldsymbol U)^k]\quad(k=1,2,3),
\qquad
v_0\coloneqq m_2-m_1^2,
\qquad
v_1\coloneqq m_3-m_1m_2.
\]
Then, for the finite matrix-family model,
\begin{align}
\mathbb E[\rho_n]&=\rho_\infty\left(1+\frac{m_1}{\sqrt n}+o(n^{-1/2})\right),\label{eq:multi-leading-mean}\\
n\operatorname{Var}(\rho_n)&=\rho_\infty^2\left(v_0+o(1)\right).\label{eq:multi-leading-var}
\end{align}
Suppose in addition that the active centered finite-support law satisfies the first-order weak Edgeworth expansion for the test functions $f(\boldsymbol u)=M(\boldsymbol u)$ and $f(\boldsymbol u)=M(\boldsymbol u)^2$. More precisely, let $\alpha!=(\alpha_1!\cdots\alpha_s!)$, $\partial^\alpha=\partial_{u_1}^{\alpha_1}\cdots\partial_{u_s}^{\alpha_s}$, and let $\varphi_{\Sigma_J}$ denote the density of $\mathcal N(\boldsymbol 0,\Sigma_J)$. Define the generalized cumulants of $\boldsymbol Z_{1,J}-\boldsymbol\mu_J$ by
\[
\kappa_\alpha
\coloneqq
\left.\partial_\beta^\alpha\log\mathbb E\exp\bigl(\boldsymbol\beta^\top(\boldsymbol Z_{1,J}-\boldsymbol\mu_J)\bigr)\right|_{\boldsymbol\beta=0}.
\]
The $\Sigma_J$-Hermite polynomials are
\[
H_\alpha^{\Sigma_J}(\boldsymbol u)
\coloneqq
(-1)^{|\alpha|}\frac{1}{\varphi_{\Sigma_J}(\boldsymbol u)}\partial^\alpha\varphi_{\Sigma_J}(\boldsymbol u).
\]
Assume
\begin{equation}\label{eq:weak_edgeworth_expectation}
\mathbb E[f(S_n)]
=
\mathbb E[f(\boldsymbol U)]
+\frac{1}{\sqrt n}\sum_{|\alpha|=3}\frac{\kappa_\alpha}{\alpha!}\,
\mathbb E\bigl[f(\boldsymbol U)H_\alpha^{\Sigma_J}(\boldsymbol U)\bigr]
+o(n^{-1/2})
\end{equation}
for these two functions. Define
\[
h_\alpha\coloneqq\mathbb E\bigl[M(\boldsymbol U)H_\alpha^{\Sigma_J}(\boldsymbol U)\bigr],
\qquad
r_\alpha\coloneqq\mathbb E\bigl[M(\boldsymbol U)^2H_\alpha^{\Sigma_J}(\boldsymbol U)\bigr].
\]
Then
\begin{align}
\label{eq:multi-mean-expansion}
\mathbb E[\rho_n]
&=
\rho_\infty\Bigg[1+\frac{m_1}{\sqrt n}
+\frac{1}{n}\bigg(\frac{m_2}{2}+\sum_{|\alpha|=3}\frac{\kappa_\alpha}{\alpha!}h_\alpha\bigg)
+o(n^{-1})\Bigg],\\[3pt]
\label{eq:multi-var-expansion}
n\operatorname{Var}(\rho_n)
&=
\rho_\infty^2\Bigg[v_0
+\frac{1}{\sqrt n}\Bigg(v_1+
\sum_{|\alpha|=3}\frac{\kappa_\alpha}{\alpha!}\bigl(r_\alpha-2m_1h_\alpha\bigr)\Bigg)
+o(n^{-1/2})\Bigg].
\end{align}
\end{theorem}

\begin{remark}
For finite non-lattice laws satisfying the usual Cram\'er condition,~\eqref{eq:weak_edgeworth_expectation} is the standard weak Edgeworth expansion. In a genuinely lattice case, the exact multinomial formula~\eqref{eq:finite_multinomial_moment} remains valid, but the next-order weak Edgeworth term may require the corresponding periodic lattice correction.
\end{remark}

\begin{proof}
For every inactive coordinate $j\notin J$, the gap $\Delta_j=\mu^\star-\mu_j$ is positive. As in the proof of~\cref{thm:edgeworth_exp_univariate}, Hoeffding's inequality~\cite{hoeffding1963probability} gives
\[
\mathbb P\left(\rho_n\ne \rho_\infty\exp\left(\frac{M_n}{\sqrt n}\right)\right)\le Ce^{-cn}
\]
for some $C,c>0$. Hence the error made by replacing $\rho_n$ with $\rho_\infty\exp(M_n/\sqrt n)$ is exponentially small in the first two moments.

By the multivariate CLT, $S_n\Rightarrow\boldsymbol U$. Since the variables $S_n$ have uniformly bounded moments of every order and $M_n$ has at most linear growth in $S_n$, the family $\{M_n^q:q\le3\}$ is uniformly integrable. Therefore $\mathbb E[M_n]\to m_1$ and $\mathbb E[M_n^2]\to m_2$. Expanding
\[
\exp(M_n/\sqrt n)=1+\frac{M_n}{\sqrt n}+\frac{M_n^2}{2n}+O\left(\frac{|M_n|^3}{n^{3/2}}\right)
\]
and using uniform integrability gives~\eqref{eq:multi-leading-mean}. The same argument applied to $\exp(2M_n/\sqrt n)$ gives~\eqref{eq:multi-leading-var}. Assume now that~\eqref{eq:weak_edgeworth_expectation} holds. Then
\[
\mathbb E[M_n]=m_1+\frac{1}{\sqrt n}\sum_{|\alpha|=3}\frac{\kappa_\alpha}{\alpha!}h_\alpha+o(n^{-1/2}),
\]
and
\[
\mathbb E[M_n^2]=m_2+\frac{1}{\sqrt n}\sum_{|\alpha|=3}\frac{\kappa_\alpha}{\alpha!}r_\alpha+o(n^{-1/2}),
\qquad
\mathbb E[M_n^3]=m_3+o(1).
\]
Substituting these identities into the Taylor expansions of $\mathbb E\exp(M_n/\sqrt n)$ and $\mathbb E\exp(2M_n/\sqrt n)$ yields~\eqref{eq:multi-mean-expansion} and~\eqref{eq:multi-var-expansion} after subtracting the square of the mean. The exponentially small inactive-coordinate error does not affect the displayed polynomial orders.
\end{proof}

\subsection*{Numerical experiments}

We validate the finite-family formulas using exact evaluations of~\eqref{eq:finite_multinomial_moment}, so the comparisons below contain no Monte Carlo noise. The examples are chosen to reflect the three regimes of~\cref{thm:edgeworth_exp_univariate,thm:edgeworth_exp_multivariate}: aligned unique maximizer, non-aligned unique maximizer, and coalescence with $|J|>1$.

\begin{enumerate}[label=\arabic*)]
\item \textbf{Aligned unique maximizer.} Consider
\[
D_1=\operatorname{diag}(4.0,1.3),\quad
D_2=\operatorname{diag}(2.2,1.5),\quad
D_3=\operatorname{diag}(3.4,1.1),
\]
with probabilities $\boldsymbol p=(0.5,0.3,0.2)$. Here the first coordinate dominates matrix-by-matrix, so $J=\{1\}$ and $\rho_n=\widetilde\rho_n$ exactly. Therefore
\[
\mathbb E[\rho_n^q]
=
\left(\sum_{i=1}^3 p_i |\lambda_1^{(i)}|^{q/n}\right)^n,
\qquad q=1,2.
\]

\item \textbf{Non-aligned unique maximizer.} Consider
\[
D_1=\operatorname{diag}(4.5,1.1,0.7),\quad
D_2=\operatorname{diag}(2.0,3.8,0.8),\quad
D_3=\operatorname{diag}(3.0,1.4,0.6),
\]
with $\boldsymbol p=(0.55,0.25,0.20)$. Then $J=\{1\}$ in the sense of expected growth, but the second coordinate dominates on $D_2$. This is the regime where~\cref{thm:edgeworth_exp_univariate} requires the exponentially small correction between $\rho_n$ and the surrogate
\[
\widetilde\rho_n
\coloneqq
\exp\left(\frac1n\sum_{k=1}^n Z_k^{(1)}\right).
\]

\item \textbf{Coalescence models.} For $|J|=2$ we take
\[
D_1^{(2)}=\operatorname{diag}(8.0,2.0,0.8),\quad
D_2^{(2)}=\operatorname{diag}(2.0,8.0,0.8),\quad
D_3^{(2)}=\operatorname{diag}(5.0,5.0,0.8),
\]
with $\boldsymbol p=(1/3,\,1/3,\,1/3)$. Then $J=\{1,2\}$ and $\Sigma_J$ is positive definite. For $|J|=3$ we take
\begin{align*}
D_1^{(3)}& =\operatorname{diag}(9.0,2.0,2.0,0.9),\quad
D_2^{(3)}=\operatorname{diag}(2.0,9.0,2.0,0.9),\\
D_3^{(3)}&=\operatorname{diag}(2.0,2.0,9.0,0.9),\quad
D_4^{(3)}=\operatorname{diag}(4.0,4.0,4.0,0.9),
\end{align*}
with $\boldsymbol p=(1/4,\,1/4,\,1/4,\,1/4)$. Then $J=\{1,2,3\}$ and again $\Sigma_J$ is positive definite. In both coalescence models, the exact mean and variance are computed from~\eqref{eq:finite_multinomial_moment}, while the asymptotic constants come from the Gaussian formulas in the preceding examples.
\end{enumerate}

\Cref{fig:edgeworth-finite-unique} compares the exact first two moments with the leading and corrected formulas from~\cref{thm:edgeworth_exp_univariate}. In the aligned model, where $\rho_n=\widetilde\rho_n$ exactly, the corrected approximation improves the mean error from order $n^{-1}$ to order $n^{-3}$ and the scaled-variance error from order $n^{-1}$ to order $n^{-2}$. In the non-aligned model the same improvement is visible after a short transient, which is precisely the regime where the exponentially small replacement error becomes smaller than the polynomial terms.

\begin{figure}[ht!]
  \centering
  \includegraphics[width=.92\textwidth]{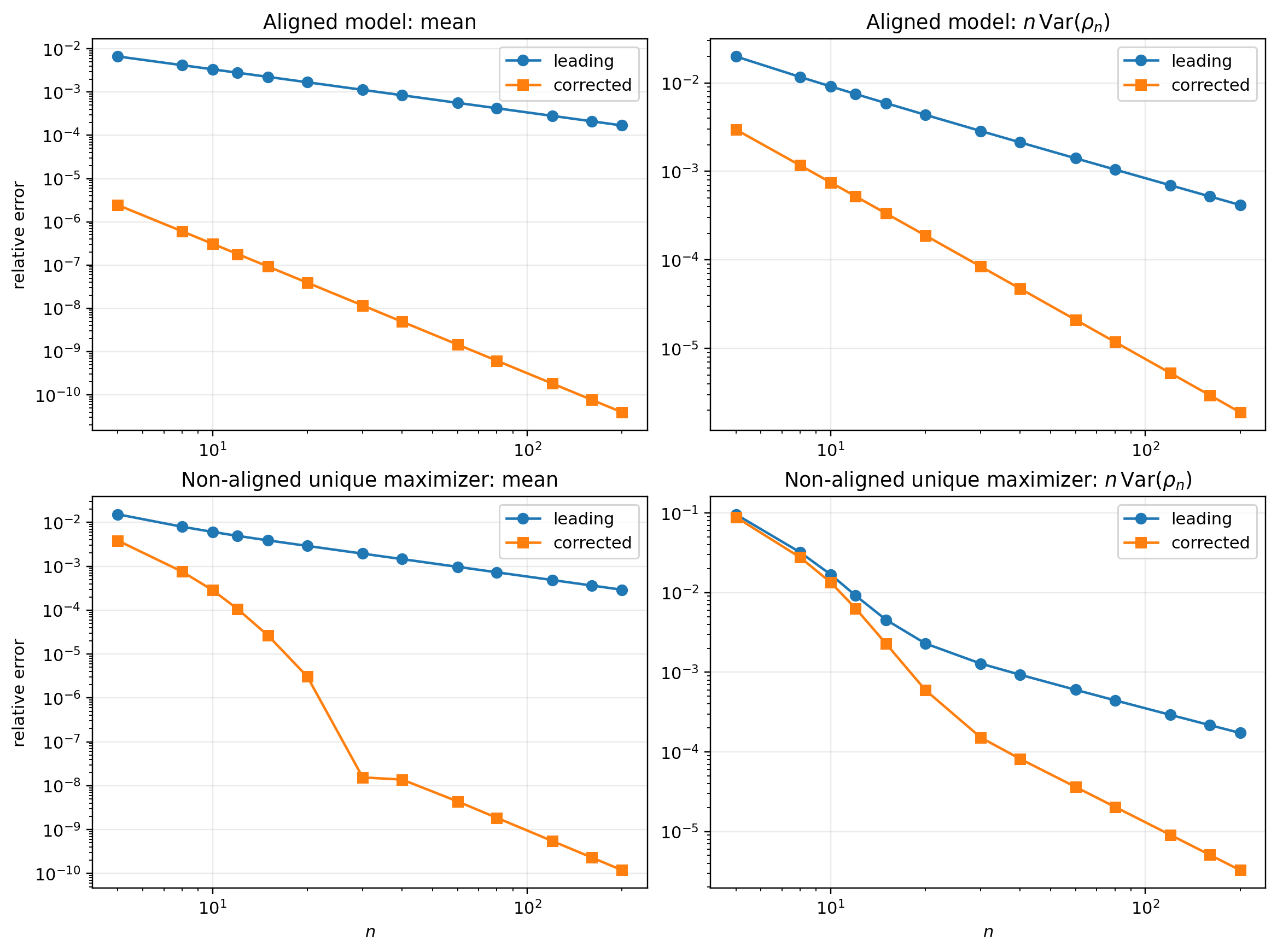}
  \caption{Finite-family validation of~\cref{thm:edgeworth_exp_univariate}. Top row: aligned unique-maximizer model. Bottom row: non-aligned unique-maximizer model. Blue circles show the leading approximations, orange squares the corrected moment formulas from~\eqref{eq:Avrho}--\eqref{eq:Varrho}, and the reference values are exact moments computed from~\eqref{eq:finite_multinomial_moment}.}
  \label{fig:edgeworth-finite-unique}
\end{figure}

The proof of~\cref{thm:edgeworth_exp_univariate} relies on the fact that, in the non-aligned unique-maximizer case, inactive coordinates overtake the active coordinate only with exponentially small probability. \Cref{fig:edgeworth-finite-switch} verifies this directly. The left panel plots $\mathbb P(\rho_n\neq\widetilde\rho_n)$, while the right panel compares the exact moments of $\rho_n$ with those of the surrogate $\widetilde\rho_n$. Both discrepancies decay exponentially, in agreement with the proof.

\begin{figure}[ht!]
  \centering
  \includegraphics[width=.92\textwidth]{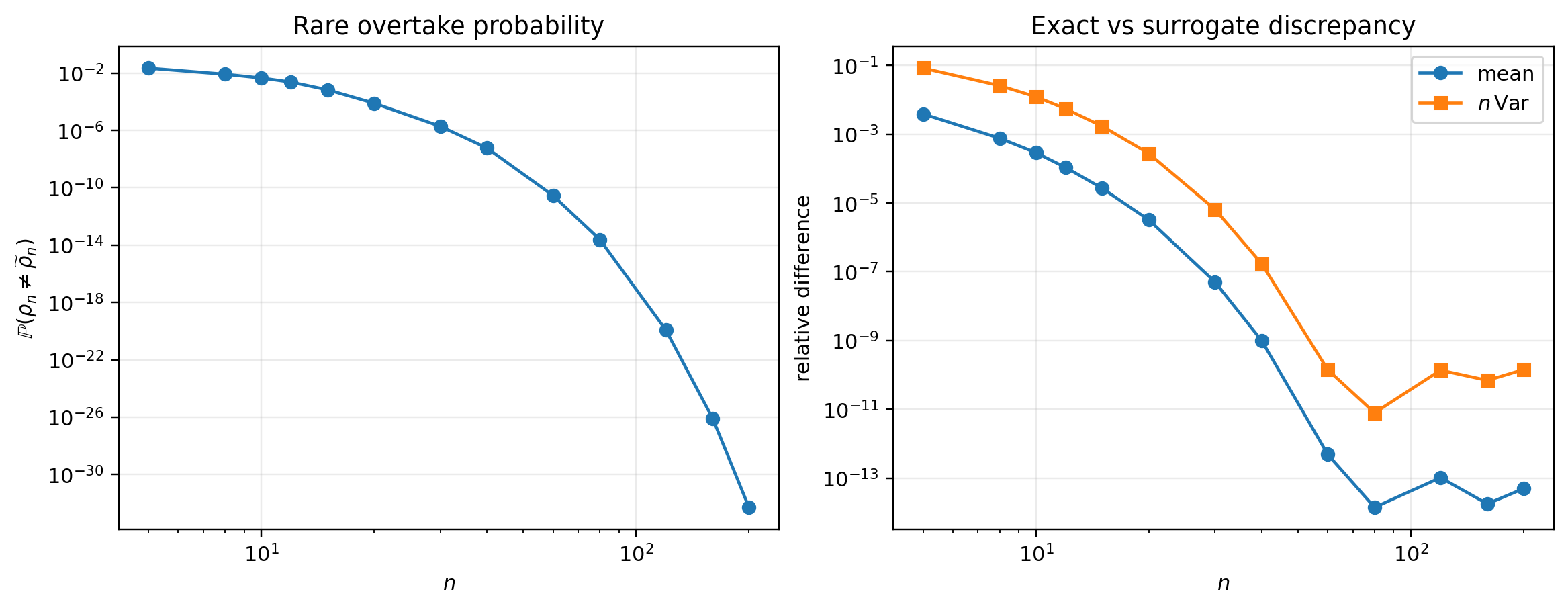}
  \caption{Non-aligned unique-maximizer experiment. Left: exact probability that an inactive coordinate overtakes the active one. Right: relative difference between the exact moments of $\rho_n$ and those of the surrogate $\widetilde\rho_n$ used in the proof of~\cref{thm:edgeworth_exp_univariate}.}
  \label{fig:edgeworth-finite-switch}
\end{figure}

For the coalescence models, the finite-family law is discrete, so the first lower-order weak Edgeworth term may carry periodic lattice corrections. For that reason, \cref{fig:edgeworth-finite-multimax} focuses on the universal leading quantities from~\eqref{eq:multi-leading-mean}--\eqref{eq:multi-leading-var}. In the $|J|=2$ model the theoretical constants are $m_1\approx 0.45156$ and $v_0\approx 0.12746$; in the $|J|=3$ model they are $m_1\approx 0.63644$ and $v_0\approx 0.13479$. The exact scaled mean and variance converge to these limits, with a noticeably slower approach than in the unique-maximizer regime, which is consistent with the absence of a universal first correction.

\begin{figure}[ht!]
  \centering
  \includegraphics[width=.92\textwidth]{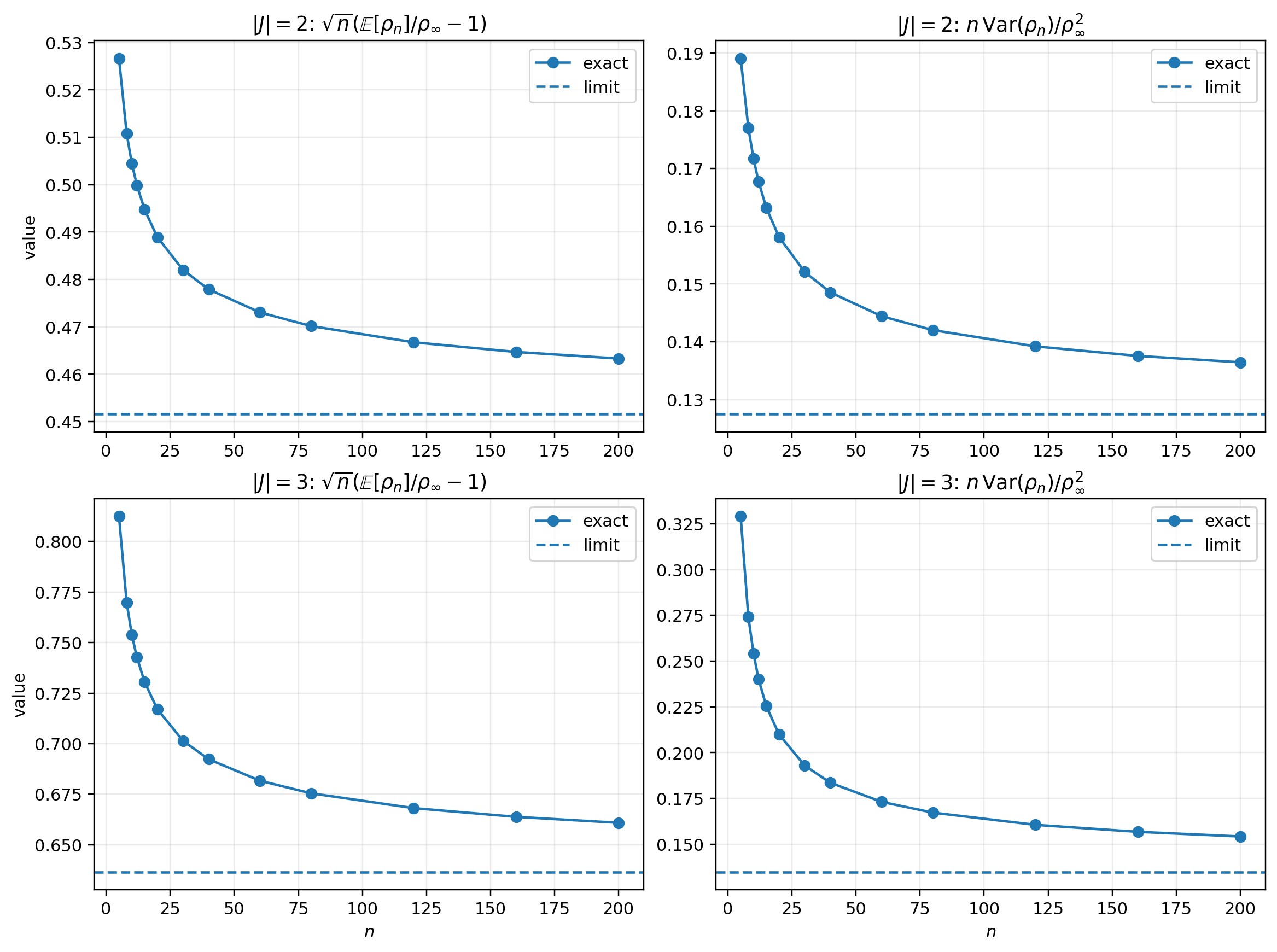}
  \caption{Finite-family coalescence experiments. Top row: $|J|=2$ model. Bottom row: $|J|=3$ model. Solid curves are exact values computed from~\eqref{eq:finite_multinomial_moment}; dashed lines show the limiting constants $m_1$ and $v_0$ from~\eqref{eq:multi-leading-mean}--\eqref{eq:multi-leading-var}.}
  \label{fig:edgeworth-finite-multimax}
\end{figure}

\subsection{Large deviations and tail probabilities} \label{subsec:large_deviations}

The CLT and its Edgeworth refinements characterize the \emph{typical} fluctuations of $\rho_n$ around $\rho_\infty$ at the scale $n^{-1/2}$. However, the motivating questions posed in the introduction, how likely is worst-case (JSR) growth, and how concentrated are typical trajectories around $\rho_\infty$, require estimates on the tails of the distribution of $\rho_n$, far from the Gaussian regime. Such estimates are provided by a \emph{large deviation principle} (LDP), which we now establish by exploiting the i.i.d. structure of the log-spectral-radius.

\begin{theorem}\label{thm:LDP}
Let $\mathcal{T} = \{T_1, \ldots, T_m\}$ be a family of upper(lower)-triangular matrices equipped with a probability distribution $\mathbb{P}_{\mathcal{T}}$. Assume there exists $j^\star\in\{1,\dots,d\}$ such that
\[
|\lambda_{j^\star}^{(i)}|>\max_{j\ne j^\star} |\lambda_j^{(i)}|,
\qquad i=1,\dots,m.
\]
Define
\begin{equation}\label{def:Lambda_LDP}
\Lambda(t) \coloneqq \log \sum_{i=1}^m p_i \, \bigl|\lambda_{j^\star}^{(i)}\bigr|^t, \qquad t \in \mathbb{R},
\end{equation}
and
\begin{equation}\label{def:rate_function}
I(z) \coloneqq \sup_{t \in \mathbb{R}} \left( tz - \Lambda(t) \right), \qquad z \in \mathbb{R}.
\end{equation}
Then $\rho_n$ satisfies a large deviation principle with speed $n$ and good rate function
\[
\widetilde I(r)\coloneqq I(\log r),\qquad r>0.
\]
That is, for any Borel set $B \subseteq (0, \infty)$,
\begin{equation}\label{eq:LDP_bounds}
-\inf_{r \in B^\circ} \widetilde{I}(r) \leq \liminf_{n \to \infty} \frac{1}{n} \log \mathbb{P}(\rho_n \in B) \leq \limsup_{n \to \infty} \frac{1}{n} \log \mathbb{P}(\rho_n \in B) \leq -\inf_{r \in \overline{B}} \widetilde{I}(r).
\end{equation}
Moreover, $I\ge0$ and $I(\log\rho_\infty)=0$. If $Z=\log |\lambda_{j^\star}^{(I)}|$ is non-degenerate, then $I$ is strictly convex on the interior of its effective domain. If $Z$ is degenerate, then the LDP is trivial: $I$ is $0$ at $\log\rho_\infty$ and $+\infty$ away from this point.
\end{theorem}

\begin{proof}
Since the index $j^\star$ dominates every matrix in the family, the RSR satisfies
\[
\rho_n = \exp\biggl(\frac{1}{n}\sum_{k=1}^n Z_k^{(j^\star)}\biggr),
\]
where $Z_k^{(j^\star)} = \log |Y_k^{(j^\star)}|$. The log-moment generating function of $Z_1^{(j^\star)}$ is
\[
\log \mathbb{E}\bigl[\mathrm{e}^{t Z_1^{(j^\star)}}\bigr] = \log \sum_{i=1}^m p_i \, \bigl|\lambda_{j^\star}^{(i)}\bigr|^t = \Lambda(t),
\]
which is finite for all $t \in \mathbb{R}$ on the support of the law. By Cram\'er's theorem (see, e.g.,~\cite[Chapter~27]{klenke2014probability}), $n^{-1}\sum_{k=1}^n Z_k^{(j^\star)}$ satisfies an LDP with speed $n$ and good rate function $I$. The map $z\mapsto \mathrm e^z$ is continuous, so the contraction principle (see~\cite[Theorem~27.5]{klenke2014probability}) yields the LDP for $\rho_n$ with rate function $\widetilde I(r)=I(\log r)$.

The minimum of $I$ is attained at $z=\Lambda'(0)=\mathbb E Z_1^{(j^\star)}=\log\rho_\infty$, where $I(\log\rho_\infty)=0$. If $Z$ is non-degenerate, then $\Lambda$ is strictly convex and essentially smooth on the finite interval generated by the support, which gives strict convexity of $I$ on the interior of its effective domain. If $Z$ is deterministic, the stated degenerate form follows from Cram\'er's theorem.
\end{proof}

The LDP provides exponential-rate estimates for the probability that $\rho_n$ deviates from $\rho_\infty$ in either direction. These estimates can be connected directly to the JSR and LSR, thereby quantifying the likelihood of extreme growth scenarios.

\begin{corollary}\label{cor:tail_JSR_LSR}
Under the hypotheses of~\cref{thm:LDP}, assume that $Z=\log |\lambda_{j^\star}^{(I)}|$ is non-degenerate. Let $I_+\coloneqq\{i \ : \ p_i>0\}$ and define the extremal values by
\[
\hat{\lambda} \coloneqq \max_{i\in I_+} \bigl|\lambda_{j^\star}^{(i)}\bigr|, \qquad \check{\lambda} \coloneqq \min_{i\in I_+} \bigl|\lambda_{j^\star}^{(i)}\bigr|.
\]
If all $p_i$ are positive and the aligned dominance assumption~\eqref{def:aligned dom convergence} holds, then these coincide with the JSR and LSR of the family, respectively. The following hold:
\begin{enumerate}[label=\textnormal{(\roman*)}]
\item For any $r \in (\rho_\infty, \hat{\lambda}]$,
\begin{equation}\label{eq:upper_tail}
\lim_{n \to \infty} \frac{1}{n} \log \mathbb{P}(\rho_n \geq r) = -I(\log r).
\end{equation}
\item For any $r \in [\check{\lambda}, \rho_\infty)$,
\begin{equation}\label{eq:lower_tail}
\lim_{n \to \infty} \frac{1}{n} \log \mathbb{P}(\rho_n \leq r) = -I(\log r).
\end{equation}
\item For every fixed $\delta$ satisfying $0<\delta<\hat\lambda-\rho_\infty$,
\begin{equation}\label{eq:JSR_tail}
\mathbb{P}\bigl(\rho_n \geq \hat{\lambda} - \delta\bigr) = \exp\bigl(-n \, I\bigl(\log(\hat{\lambda} - \delta)\bigr) + o(n)\bigr).
\end{equation}
If $\delta\ge \hat\lambda-\rho_\infty$, then the threshold is at or below the typical value and the exponential rate is $0$ rather than $-I(\log(\hat\lambda-\delta))$.
\end{enumerate}
\end{corollary}

\begin{proof}
For $r\in(\rho_\infty,\hat\lambda)$ and $r\in(\check\lambda,\rho_\infty)$, the statements follow from~\eqref{eq:LDP_bounds} applied to $B=[r,\infty)$ and $B=(0,r]$, respectively. Since $Z$ is non-degenerate, $I$ is continuous and strictly increasing on $(\log\rho_\infty,\log\hat\lambda)$ and strictly decreasing on $(\log\check\lambda,\log\rho_\infty)$, so the infima over the interior and closure of these sets coincide. It remains only to discuss the endpoints. Let
\[
p_{\max}\coloneqq \sum_{i\in I_+:\, |\lambda_{j^\star}^{(i)}|=\hat\lambda}p_i,
\qquad
p_{\min}\coloneqq \sum_{i\in I_+:\, |\lambda_{j^\star}^{(i)}|=\check\lambda}p_i.
\]
Then
\[
\mathbb P(\rho_n\ge \hat\lambda)=p_{\max}^n,
\qquad
\mathbb P(\rho_n\le \check\lambda)=p_{\min}^n.
\]
Moreover, the Legendre transform satisfies $I(\log\hat\lambda)=-\log p_{\max}$ and $I(\log\check\lambda)=-\log p_{\min}$, obtained by letting $t\to+\infty$ and $t\to-\infty$ in~\eqref{def:rate_function}, respectively. This proves the endpoint cases. Statement (iii) is (i) with $r=\hat\lambda-\delta$, under the stated restriction on $\delta$.
\end{proof}

\begin{remark}
In the non-degenerate case, the rate function $\widetilde{I}$ has a natural local interpretation. Near $r = \rho_\infty$, a second-order Taylor expansion using $I''(\log\rho_\infty) = 1/\sigma_Z^2$ yields
\[
\widetilde{I}(r) \approx \frac{(\log r - \log \rho_\infty)^2}{2\sigma_Z^2} \approx \frac{(r - \rho_\infty)^2}{2\sigma_\infty^2},
\]
recovering the Gaussian fluctuations of the CLT. For larger deviations, the full rate function captures the non-Gaussian structure of the tails. On the interior of the effective domain, the rate function is explicitly computable from~\eqref{def:Lambda_LDP}--\eqref{def:rate_function}: the maximizer $t^\star$ satisfies
\[
\Lambda'(t^\star) = \log r, \qquad I(\log r) = t^\star \log r - \Lambda(t^\star).
\]
\end{remark}

When multiple maximizers are present, the LDP extends through the multivariate Cram\'er theorem and the contraction principle.

\begin{theorem}\label{thm:LDP_multi}
Let $\mathcal{T}$ be as in~\cref{thm:clt_general} and define the multivariate log-moment generating function
\[
\Lambda(\boldsymbol{t}) \coloneqq \log \sum_{i=1}^m p_i \exp\!\left(\sum_{j=1}^d t_j \log \bigl|\lambda_j^{(i)}\bigr|\right), \qquad \boldsymbol{t} \in \mathbb{R}^d,
\]
and the multivariate rate function $I(\boldsymbol{z}) := \sup_{\boldsymbol{t} \in \mathbb{R}^d}\bigl(\boldsymbol{t} \cdot \boldsymbol{z} - \Lambda(\boldsymbol{t})\bigr)$. Then $\rho_n$ satisfies an LDP with speed $n$ and good rate function
\begin{equation}\label{eq:rate_multi}
\widetilde{I}(r) \coloneqq \inf \left\{ I(\boldsymbol{z}) \ : \ \boldsymbol{z} \in \mathbb{R}^d, \; \max_{1 \leq j \leq d} z_j = \log r \right\}, \qquad r > 0.
\end{equation}
\end{theorem}

\begin{proof}
Since $\boldsymbol{Z}_1$ takes finitely many values, $\Lambda(\boldsymbol{t})$ is finite for all $\boldsymbol{t} \in \mathbb{R}^d$. The multivariate Cram\'er theorem gives an LDP for $\frac{1}{n}\sum_{k=1}^n \boldsymbol{Z}_k$ in $\mathbb{R}^d$ with rate function $I$. The map $\boldsymbol{z} \mapsto \exp(\max_j z_j)$ is continuous, so the contraction principle yields the LDP as above.
\end{proof}

\subsection*{Numerical experiments}

We validate the LDP of~\cref{thm:LDP} and~\cref{cor:tail_JSR_LSR} by Monte Carlo simulation. Consider $m=3$ diagonal matrices in $\mathbb R^{3\times 3}$, e.g.
\[
D_1=\operatorname{diag}(3.5,1.2,0.8),\quad
D_2=\operatorname{diag}(2.0,1.5,0.6),\quad
D_3=\operatorname{diag}(4.0,0.9,1.1),
\]
with $\boldsymbol{p}=(0.3,0.5,0.2)$. The dominant eigenvalues are $|\lambda_1^{(i)}|\in\{3.5,2.0,4.0\}$, giving $\rho_\infty\approx 2.7174$, $\hat\lambda=4.0$ (JSR), and $\check\lambda=2.0$ (LSR).

For each product length $n\in\{10,20,50,100,200,500\}$, we generate $N_{\mathrm{MC}}=5\times 10^5$ independent realizations of $\rho_n$ and estimate the tail probabilities $\mathbb P(\rho_n\ge r)$ and $\mathbb P(\rho_n\le r)$ for several values of $r$ above and below $\rho_\infty$. 

\cref{fig:ldp-validation} displays the empirical rates $-\frac{1}{n}\log\mathbb P(\rho_n\ge r)$ (upper tail, left panel) and $-\frac{1}{n}\log\mathbb P(\rho_n\le r)$ (lower tail, right panel) as a function of $n$, together with the theoretical rate $\widetilde I(r)$ (dashed lines).

In both panels, the empirical rates converge to the predicted values as $n$ increases, confirming the exponential decay rates given by~\cref{cor:tail_JSR_LSR}.

\begin{figure}[ht!]
  \centering
  \includegraphics[width=.85\textwidth]{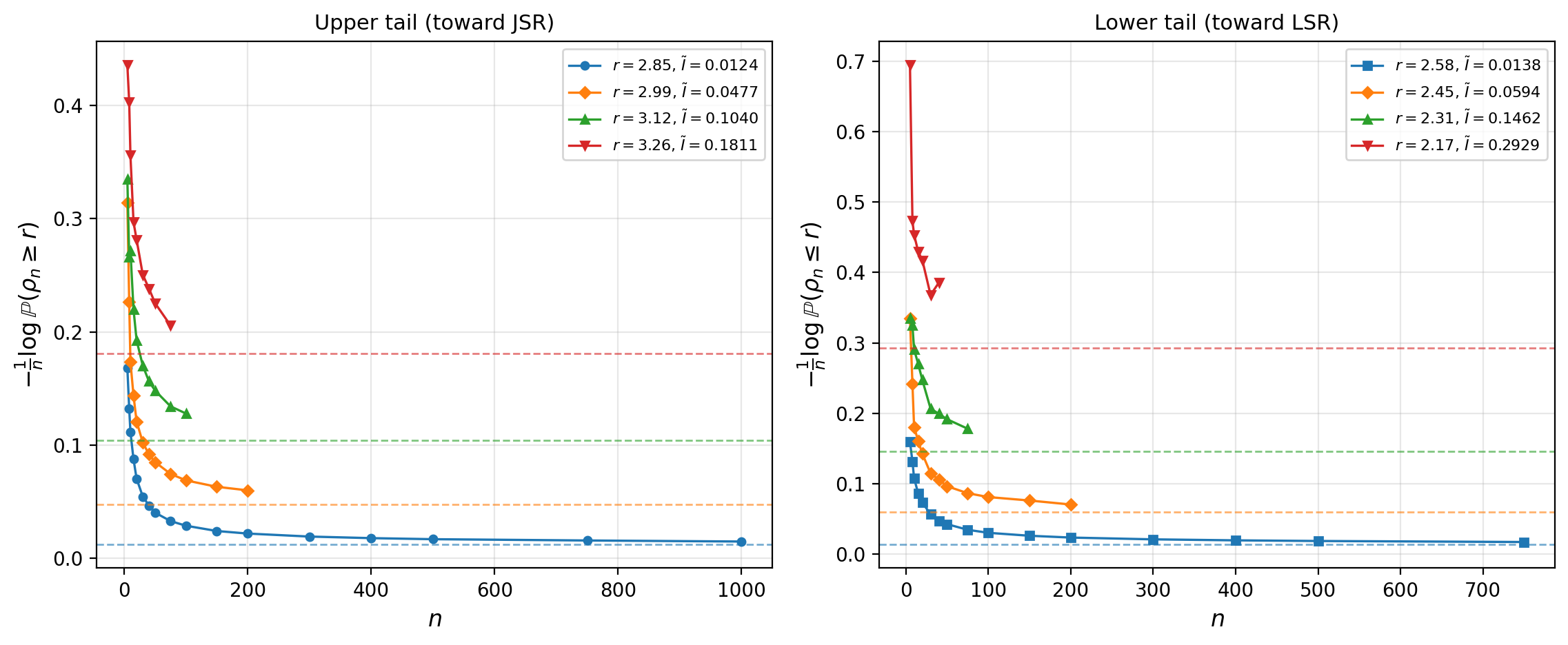}
  \caption{Validation of~\cref{thm:LDP}. Empirical rates $-\frac1n\log\mathbb P(\rho_n\ge r)$ (left) and $-\frac1n\log\mathbb P(\rho_n\le r)$ (right) for several threshold values $r$. Dashed lines indicate the theoretical $\widetilde I(r)=I(\log r)$ from~\eqref{def:rate_function}. In both tails, the MC estimates converge to the predicted rates as $n$ increases.}
  \label{fig:ldp-validation}
\end{figure}

\subsection{Dependence on the probability distribution} \label{subsec:rho_infty_prob}

The limit $\rho_\infty$ established in~\cref{thm:LLN gen diag} depends on $\boldsymbol{p} = (p_1, \ldots, p_m)$, and it is natural to ask how this dependence is structured. The following proposition answers the question and connects the RSR to the JSR and the LSR.

\begin{proposition}\label{prop:rho_infty_p}
Let $\mathcal{S} = \{D_1, \ldots, D_m\}$ be a diagonal family, and let $\Delta_m := \{\boldsymbol{p} \in [0,1]^m : \sum_i p_i = 1\}$ denote the probability simplex. Then $\boldsymbol{p} \mapsto \rho_\infty(\mathcal{S}, \boldsymbol{p})$, defined by~\eqref{def:rho_infty2}, satisfies:
\begin{enumerate}[label=\textup{(\roman*)}]
\item The function $\boldsymbol{p} \mapsto \log \rho_\infty(\mathcal{S}, \boldsymbol{p})$ is convex on $\Delta_m$.

\item The function $\boldsymbol{p} \mapsto \rho_\infty(\mathcal{S}, \boldsymbol{p})$ is continuous on $\Delta_m$.

\item For diagonal families,
\begin{equation}\label{eq:JSR_LSR_bound}
\check{\rho}(\mathcal S) = \min_{\boldsymbol p\in\Delta_m}\rho_\infty(\mathcal S,\boldsymbol p), \qquad \hat{\rho}(\mathcal S) = \max_{\boldsymbol p\in\Delta_m}\rho_\infty(\mathcal S,\boldsymbol p),
\end{equation}
and therefore $\check{\rho}(\mathcal S)\le \rho_\infty(\mathcal S,\boldsymbol p)\le \hat{\rho}(\mathcal S)$ for $\boldsymbol p\in\Delta_m$. Moreover,
\[
\hat{\rho}(\mathcal{S}) = \max_{1 \leq j \leq d} \max_{1 \leq i \leq m} |\lambda_j^{(i)}| = \sup_{\boldsymbol{p} \in \Delta_m} \rho_\infty(\mathcal{S}, \boldsymbol{p}),
\]
whereas the lower spectral radius is given by the convex optimization problem
\begin{equation}\label{eq:lsr_lp}
\log\check\rho(\mathcal S)
=
\min_{\boldsymbol p\in\Delta_m}\max_{1\le j\le d}\sum_{i=1}^m p_i\log |\lambda_j^{(i)}|.
\end{equation}

\item For $\boldsymbol p$ in the interior of $\Delta_m$ such that $J(\boldsymbol p)=\{j^\star\}$, the function $\boldsymbol p\mapsto \rho_\infty$ is differentiable along tangent directions $h=(h_1,\dots,h_m)$ satisfying $\sum_{i=1}^m h_i=0$, and
\begin{equation}\label{eq:rho_infty_gradient}
D\log \rho_\infty(\mathcal S,\boldsymbol p)[h] = \sum_{i=1}^m h_i \log |\lambda_{j^\star}^{(i)}|,
\end{equation}
equivalently,
\[
D\rho_\infty(\mathcal S,\boldsymbol p)[h] = \rho_\infty(\mathcal S,\boldsymbol p) \sum_{i=1}^m h_i \log |\lambda_{j^\star}^{(i)}|.
\]
In particular, if one transfers mass from index $k$ to index $i$, i.e. $h=e_i-e_k$, then
\[
D\rho_\infty(\mathcal S,\boldsymbol p)[e_i-e_k] = \rho_\infty(\mathcal S,\boldsymbol p) \bigl(\log |\lambda_{j^\star}^{(i)}|-\log |\lambda_{j^\star}^{(k)}|\bigr).
\]
\end{enumerate}
\end{proposition}

\begin{proof}
We divide the proof into each point:
\begin{enumerate}[label=\textup{(\roman*)}]
\item From~\eqref{def:rho_infty2},
\[
\log \rho_\infty(\mathcal{S}, \boldsymbol{p}) = \max_{1 \leq j \leq d} \sum_{i=1}^m p_i \log|\lambda_j^{(i)}|.
\]
Each map $\boldsymbol{p} \mapsto \sum_{i} p_i \log|\lambda_j^{(i)}|$ is linear, hence convex, on $\Delta_m$. Since the pointwise maximum of finitely many convex functions is convex, $\log \rho_\infty$ is convex on $\Delta_m$.

\item As the pointwise maximum of finitely many continuous functions, $\log \rho_\infty$ is continuous. Since $\exp$ is continuous, $\rho_\infty$ is continuous on $\Delta_m$.

\item For any word $\sigma \in [m]^n$, let $\nu_i(\sigma)$ denote the number of occurrences of index $i$, and set $p_i^{(n)} := \nu_i(\sigma)/n$. Then
\[
\rho_n(\sigma)=\max_{1\le j\le d}\prod_{i=1}^m |\lambda_j^{(i)}|^{p_i^{(n)}}=\rho_\infty(\mathcal S,\boldsymbol p^{(n)}).
\]
Thus $\rho_n(\sigma)$ is obtained by evaluating $\rho_\infty(\mathcal S,\cdot)$ at a rational point of the simplex. Conversely, rational points are dense in $\Delta_m$, and $\rho_\infty(\mathcal S,\cdot)$ is continuous by (ii). Therefore
\[
\hat\rho(\mathcal S)=\sup_{\boldsymbol p\in\Delta_m}\rho_\infty(\mathcal S,\boldsymbol p),\qquad \check\rho(\mathcal S)=\inf_{\boldsymbol p\in\Delta_m}\rho_\infty(\mathcal S,\boldsymbol p).
\]
Since $\Delta_m$ is compact, the inf and sup are attained. The formula for $\hat\rho(\mathcal S)$ follows because, for fixed $j$, the weighted geometric mean $\prod_i |\lambda_j^{(i)}|^{p_i}$ is maximized by placing all mass on an index realizing $\max_i |\lambda_j^{(i)}|$. The formula~\eqref{eq:lsr_lp} follows by taking logarithms.

\item When $J(\boldsymbol p)=\{j^\star\}$, the maximizing index is locally constant, so in a neighborhood of $\boldsymbol p$ one has
\[
\log \rho_\infty(\mathcal S,\boldsymbol p)=\sum_{i=1}^m p_i\log |\lambda_{j^\star}^{(i)}|.
\]
Differentiating along tangent directions $h$ with $\sum_i h_i=0$ gives~\eqref{eq:rho_infty_gradient}, and the corresponding formula for $D\rho_\infty$ follows by the chain rule.
\end{enumerate}
\end{proof}

\begin{remark}
The log-convexity of $\rho_\infty$ in~\cref{prop:rho_infty_p}(i) reflects the geometric structure of the problem.  The minimization formula in~\cref{prop:rho_infty_p}(iii) is often more useful than a coordinatewise expression.  For instance, for
\[
D_1=\operatorname{diag}(2,1),\qquad D_2=\operatorname{diag}(1,2),
\]
one has $\max_j\min_i |\lambda_j^{(i)}|=1$, whereas
\[
\check\rho(\mathcal S)=\min_{p\in[0,1]}\max(2^p,2^{1-p})=\sqrt 2.
\]
Thus the LSR is obtained by balancing the competing coordinates rather than by minimizing each coordinate separately.  The extremal bounds~\eqref{eq:JSR_LSR_bound} complement the tail bounds in~\cref{cor:tail_JSR_LSR}: the LDP quantifies the probability of observing $\rho_n$ near the JSR/LSR for a fixed $\boldsymbol p$, whereas~\eqref{eq:JSR_LSR_bound} shows how these extremes arise by optimizing over the sampling distribution.
\end{remark}

\section{Perturbation theory } \label{sec:perturbation_theory}

In this section, we study perturbations of matrix families and the effect of small perturbations on the random spectral radius. Let
\[
\cF_\eps := \bigl\{A_1(\eps), \ldots, A_m(\eps) \bigr\},
\]
be obtained by perturbing a family $\mathcal{F}=\{A_1,\dots,A_m\}$ with additive matrices $\Delta^{(i)}$, each normalized in Frobenius norm:
\[
A_i(\eps) := A_i + \eps \Delta^{(i)}, \qquad \bigl\| \Delta^{(i)} \bigr\|_F = 1.
\]
Throughout this section we assume that the unperturbed matrices $A_i$ have simple, nonzero eigenvalues, and that $\eps$ is sufficiently small so that the relevant eigenvalue branches remain well-defined.

For clarity, we carry out the perturbation expansion for a family of only two complex matrices; specifically, throughout \cref{subsec:pert_eig_expansion} we consider
\[
\cF_\eps := \bigl\{A_1(\eps), A_2(\eps)\bigr\}.
\]
Since the perturbed matrices $A_1(\eps)$ and $A_2(\eps)$ need not commute, their product expansions exhibit noncommutativity starting at order $\eps$. Specifically, we get
\[
[A_1(\eps), A_2(\eps)] = A_1 A_2 - A_2 A_1 + \eps \big( [A_1, \Delta^{(2)}] + [\Delta^{(1)}, A_2] \big) + \eps^2 [\Delta^{(1)}, \Delta^{(2)}].
\]
In the following, we focus first on triangular matrices, and then specialize further to diagonal ones. For diagonal matrices, the constant term of the commutator vanishes, while in the upper(lower)-triangular case, it has zero diagonal entries.

\subsection{Perturbation eigenvalue expansion}\label{subsec:pert_eig_expansion}

In this subsection, we summarize auxiliary results on eigenvalues of perturbed matrices that are already well-known in the literature.

\begin{lemma} \label{lem:RS-diagonal}
Let $A(\eps):=A+\eps\Delta$, where $A\in\mathbb C^{d\times d}$ has simple, nonzero eigenvalues $\lambda_1,\dots,\lambda_d$. Let $v_k$ and $u_k$ be, respectively, right and left eigenvectors of $A$, normalized so that
\[
u_j^\ast v_k = \delta_{jk}, \qquad j,k=1,\ldots,d.
\]
Then, as $\eps \to 0$, the eigenvalues of $A(\eps)$ admit the expansion
\[
\lambda_k(\eps) = \lambda_k +\eps\,u_k^*\Delta v_k +\eps^2\,u_k^*\Delta v_k^{(1)} +O(\eps^3), \qquad \text{where} \quad v_k^{(1)} = \sum_{j\ne k}\frac{u_j^*\Delta v_k}{\lambda_k-\lambda_j}\,v_j.
\]
\end{lemma}

\begin{proof}
This is the standard Rayleigh--Schr\"odinger expansion for simple eigenvalues in the non-Hermitian setting, see, for example,~\cite{bamieh2022tutorialmatrixperturbationtheory,kato1995perturbation,horn2012matrix}.
\end{proof}

\begin{remark}
In particular, if $A$ is a diagonal matrix, then the eigenvectors are the standard unit vectors $u_k=v_k=e_k$. The eigenvalue expansion simplifies to
\[
\lambda_k(\eps) = \lambda_k + \eps\,\Delta_{kk} + \eps^2 \sum_{j\neq k} \frac{\Delta_{kj}\Delta_{jk}}{\lambda_k - \lambda_j} + O(\eps^3), \qquad \eps \rightarrow 0.
\]
\end{remark}

\begin{remark}
If $A$ is normal, its left and right eigenvectors can be chosen to coincide into an orthonormal basis, and the first-order correction reduces to $\lambda_k^{(1)} = v_k^*\Delta v_k$; if $\Delta$ is Hermitian, this is real.
\end{remark}

\begin{corollary}\label{cor:log-modulus}
With the notation of~\cref{lem:RS-diagonal}, let $a_k:= u_k^* \Delta v_k$ and $b_k:= u_k^* \Delta v_k^{(1)}$. Then
\[
\log|\lambda_k(\eps)|=\log|\lambda_k(0)|+\eps \Re\left(\frac{a_k}{\lambda_k(0)}\right) +\eps^2 \Re \left(\frac{b_k}{\lambda_k(0)}-\frac{a_k^2}{2\lambda_k^2(0)}\right)+O(\eps^3).
\]
\end{corollary}
 
Fix $n \geq 1$ and consider $A_j(\eps) = A_j + \eps\Delta^{(j)}$ for $j=1,2$. We consider products containing exactly $k$ factors of $A_1(\eps)$ and $n-k$ factors of $A_2(\eps)$. Since these operators generally do not commute, we index products by the order of their factors and keep the global convention from~\cref{sec:expected_spectral_radius}:
\begin{equation} \label{def:ProducA_eps_sigma}
\Pi(\eps, \sigma) := A_{\sigma_n}(\eps)\cdots A_{\sigma_1}(\eps),
\qquad
\sigma=(\sigma_1,\ldots,\sigma_n)\in [2]^n,
\end{equation}
and require that $\sigma$ contain exactly $k$ entries equal to $1$ and $n-k$ entries equal to $2$.

\begin{definition}
Let $\sigma\in[2]^n$ be such that $\Pi(0,\sigma)$ has simple, nonzero eigenvalues. We denote by $\Lambda_i(\eps,\sigma)$ the corresponding local eigenvalue branches of $\Pi(\eps,\sigma)$, numbered so that
\[
\Lambda_i(0,\sigma)=\prod_{k=1}^n \lambda_i^{(\sigma_k)}.
\]
The same notation will be used below for words $\sigma\in[m]^n$ in the diagonal perturbation setting.
\end{definition}

At $\eps = 0$, the matrices in a product~\eqref{def:ProducA_eps_sigma} are triangular and therefore the dominant eigenvalue (that is, the eigenvalue of largest modulus) of a product does not depend on the multiplication order. Since $\lambda_i^{(h)} = (A_h)_{ii}\neq0$ for $h=1,2$, we have
\[
\Lambda_i(0,\sigma)=\left(\lambda_i^{(1)}\right)^k \left(\lambda_i^{(2)}\right)^{n-k}.
\]
Denote $\Pi_\sigma^{(0)}:=\Pi(0, \sigma)$. For $|\eps|$ small, we expand $\Pi(\eps, \sigma)$ as
\[
\Pi(\eps, \sigma) = \Pi_\sigma^{(0)} + \eps \, \Pi_\sigma^{(1)} + \eps^2 \, \Pi_\sigma^{(2)} + O(\eps^3),
\]
with
\begin{align*}
\Pi_\sigma^{(1)} & = \sum_{\ell=1}^n A_{\sigma_n}\cdots A_{\sigma_{\ell+1}} \Delta^{(\sigma_\ell)} A_{\sigma_{\ell-1}}\cdots A_{\sigma_1},
\\[1pt]
\Pi_\sigma^{(2)} &= \sum_{1\le p<q\le n} A_{\sigma_n}\cdots A_{\sigma_{q+1}} \Delta^{(\sigma_q)} A_{\sigma_{q-1}}\cdots A_{\sigma_{p+1}} \Delta^{(\sigma_p)} A_{\sigma_{p-1}}\cdots A_{\sigma_1}.
\end{align*}
For $1\le a\le b\le n$ and $\sigma \in [2]^n$, we introduce the compact notation
\[
\Lambda_i^{(a:b)}(0,\sigma) := \prod_{p=a}^b \lambda_i^{(\sigma_p)},
\qquad
\Lambda_i^{(a:b)}(0,\sigma):=1 \ \text{if } a>b,
\]
so that in particular $\Lambda_i^{(1:n)}(0,\sigma)=\Lambda_i(0,\sigma)$.

\begin{remark}
If every factor $A_{\sigma_\ell}(\eps)$ is invertible (e.g., for $|\eps|$ small enough), then for the one-step cyclic shift
\[
\tau(\sigma_1,\sigma_2,\ldots,\sigma_n):=(\sigma_2,\ldots,\sigma_n,\sigma_1)
\]
the products $\Pi(\eps,\tau\sigma)$ and $\Pi(\eps,\sigma)$ are similar:
\[
\Pi(\eps,\tau\sigma)=A_{\sigma_1}(\eps)\,\Pi(\eps,\sigma)\,A_{\sigma_1}(\eps)^{-1}.
\]
Hence $\Lambda_i(\eps,\tau\sigma)=\Lambda_i(\eps,\sigma)$ for all $\eps$. More generally, a cyclic shift by $r$ positions is a conjugation by the partial product $A_{\sigma_r}(\eps)\cdots A_{\sigma_1}(\eps)$.
\end{remark}

\begin{lemma}\label{lemma:eigenvalue_products}
Fix $n\ge 1$ and $\sigma\in[2]^n$, and assume that $\Pi_\sigma^{(0)}=A_{\sigma_n}\cdots A_{\sigma_1}$ has simple, nonzero eigenvalues $\Lambda_1(0,\sigma),\dots,\Lambda_d(0,\sigma)$. Let $v_i(\sigma)$ and $u_i(\sigma)$ be, respectively, right and left eigenvectors of $\Pi_\sigma^{(0)}$ associated with $\Lambda_i(0,\sigma)$, normalized. Then
\begin{align}\label{eq:lambda_pert_product}
\Lambda_i(\eps,\sigma) = \Lambda_i(0,\sigma) &+ \eps\, u_i^*\Pi_\sigma^{(1)} v_i \nonumber\\
&+ \eps^2 \left[ u_i^*\Pi_\sigma^{(2)} v_i + \sum_{j\ne i}\frac{\bigl(u_i^*\Pi_\sigma^{(1)}v_j\bigr)\bigl(u_j^*\Pi_\sigma^{(1)}v_i\bigr)}{\Lambda_i(0,\sigma)-\Lambda_j(0,\sigma)} \right] + O(\eps^3),
\end{align}
where the $O(\eps^3)$ remainder is bounded uniformly in $\sigma$ for fixed $n$.
\end{lemma}

\begin{remark}
The eigenvectors $v_i(\sigma), u_i(\sigma)$ of $\Pi_\sigma^{(0)}$ depend on $\sigma$ and, in general, do not coincide with the standard basis vectors, even though each individual factor $A_{\sigma_\ell}$ has $e_1$ as the right eigenvector for its dominant eigenvalue (upper triangular case). Consequently, the quadratic forms $u_i^*\Pi_\sigma^{(k)}v_i$ in~\eqref{eq:lambda_pert_product} do not reduce to matrix entries $(\Pi_\sigma^{(k)})_{ii}$ unless an additional structural assumption, such as diagonality of the $A_{\sigma_\ell}$, is imposed.
\end{remark}

We now specialize the expansion result above (\cref{lemma:eigenvalue_products}) to the diagonal case, where the formulas collapse to explicit matrix entries.

\begin{lemma}\label{lemma:eigenvalue_products_diag}
Assume that in~\cref{lemma:eigenvalue_products} the matrices $A_{\sigma_\ell}$ are diagonal, so that $\Pi_\sigma^{(0)}$ is diagonal with simple, nonzero entries. Then $u_i(\sigma)=v_i(\sigma)=e_i$ for all $i$ and $\sigma$, and
\[
\Lambda_i(\eps,\sigma) = \Lambda_i(0,\sigma) + \eps\,(\Pi_\sigma^{(1)})_{ii} + \eps^2\left[(\Pi_\sigma^{(2)})_{ii} + \sum_{j\ne i}\frac{(\Pi_\sigma^{(1)})_{ji}(\Pi_\sigma^{(1)})_{ij}}{\Lambda_i(0,\sigma)-\Lambda_j(0,\sigma)}\right] + O(\eps^3),
\]
with the explicit expressions
\begin{align*}
(\Pi_\sigma^{(1)})_{ii} &= \Lambda_i(0,\sigma)\sum_{\ell=1}^n \frac{\Delta^{(\sigma_\ell)}_{ii}}{\lambda_i^{(\sigma_\ell)}}, \\
(\Pi_\sigma^{(1)})_{ij} &= \sum_{\ell=1}^n \Delta^{(\sigma_\ell)}_{ij}\,\Lambda_i^{(\ell+1:n)}(0,\sigma)\,\Lambda_j^{(1:\ell-1)}(0,\sigma), \qquad (i\ne j), \\
(\Pi_\sigma^{(2)})_{ii} &= \sum_{1\le p<q\le n}\sum_{r=1}^d \Delta^{(\sigma_q)}_{ir}\Delta^{(\sigma_p)}_{ri}\,\Lambda_i^{(q+1:n)}(0,\sigma)\,\Lambda_r^{(p+1:q-1)}(0,\sigma)\,\Lambda_i^{(1:p-1)}(0,\sigma).
\end{align*}
\end{lemma}

If $n\ge1$, $k\in\{0,\ldots,n\}$ and $\sigma$ has exactly $k$ entries equal to $1$ and $n-k$ ones equal to $2$, then both the leading term 
\[
\Lambda_i(0,\sigma)=\left(\lambda_i^{(1)}\right)^k\left(\lambda_i^{(2)}\right)^{n-k}
\]
and the first-order perturbation 
\[
\Lambda_i^{(1)}(0,\sigma)=(\Pi_\sigma^{(1)})_{ii} =\Lambda_i(0,\sigma)\left( k\,\frac{\Delta^{(1)}_{ii}}{\lambda_i^{(1)}}+(n-k)\,\frac{\Delta^{(2)}_{ii}}{\lambda_i^{(2)}} \right)
\]
\emph{do not depend on the multiplication order in the product but depend only on the occurrence numbers}. The second-order coefficient $\Lambda_i^{(2)}(0,\sigma)$ generally \textbf{depends} on the arrangement of factors. It becomes independent if either
\begin{itemize}
\item[\textnormal{(a)}] $\Delta^{(1)}$ and $\Delta^{(2)}$ are diagonal (all factors commute);
\item[\textnormal{(b)}] $D_i=\alpha_i I$ for $i\in\{1,2\}$, and $[\Delta^{(1)},\Delta^{(2)}]=0$.
\end{itemize}

\begin{example}
Let $d=2$ and consider the following matrices:
\[
D_1=\begin{pmatrix}1&0\\0&2\end{pmatrix},\qquad
D_2=\begin{pmatrix}3&0\\0&5\end{pmatrix},\qquad
\Delta^{(1)}=\Delta^{(2)}=\begin{pmatrix}0&1\\1&0\end{pmatrix}.
\]
With the global convention $\Pi(\eps,\sigma)=A_{\sigma_n}(\eps)\cdots A_{\sigma_1}(\eps)$, we have
\[
\Pi(\eps,(1,1,2,2))=A_2(\eps)^2A_1(\eps)^2,
\qquad
\Pi(\eps,(1,2,1,2))=A_2(\eps)A_1(\eps)A_2(\eps)A_1(\eps).
\]
For the eigenvalue associated with $e_1$, $\Lambda_1(0,\sigma)=9$ and $\Lambda_1^{(1)}(0,\sigma)=0$ for both permutations. However, the second-order terms are different:
\[
\Lambda_1^{(2)} \bigl(0,(1,1,2,2) \bigr) = -\tfrac{1803}{91} \neq -\tfrac{138}{7} = \Lambda_1^{(2)}\bigl(0,(1,2,1,2)\bigr),
\]
so the ordering is crucial even though the perturbations are symmetric and off-diagonal.
\end{example}

\begin{corollary}\label{cor:log-eig-product}
Under the hypotheses of~\cref{lemma:eigenvalue_products_diag}, for any $i$ and any word $\sigma$,
\begin{align*}
& \log |\Lambda_i(\eps,\sigma)|  = \log |\Lambda_i(0,\sigma)| +\eps\,\Re \left(\frac{(\Pi_\sigma^{(1)})_{ii}}{\Lambda_i(0,\sigma)}\right)
\\ &\qquad  +\eps^2\,\Re \left( \frac{1}{\Lambda_i(0,\sigma)} \left[ (\Pi_\sigma^{(2)})_{ii} + \sum_{j\ne i} \frac{(\Pi_\sigma^{(1)})_{ji}(\Pi_\sigma^{(1)})_{ij}} {\Lambda_i(0,\sigma)-\Lambda_j(0,\sigma)} \right] -\frac12\left(\frac{(\Pi_\sigma^{(1)})_{ii}}{\Lambda_i(0,\sigma)}\right)^2 \right) +O(\eps^3).
\end{align*}
\end{corollary}

\subsection{Perturbation theory for diagonal matrices}

In the proofs of the LLN and CLT for triangular and commuting families of matrices, we used independent random variables associated with the matrices. This approach depended on two properties: the preservation of matrix structure under multiplication and the ability to determine the spectral radius of a product directly from the original data. However, perturbations typically disrupt these structures, causing eigenvalue perturbations in matrix products that depend on the multiplication order. Therefore, in the following subsections we focus on diagonal matrix families. Our first goal is to derive a rigorous first-order logarithmic approximation for the perturbed random spectral radius. To recover an exact fixed-$\eps$ law of large numbers and central limit theorem, we then impose the stronger assumption that the perturbed family remains simultaneously triangularizable. Let
\[
\mathcal F_\eps = \bigl\{A_i(\eps):=D_i+\eps\,\Delta^{(i)} \bigr\}_{i=1}^m, \qquad \|\Delta^{(i)}\|_F=1,
\]
where $D_i=\operatorname{diag}(\lambda^{(i)}_1,\dots,\lambda^{(i)}_d)$, for
$i=1,\dots,m$. In this subsection we assume that the diagonal family satisfies the aligned dominance assumption from~\cref{subsec:ordered_diagonal}, namely that, after relabeling if necessary,
\[
|\lambda_1^{(i)}|>\max_{r\ne1}|\lambda_r^{(i)}|,\qquad i=1,\dots,m.
\]
Let the perturbed RSR be the random variable on the word space $[m]^n$ defined by
\[
\rho_n(\eps,\sigma) := \left(\max_{1\le i\le d}|\Lambda_i(\eps,\sigma)|\right)^{1/n}, \qquad \sigma=(i_1,\dots,i_n)\in[m]^n,
\]
and equip $[m]^n$ with the usual product probability measure $\mathbb P(\sigma)=\prod_{j=1}^n p_{i_j}$.

\subsubsection{Logarithmic approximation}

The perturbative expressions for the eigenvalues~\eqref{eq:lambda_pert_product} show that the branches of matrix products vary continuously with $\eps$. In the diagonal case,~\eqref{def:aligned dom convergence} gives a canonical branch issuing from the dominant unperturbed eigenvalue. The next lemma provides a uniform logarithmic approximation for that branch.

\begin{lemma}\label{lem:pert_remainder}
Let $\mathcal{F}_\eps = \{D_i+\eps\Delta^{(i)}\}_{i=1}^m$ where the diagonal matrices $D_i=\operatorname{diag}(\lambda_1^{(i)},\ldots,\lambda_d^{(i)})$ satisfy the aligned dominance assumption. Set
\begin{equation}\label{def:eta}
\eta := \max_{\substack{1<r\leq d \\ 1\leq i\leq m}} \frac{|\lambda_r^{(i)}|}{|\lambda_1^{(i)}|} < 1, \qquad \lambda_{\min} := \min_i|\lambda_1^{(i)}|, \qquad \Delta_{\max} := \max_i\|\Delta^{(i)}\|_F,
\end{equation}
and $\delta := \Delta_{\max}/\lambda_{\min}$. Then there exists $\eps_0>0$, depending only on $\eta$, $\lambda_{\min}$, and $\Delta_{\max}$, such that for every $|\eps|\leq \eps_0$, every $n\geq 1$, and every $\sigma\in[m]^n$, the eigenvalue branch $\Lambda_1(\eps,\sigma)$ issued from $\Lambda_1(0,\sigma)$ is simple, nonzero, and strictly dominant in modulus. Moreover,
\[
\frac{1}{n}\log|\Lambda_1(\eps,\sigma)| = \frac{1}{n}\sum_{k=1}^n \left(\log|\lambda_1^{(\sigma_k)}| + \eps\,\Re\left(\frac{\Delta^{(\sigma_k)}_{11}}{\lambda_1^{(\sigma_k)}}\right)\right) + R_n(\eps,\sigma),
\]
with the following estimate uniform in $n$ and $\sigma$:
\begin{equation}\label{eq:remainder_bound}
|R_n(\eps,\sigma)|\leq C_{\mathcal{F}}\,\eps^2, \qquad C_{\mathcal{F}} := \frac{6\,\Delta_{\max}^2}{\lambda_{\min}^2(1-\eta)}.
\end{equation}
\end{lemma}

\begin{proof}
Set
\[
B_i(\eps):=\frac{A_i(\eps)}{\lambda_1^{(i)}}=
\operatorname{diag}(1,\nu_2^{(i)},\dots,\nu_d^{(i)})+\eps M_i,
\qquad
\nu_r^{(i)}:=\frac{\lambda_r^{(i)}}{\lambda_1^{(i)}},
\]
so that $|\nu_r^{(i)}|\le \eta$ for $r>1$ and $\|M_i\|_2\le \delta$. Let $
\widetilde\Pi(\eps,\sigma):=B_{\sigma_n}(\eps)\cdots B_{\sigma_1}(\eps)$. Then
\[
\Lambda_1(\eps,\sigma)=\Lambda_1(0,\sigma)\,\widetilde\Lambda_1(\eps,\sigma),
\]
where $\widetilde\Lambda_1(\eps,\sigma)$ denotes the branch of
$\widetilde\Pi(\eps,\sigma)$ issuing from $1$ at $\eps=0$. Write $B_i(\eps)$ in block form relative to $\mathbb C e_1\oplus e_1^\perp$, that is
\[
B_i(\eps)=
\begin{pmatrix}
a_i & b_i^\ast\\
c_i & D_i^\perp
\end{pmatrix},
\]
with
\[
|a_i-1|\le |\eps|\delta,\qquad
\|b_i\|_2,\|c_i\|_2\le |\eps|\delta,\qquad
\|D_i^\perp\|_2\le \eta+|\eps|\delta.
\]
For $w\in e_1^\perp$, define $\Phi_i(w)$ by
\[
B_i(\eps)(e_1+w)=\bigl(a_i+b_i^\ast w\bigr)\,\bigl(e_1+\Phi_i(w)\bigr),
\qquad
\Phi_i(w):=\frac{c_i+D_i^\perp w}{a_i+b_i^\ast w}.
\]
Fix
\[
\beta:=\frac{2\delta|\eps|}{1-\eta},
\qquad
\mathbb B_\beta:=\{w\in e_1^\perp:\ \|w\|_2\le \beta\}.
\]
After choosing $\eps_0>0$ small enough, we may assume that for all
$|\eps|\le \eps_0$ one has $\beta\le \sqrt2-1$ and
$|\eps|\delta(1+\beta)\le 1/2$. Hence, for every $w\in\mathbb B_\beta$,
\[
|a_i+b_i^\ast w|
\ge 1-|\eps|\delta(1+\beta)
>0,
\]
and
\[
\|\Phi_i(w)\|_2
\le
\frac{(\eta+|\eps|\delta)\beta+|\eps|\delta}{1-|\eps|\delta(1+\beta)}
\le \beta.
\]
Thus each $\Phi_i$ maps $\mathbb B_\beta$ into itself. A direct computation gives,
for $w,z\in\mathbb B_\beta$,
\[
\|\Phi_i(w)-\Phi_i(z)\|_2
\le q(\eps)\,\|w-z\|_2,
\]
where
\[
q(\eps):=
\frac{(1+|\eps|\delta)(\eta+|\eps|\delta)+|\eps|^2\delta^2
+2\beta|\eps|\delta(\eta+|\eps|\delta)}
{\bigl(1-|\eps|\delta(1+\beta)\bigr)^2}.
\]
Since $q(\eps)\to \eta<1$ as $\eps\to 0$, after shrinking $\eps_0$ if necessary we may assume
\[
q_*:=\sup_{|\eps|\le \eps_0}q(\eps)<1.
\]
For a word $\sigma\in[m]^n$, define the projective map
\[
\Phi_\sigma:=\Phi_{\sigma_n}\circ\cdots\circ\Phi_{\sigma_1}.
\]
Then $\Phi_\sigma$ is a contraction on $\mathbb B_\beta$ with Lipschitz constant at most
$q_*^n$. By Banach's fixed-point theorem, $\Phi_\sigma$ has a unique fixed point
$w^\ast\in\mathbb B_\beta$. Setting $v^\ast:=e_1+w^\ast$, we obtain
\[
\widetilde\Pi(\eps,\sigma)v^\ast=\widetilde\Lambda_1(\eps,\sigma)\,v^\ast
\]
for some eigenvalue $\widetilde\Lambda_1(\eps,\sigma)$ of
$\widetilde\Pi(\eps,\sigma)$. To identify the derivative, let $M:=\widetilde\Pi(\eps,\sigma)$ and $\Lambda:=\widetilde\Lambda_1(\eps,\sigma)$. Differentiating the affine-chart identity
\[
M(e_1+w)=\alpha(w)\bigl(e_1+\Phi_\sigma(w)\bigr)
\]
at $w=w^\ast$ in the direction $\xi\in e_1^\perp$ gives
\[
M\binom{0}{\xi}
=
\alpha'_{w^\ast}[\xi]\,v^\ast
+
\Lambda\binom{0}{D\Phi_\sigma(w^\ast)\xi}.
\]
Now let $u=(u_1,u_\perp)$ be an eigenvector of $M$ with eigenvalue
$\nu\neq \Lambda$, and set
\[
\xi:=u_\perp-u_1w^\ast\in e_1^\perp.
\]
Note that $\xi\neq 0$, since otherwise $u=u_1v^\ast$ would force
$\nu=\Lambda$. Since $u=u_1v^\ast+\binom{0}{\xi}$, using $Mv^\ast=\Lambda v^\ast$ and $Mu=\nu u$ we obtain
\[
M\binom{0}{\xi}
=
M(u-u_1v^\ast)
=
\nu\binom{0}{\xi}+(\nu-\Lambda)u_1v^\ast.
\]
Comparing the $e_1^\perp$-components yields
\[
D\Phi_\sigma(w^\ast)\xi=\frac{\nu}{\Lambda}\,\xi.
\]
Thus $D\Phi_\sigma(w^\ast)$ is the linear map induced by
$\Lambda^{-1}M$ on the quotient $\mathbb C^d/\mathbb C v^\ast$.
In particular, every eigenvalue $\nu\neq \Lambda$ of $M$ contributes the
quotient eigenvalue $\nu/\Lambda$, and if $\Lambda$ had algebraic multiplicity
greater than one then $1$ would also belong to
$\operatorname{spec}(D\Phi_\sigma(w^\ast))$. Since
\[
\|D\Phi_\sigma(w^\ast)\|\le q_*^n<1,
\]
all eigenvalues of $D\Phi_\sigma(w^\ast)$ have modulus $<1$. Hence
$|\nu|<|\Lambda|$ for every eigenvalue $\nu\neq \Lambda$, and
$1\notin \sigma(D\Phi_\sigma(w^\ast))$. This rules out algebraic
multiplicity greater than one for $\Lambda$. Therefore
$\widetilde\Lambda_1(\eps,\sigma)$ is algebraically simple and strictly dominant
in modulus. 

By continuity in $\eps$ and the identity
$\widetilde\Lambda_1(0,\sigma)=1$, it is exactly the branch issuing from $1$. Set $v^{(0)}:=v^\ast$, and for $k=1,\dots,n$ define recursively
\[
B_{\sigma_k}(\eps)v^{(k-1)}=\alpha_k\,v^{(k)},
\qquad
v^{(k)}=e_1+w^{(k)},
\qquad
w^{(k)}\in\mathbb B_\beta.
\]
Since $w^\ast$ is a fixed point of $\Phi_\sigma$, one has
$v^{(n)}=v^{(0)}=v^\ast$. Iterating the previous identities,
\[
\widetilde\Pi(\eps,\sigma)v^\ast
=
B_{\sigma_n}(\eps)\cdots B_{\sigma_1}(\eps)v^\ast
=
\left(\prod_{k=1}^n \alpha_k\right)v^\ast,
\]
hence
\[
\widetilde\Lambda_1(\eps,\sigma)=\prod_{k=1}^n \alpha_k.
\]
Moreover, each factor is nonzero because
\[
|\alpha_k|
=
|a_{\sigma_k}+b_{\sigma_k}^\ast w^{(k-1)}|
\ge 1-|\eps|\delta(1+\beta)>0.
\]
Therefore $\widetilde\Lambda_1(\eps,\sigma)\neq 0$, and hence
$\Lambda_1(\eps,\sigma)=\Lambda_1(0,\sigma)\widetilde\Lambda_1(\eps,\sigma)$ is
simple, nonzero, and strictly dominant. Writing $v^{(k-1)}=e_1+w^{(k-1)}$ with $\|w^{(k-1)}\|_2\le \beta$, we obtain
\[
\alpha_k
=
1+\eps\,\frac{\Delta^{(\sigma_k)}_{11}}{\lambda_1^{(\sigma_k)}}
+\eps\,\zeta_k,
\qquad
|\zeta_k|
\le
\delta\,\|w^{(k-1)}\|_2
\le
\delta\beta
=
\frac{2\delta^2|\eps|}{1-\eta}.
\]
Set $u_k:=\alpha_k-1$. Since $\beta\le 1$ for $|\eps|\le \eps_0$, after one final
shrink of $\eps_0$ we may also assume $2|\eps|\delta\le 1/2$, and therefore
\[
|u_k|
\le
|\eps|\delta(1+\beta)
\le 2|\eps|\delta
\le \tfrac12.
\]
For $|u|\le \frac12$, the Taylor remainder estimate
\[
\bigl|\log|1+u|-\Re u\bigr|\le |u|^2
\]
yields
\[
\left|
\log|\alpha_k|
-\eps\,\Re\!\left(\frac{\Delta^{(\sigma_k)}_{11}}{\lambda_1^{(\sigma_k)}}\right)
\right|
\le
|\eps|\,|\zeta_k|+|u_k|^2
\le
\frac{2\delta^2\eps^2}{1-\eta}+4\delta^2\eps^2
\le
\frac{6\delta^2\eps^2}{1-\eta}.
\]
Summing over $k$ and dividing by $n$, together with
\[
\log|\widetilde\Lambda_1(\eps,\sigma)|
=
\log|\Lambda_1(\eps,\sigma)|
-\sum_{k=1}^n \log|\lambda_1^{(\sigma_k)}|,
\]
gives
\[
\frac{1}{n}\log|\Lambda_1(\eps,\sigma)|
=
\frac{1}{n}\sum_{k=1}^n
\left(
\log|\lambda_1^{(\sigma_k)}|
+\eps\,\Re\!\left(\frac{\Delta^{(\sigma_k)}_{11}}{\lambda_1^{(\sigma_k)}}\right)
\right)
+R_n(\eps,\sigma),
\]
with
\[
|R_n(\eps,\sigma)|
\le
\frac{6\delta^2}{1-\eta}\,\eps^2
=
\frac{6\,\Delta_{\max}^2}{\lambda_{\min}^2(1-\eta)}\,\eps^2.
\]
This is exactly~\eqref{eq:remainder_bound}.
\end{proof}

\begin{remark}
\cref{lem:pert_remainder} provides a uniform first-order logarithmic approximation for the dominant eigenvalue branch of perturbed products. It does not by itself imply an exact fixed-$\eps$ LLN or CLT, since the bounded remainder $R_n(\eps,\sigma)$ need not converge as $n\to\infty$.
\end{remark}

\begin{theorem}\label{thm:pert_log_approx}
Under the hypotheses of~\cref{lem:pert_remainder}, define
\[
\alpha_i:=\Re\!\left(\frac{\Delta^{(i)}_{11}}{\lambda_1^{(i)}}\right),
\qquad
\widetilde\mu_\eps:=\sum_{i=1}^m p_i\bigl(\log|\lambda_1^{(i)}|+\eps\,\alpha_i\bigr),
\qquad
\widetilde\rho_\infty(\eps):=\exp(\widetilde\mu_\eps).
\]
Then, for every $|\eps|\le \eps_0$, every $n\ge1$, and every $\sigma\in[m]^n$,
\begin{equation}\label{eq:pert_log_main}
\log \rho_n(\eps,\sigma)
=
\frac1n\sum_{k=1}^n\bigl(\log|\lambda_1^{(\sigma_k)}|+\eps\,\alpha_{\sigma_k}\bigr)
+R_n(\eps,\sigma),
\end{equation}
where $R_n(\eps,\sigma)$ satisfies~\eqref{eq:remainder_bound}. Now let $\omega=(\omega_k)_{k\ge1}\in[m]^{\mathbb N}$ be sampled with the product measure and write $\omega|_n:=(\omega_1,\dots,\omega_n)$. Then, almost surely,
\[
\limsup_{n\to\infty}
\left|
\log \rho_n(\eps,\omega|_n)-\widetilde\mu_\eps
\right|
\le C_{\mathcal F}\eps^2.
\]
Equivalently, almost surely,
\[
\mathrm e^{-C_{\mathcal F}\eps^2}\widetilde\rho_\infty(\eps) \le \liminf_{n\to\infty}\rho_n(\eps,\omega|_n) \le \limsup_{n\to\infty}\rho_n(\eps,\omega|_n) \le \mathrm e^{C_{\mathcal F}\eps^2}\widetilde\rho_\infty(\eps).
\]
\end{theorem}

\begin{proof}
By~\cref{lem:pert_remainder}, the branch $\Lambda_1(\eps,\sigma)$ is strictly dominant, so
\[
\rho_n(\eps,\sigma)=|\Lambda_1(\eps,\sigma)|^{1/n}
\implies
\log\rho_n(\eps,\sigma)=\frac1n\log|\Lambda_1(\eps,\sigma)|.
\]
Applying the logarithmic expansion from~\cref{lem:pert_remainder} yields~\eqref{eq:pert_log_main}. Now let $\omega=(\omega_k)_{k\ge1}\in[m]^{\mathbb N}$ and apply~\eqref{eq:pert_log_main} to the prefix $\sigma=\omega|_n$. The first term on the right-hand side is the empirical mean of i.i.d.\ random variables with expectation $\widetilde\mu_\eps$, so the strong law of large numbers gives
\[
\frac1n\sum_{k=1}^n
\bigl(\log|\lambda_1^{(\omega_k)}|+\eps\,\alpha_{\omega_k}\bigr)
\to
\widetilde\mu_\eps
\qquad \text{a.s.}
\]
Combining this with the uniform remainder bound~\eqref{eq:remainder_bound} yields the logarithmic enclosure. Taking exponentials gives the multiplicative bounds for $\rho_n(\eps,\omega|_n)$.
\end{proof}

\subsubsection{Numerical validation of the remainder bound}

We illustrate \cref{lem:pert_remainder} on the same diagonal family used in the large-deviation experiment, namely
\[
D_1=\operatorname{diag}(3.5,1.2,0.8),\quad
D_2=\operatorname{diag}(2.0,1.5,0.6),\quad
D_3=\operatorname{diag}(4.0,0.9,1.1),
\]
with probabilities $\boldsymbol p=(0.3,0.5,0.2)$. We draw one fixed Gaussian
realization of the perturbation matrices $\Delta^{(i)}$ and normalize each matrix so that $\|\Delta^{(i)}\|_F=1$, hence
$\Delta_{\max}=1$. For this family,
\[
\eta=0.75,\qquad
\lambda_{\min}=2.0,\qquad
\delta=\frac{\Delta_{\max}}{\lambda_{\min}}=0.5,\qquad
C_{\mathcal F}=\frac{6\,\Delta_{\max}^2}{\lambda_{\min}^2(1-\eta)}=6.
\]
Moreover, the inequalities used in the proof of~\cref{lem:pert_remainder} are satisfied with $|\eps|\le 0.09$. Indeed
$\beta=4|\eps|\le 0.36<\sqrt2-1$,
$|\eps|\delta(1+\beta)\le 0.0612<\tfrac12$, and the contraction factor
$q(\eps)$ appearing in the proof satisfies $q(0.09)\approx 0.974<1$.
Accordingly, we restrict the numerical validation to the theorem-covered regime
$\eps\in[0.002,0.09]$.

For each pair $(n,\eps)$ with $n\in\{10,50,200,500\}$, we generate
$3\,000$ i.i.d.\ words $\sigma\in[m]^n$ and evaluate
\[
R_n(\eps,\sigma)
:=
\log \rho_n(\eps,\sigma)
-
\frac1n\sum_{k=1}^n
\Bigl(\log|\lambda_1^{(\sigma_k)}|+\eps\,\alpha_{\sigma_k}\Bigr),
\qquad
\alpha_i=\Re\!\left(\frac{\Delta^{(i)}_{11}}{\lambda_1^{(i)}}\right).
\]
The exact value of $\log \rho_n(\eps,\sigma)$ is computed from the spectral
radius of the perturbed product using a stabilized product accumulation, which
avoids floating-point overflow for the largest values of $n$.
\cref{fig:remainder-bound} displays the resulting sample maxima.

The left panel shows that the empirical maxima
$\max_{1\le s\le 3000}|R_n(\eps,\sigma^{(s)})|$ are parallel to the theoretical
envelope $C_{\mathcal F}\eps^2$ on a log--log scale and remain well below it
for all tested values of $n$. The right panel plots the normalized sample
maxima $\max_{1\le s\le 3000}|R_n(\eps,\sigma^{(s)})|/\eps^2$; these curves are
nearly flat in~$\eps$ and show no growth with~$n$, which is consistent with the
uniform $O(\eps^2)$ estimate in \cref{lem:pert_remainder}. As expected, the
explicit constant $C_{\mathcal F}=6$ is conservative relative to the observations.

\begin{figure}[ht!]
  \centering
  \includegraphics[width=\textwidth]{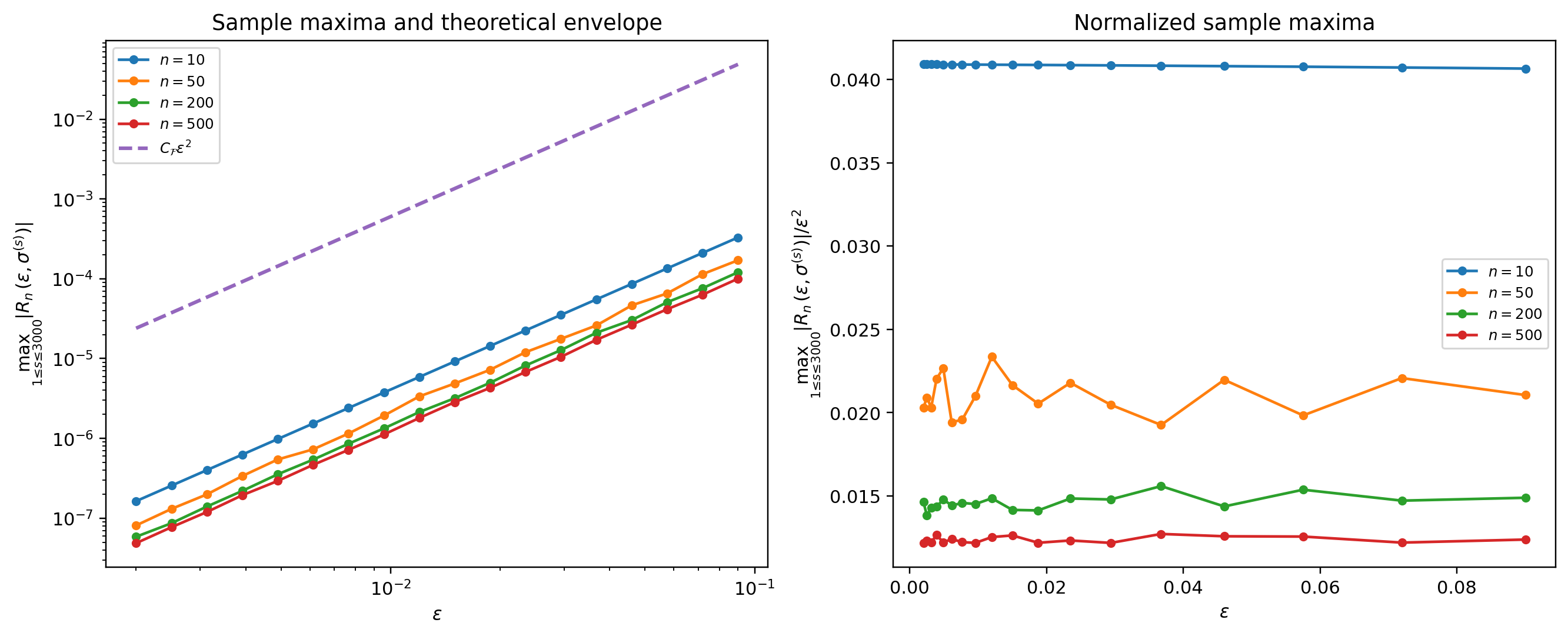}
  \caption{\textbf{Left:} sample maxima
  $\max_{1\le s\le 3000}|R_n(\eps,\sigma^{(s)})|$ for
  $n\in\{10,50,200,500\}$, together with the theoretical reference
  $C_{\mathcal F}\eps^2$ with $C_{\mathcal F}=6$. \textbf{Right:} normalized
  sample maxima $\max_{1\le s\le 3000}|R_n(\eps,\sigma^{(s)})|/\eps^2$. The
  nearly flat dependence on~$\eps$ supports the $O(\eps^2)$ scaling.}
  \label{fig:remainder-bound}
\end{figure}

\subsection{Perturbed LLN and CLT}

The first-order approximation above is useful for diagonal perturbations, but it does not by itself provide an exact fixed-$\eps$ LLN/CLT. To recover such results, we now impose a stronger structural hypothesis ensuring that the perturbed family remains within the triangular framework of \cref{sec:expected_spectral_radius}.

\begin{assumption}[Simultaneous triangularizability]\label{ass:simtri}
There exists $\eps_0>0$ and a family of invertible matrices $S(\eps)\in\mathbb C^{d\times d}$, defined for $|\eps|\le \eps_0$, such that the matrices
\[
T_i(\eps):=S(\eps)^{-1}A_i(\eps)S(\eps),
\]
are upper triangular for all $i=1,\ldots,m$. We denote their diagonal entries by
\[
\lambda_j^{(i)}(\eps):=\bigl(T_i(\eps)\bigr)_{jj},\qquad j=1,\dots,d,
\]
ordered so that $\lambda_j^{(i)}(0)=\lambda_j^{(i)}$. We also assume that $\lambda_j^{(i)}(\eps)\neq0$ for all $i,j$ and all $|\eps|\le \eps_0$. Since the diagonal entries of an upper-triangular matrix are its eigenvalues and the eigenvalues of $A_i(0)$ are simple, these functions $\eps\mapsto \lambda_j^{(i)}(\eps)$ agree, after the chosen relabeling, with the analytic eigenvalue branches of $A_i(\eps)$ issuing from $\lambda_j^{(i)}$.
\end{assumption}

\begin{remark}
Since spectral radius is invariant under similarity transformations, for every word
$\sigma=(i_1,\dots,i_n)\in[m]^n$ one has
\[
\rho \bigl(A_{i_n}(\eps)\cdots A_{i_1}(\eps)\bigr) = \rho\bigl(T_{i_n}(\eps)\cdots T_{i_1}(\eps)\bigr).
\]
Therefore, the random spectral radius associated with the perturbed family $\mathcal F_\eps=\{A_i(\eps)\}_{i=1}^m$ coincides with that of the triangular family $\mathcal T_\eps=\{T_i(\eps)\}_{i=1}^m$.
\end{remark}

For the perturbed triangular family, define
\[
\mu_j(\eps):=\sum_{i=1}^m p_i\log |\lambda_j^{(i)}(\eps)|, \qquad \rho_\infty^{(j)}(\eps):=\exp(\mu_j(\eps)) = \prod_{i=1}^m |\lambda_j^{(i)}(\eps)|^{p_i},
\]
and $\rho_\infty(\eps):=\max_{1\le j\le d}\rho_\infty^{(j)}(\eps)$. In addition, let
\[
J(\eps):=\arg\max_{1\le j\le d}\rho_\infty^{(j)}(\eps).
\]
Moreover, for $l,h\in J(\eps)$ define the covariance matrix
\[
\Sigma_\eps(l,h)
:=
\sum_{i=1}^m p_i \log |\lambda_l^{(i)}(\eps)|\,\log |\lambda_h^{(i)}(\eps)|
-
\mu_l(\eps)\mu_h(\eps).
\]

\begin{theorem}\label{thm:pert_strong}
Assume~\cref{ass:simtri}. Then, for every fixed $|\eps|\le \eps_0$, there holds
\[
\rho_n(\mathcal F_\eps,\mathbb P_{\mathcal F_\eps}) \xrightarrow{\ \mathrm{a.s.}\ } \rho_\infty(\eps) \qquad\text{as }n\to\infty.
\]
More precisely,
\[
\rho_\infty(\eps)=\max_{1\le j\le d}\prod_{i=1}^m |\lambda_j^{(i)}(\eps)|^{p_i}.
\]
If $J(\eps)=\{j^\star(\eps)\}$, define
\[
\sigma_\infty(\eps) = \rho_\infty(\eps) \left( \sum_{i=1}^m p_i\bigl(\log |\lambda_{j^\star(\eps)}^{(i)}(\eps)|\bigr)^2 - \mu_{j^\star(\eps)}(\eps)^2 \right)^{1/2}.
\]
If in addition $\sigma_\infty(\eps)>0$, then
\[
\frac{
\rho_n(\mathcal F_\eps,\mathbb P_{\mathcal F_\eps})-\rho_\infty(\eps)}{n^{-1/2}\sigma_\infty(\eps)} \ \xRightarrow{\ d\ }\ \mathcal N(0,1).
\]
If $J(\eps)=\{j^\star(\eps)\}$ and $\sigma_\infty(\eps)=0$, then the fluctuation is degenerate:
\[
\sqrt n\Bigl(\rho_n(\mathcal F_\eps,\mathbb P_{\mathcal F_\eps})-\rho_\infty(\eps)\Bigr)\xRightarrow{\ d\ } 0.
\]
If $|J(\eps)|>1$, then
\[
\sqrt n\Bigl(\rho_n(\mathcal F_\eps,\mathbb P_{\mathcal F_\eps})-\rho_\infty(\eps)\Bigr) \ \xRightarrow{\ d\ }\ \rho_\infty(\eps)\max_{j\in J(\eps)} G_j,
\]
where $\boldsymbol G=(G_j)_{j\in J(\eps)}$ is a centered Gaussian vector with covariance matrix $\Sigma_\eps$.
\end{theorem}

\begin{proof}
By~\cref{ass:simtri}, the perturbed family $\mathcal F_\eps$ is similar to the upper-triangular family $\mathcal T_\eps=\{T_i(\eps)\}_{i=1}^m$, and the random spectral radii coincide. The stated LLN and CLT therefore follow directly from~\cref{thm:LLN gen diag,thm:clt_general,Cor:genericCLT} applied to $\mathcal T_\eps$.
\end{proof}

We now show how the perturbation expansions from~\cref{lem:RS-diagonal,cor:log-modulus} translate into expansions of the asymptotic parameters $\rho_\infty(\eps)$ and $\sigma_\infty(\eps)$.

\begin{assumption}\label{ass:persist_unique}
Assume that there exists an index $j^\star\in\{1,\dots,d\}$ such that
\[
\mu_{j^\star}(0)>\max_{j\ne j^\star}\mu_j(0).
\]
By continuity of the functions $\eps\mapsto \mu_j(\eps)$, the same maximizer $j^\star$ persists for all sufficiently small $|\eps|$, and we keep this label fixed in the sequel.
\end{assumption}

\begin{corollary}\label{cor:rhoinf_eps}
Assume~\cref{ass:simtri} and~\cref{ass:persist_unique}. Then, for $|\eps|$ sufficiently small,
\[
\log \rho_\infty(\eps) = \sum_{i=1}^m p_i \log |\lambda_{j^\star}^{(i)}(\eps)| =
\log \rho_\infty(0) + \eps \sum_{i=1}^m p_i\,\Re \left(\frac{a_{i,j^\star}}{\lambda_{j^\star}^{(i)}(0)}\right) + O(\eps^2),
\]
where $a_{i,j^\star}:=u_{i,j^\star}^*\,\Delta^{(i)}\,v_{i,j^\star}$, with $u_{i,j^\star},v_{i,j^\star}$ denoting normalized left/right eigenvectors of $A_i$
associated with $\lambda_{j^\star}^{(i)}(0)$. Equivalently,
\[
\rho_\infty(\eps)= \rho_\infty(0)\,
\exp \left( \eps \sum_{i=1}^m p_i\,\Re \left(\frac{a_{i,j^\star}}{\lambda_{j^\star}^{(i)}(0)}\right) +O(\eps^2) \right).
\]
In the diagonal case, where $u_{i,j^\star}=v_{i,j^\star}=e_{j^\star}$, this reduces to
\[
\log \rho_\infty(\eps) = \log \rho_\infty(0) + \eps \sum_{i=1}^m p_i\,\Re\!\left(\frac{\Delta^{(i)}_{j^\star j^\star}}{\lambda_{j^\star}^{(i)}}\right) +O(\eps^2).
\]
\end{corollary}

\begin{proof}
By~\cref{ass:persist_unique}, continuity implies that for all sufficiently small $|\eps|$ the maximum defining $\rho_\infty(\eps)$ is attained uniquely at $j^\star$, hence
\[
\log \rho_\infty(\eps)=\sum_{i=1}^m p_i \log |\lambda_{j^\star}^{(i)}(\eps)|.
\]
Applying~\cref{cor:log-modulus} to each $A_i(\eps)$ and summing with weights $p_i$
yields the expansion.
\end{proof}

\begin{corollary}\label{cor:sigmainf_eps}
Assume~\cref{ass:simtri} and~\cref{ass:persist_unique}. Then, for $|\eps|$ sufficiently small,
\[
\sigma_\infty^2(\eps) = \rho_\infty(\eps)^2 \operatorname{Var}_{\boldsymbol p}\bigl(\log |\lambda_{j^\star}^{(i)}(\eps)|\bigr),
\]
and, in particular, $\sigma_\infty(\eps)=\sigma_\infty(0)+O(\eps)$.
\end{corollary}

\begin{proof}
The formula for $\sigma_\infty(\eps)$ is exactly the unique-maximizer variance formula from~\cref{thm:pert_strong}. For $i=1,\dots,m$, set
\[
x_i(\eps):=\log |\lambda_{j^\star}^{(i)}(\eps)|,
\]
and let $X_\eps$ be the discrete random variable taking the value $x_i(\eps)$ with probability $p_i$. By~\cref{cor:log-modulus}, each $x_i(\eps)$ admits an expansion of the form
\[
x_i(\eps)=x_i(0)+\eps\,b_i+O(\eps^2).
\]
Hence
\[
\operatorname{Var}(X_\eps)
=
\operatorname{Var}(X_0)+O(\eps)
\]
whenever $\operatorname{Var}(X_0)>0$. If instead $\operatorname{Var}(X_0)=0$, then $x_i(0)$ is constant on the support of $\boldsymbol p$, and therefore
\[
\operatorname{Var}(X_\eps)
=
\eps^2\operatorname{Var}_{\boldsymbol p}(b_i)+O(\eps^3).
\]
In either case,
\[
\operatorname{Var}(X_\eps)^{1/2}
=
\operatorname{Var}(X_0)^{1/2}+O(\eps).
\]
Together with~\cref{cor:rhoinf_eps}, which gives $\rho_\infty(\eps)=\rho_\infty(0)+O(\eps)$, this yields
\[
\sigma_\infty(\eps)=\sigma_\infty(0)+O(\eps).
\]
\end{proof}

\section{Numerical experiments} \label{sec:numerics}

In this section we present empirical evidence for noncommuting families sampled directly from continuous matrix ensembles. In particular, the experiments below do \textbf{not} start from a diagonal or triangular reference family and do \textbf{not} use a perturbative construction. The purpose of the section is therefore purely exploratory: we test whether the LLN/CLT/LDP phenomenology established earlier for commuting and triangular families is still visible in genuinely nonperturbative examples. Fix
\[
\cF=\{A_1,\dots,A_m\}\subseteq \mathbb C^{d\times d}
\]
and a probability vector $\mathbf p=(p_1,\dots,p_m)\in\Delta_m$. For an infinite i.i.d. index sequence
\[
\omega=(\omega_k)_{k\ge1}\in[m]^{\mathbb N},
\qquad
\mathbb P(\omega_k=i)=p_i,
\]
we write $\omega|_n:=(\omega_1,\dots,\omega_n)$ and
\[
\Pi(\omega|_n):=A_{\omega_n}\cdots A_{\omega_1},
\qquad
\rho_n(\omega):=\rho \bigl(\Pi(\omega|_n)\bigr)^{1/n}.
\]
Thus the family $\cF$ is sampled once and then kept fixed; the only randomness in $\rho_n$ comes from the path $\omega$, exactly as in the rest of the paper.

To quantify noncommutativity we use the normalized commutator:
\begin{equation} \label{eq:non_commutator_level}
\gamma(\cF)
:=
\frac{2}{m(m-1)}
\sum_{1\le i<j\le m}
\frac{\|[A_i,A_j]\|_F}{\|A_i\|_F\|A_j\|_F}.
\end{equation}
This is only an empirical quantity, but it is useful for comparing families of different sizes and scales. For $M$ independent trajectories $\omega^{(1)},\dots,\omega^{(M)}$, write
\[
\rho_n^{(b)}:=\rho_n(\omega^{(b)}),\qquad b=1,\dots,M,
\]
and define the empirical center, scaled empirical variance, and empirical tail probability as follows:
\begin{equation} \label{eq:empirical_estimators}
\hat m_n
:=
\frac1M\sum_{b=1}^M \rho_n^{(b)},
\qquad
\hat v_n
:=
n\,\widehat{\operatorname{Var}} \bigl(\rho_n^{(1)},\dots,\rho_n^{(M)}\bigr),
\qquad
\hat p_n(r)
:=
\frac1M\sum_{b=1}^M \mathbf 1_{\{\rho_n^{(b)}\ge r\}}.
\end{equation}
For sample-path plots we also use a long-horizon reference level
\[
\hat\rho_\star:=\hat m_{N_\star},
\]
computed at a large fixed $N_\star$. This is only an empirical reference level; unlike in the diagonal or triangular settings, no closed-form asymptotic center is assumed to be known here. The main figures use the following directly sampled dense families.

\smallskip
\noindent
\emph{Family A.} Three $3\times3$ matrices with independent lognormal entries, rescaled so that each individual spectral radius lies in $[1.15,1.85]$, with
\[
\mathbf p=(0.30,0.40,0.30),
\qquad
\gamma(\cF_A)\approx 0.487.
\]

\smallskip
\noindent
\emph{Family B.} Four $4\times4$ matrices with independent shifted-Gaussian entries $\mathcal N(0.25,0.7^2)$, again rescaled to a moderate spectral-radius range, with
\[
\mathbf p=(0.25,0.20,0.25,0.30),
\qquad
\gamma(\cF_B)\approx 0.654.
\]

\medskip
\noindent\textbf{LLN.}
\Cref{fig:sample_paths} shows $25$ individual trajectories of $\rho_n(\omega)$ for Family~A. The paths fluctuate visibly at short lengths, but by $n\approx 200$ they have already entered a narrow band around the empirical reference level $\hat\rho_\star$. This is the same concentration mechanism that underlies the almost-sure law of large numbers in the triangular setting, now observed for a genuinely dense noncommuting family.

\begin{figure}[ht!]
  \centering
  \includegraphics[width=0.7\linewidth]{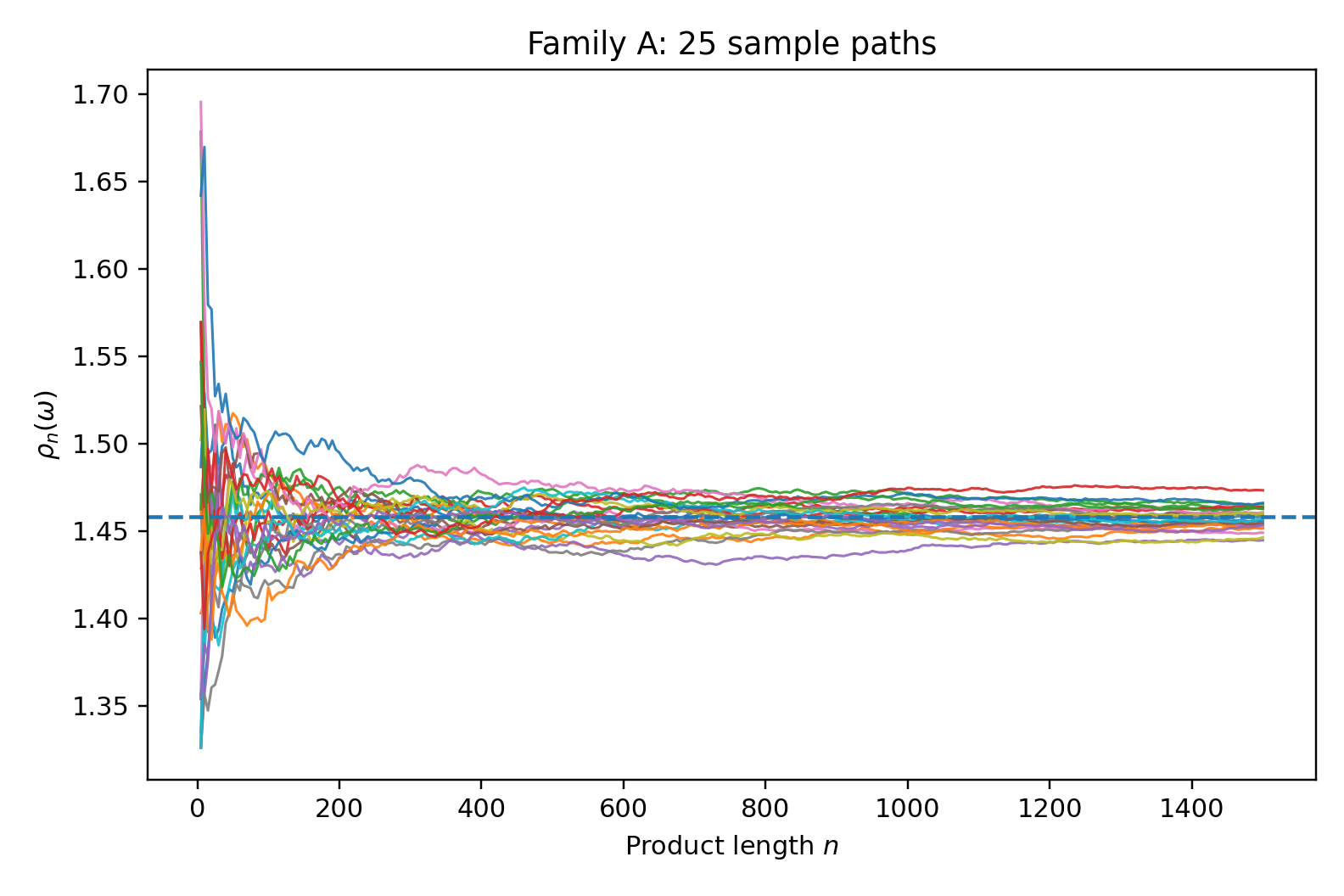}
  \caption[Sample paths of \(\rho_n\) for a generic dense family]%
  {\protect\parbox[t]{.93\textwidth}{%
  Twenty-five independent paths $\rho_n(\omega)$ for Family~A. The dashed line is the empirical reference level $\hat\rho_\star=\hat m_{1500}\approx1.4584$, estimated from $8000$ independent trajectories.}}
  \label{fig:sample_paths}
\end{figure}

\medskip
\noindent\textbf{CLT.}
To test the shape of the fluctuations, we form the standardized empirical variable
\begin{equation} \label{eq:empirical_standardization}
Z_n^{(b)}
:=
\frac{\sqrt n\,\bigl(\rho_n^{(b)}-\hat m_n\bigr)}{\sqrt{\hat v_n}},
\qquad b=1,\dots,M.
\end{equation}
\Cref{fig:qq_plots} displays QQ-plots of $Z_n^{(b)}$ against standard Gaussian quantiles for both Family~A and Family~B, at two values of $n$. The agreement is clear in all four panels and improves slightly as $n$ increases, which is consistent with a Gaussian fluctuation picture even though no general theorem is presently available for these examples.

\begin{figure}[ht!]
  \centering
  \includegraphics[width=0.7\linewidth]{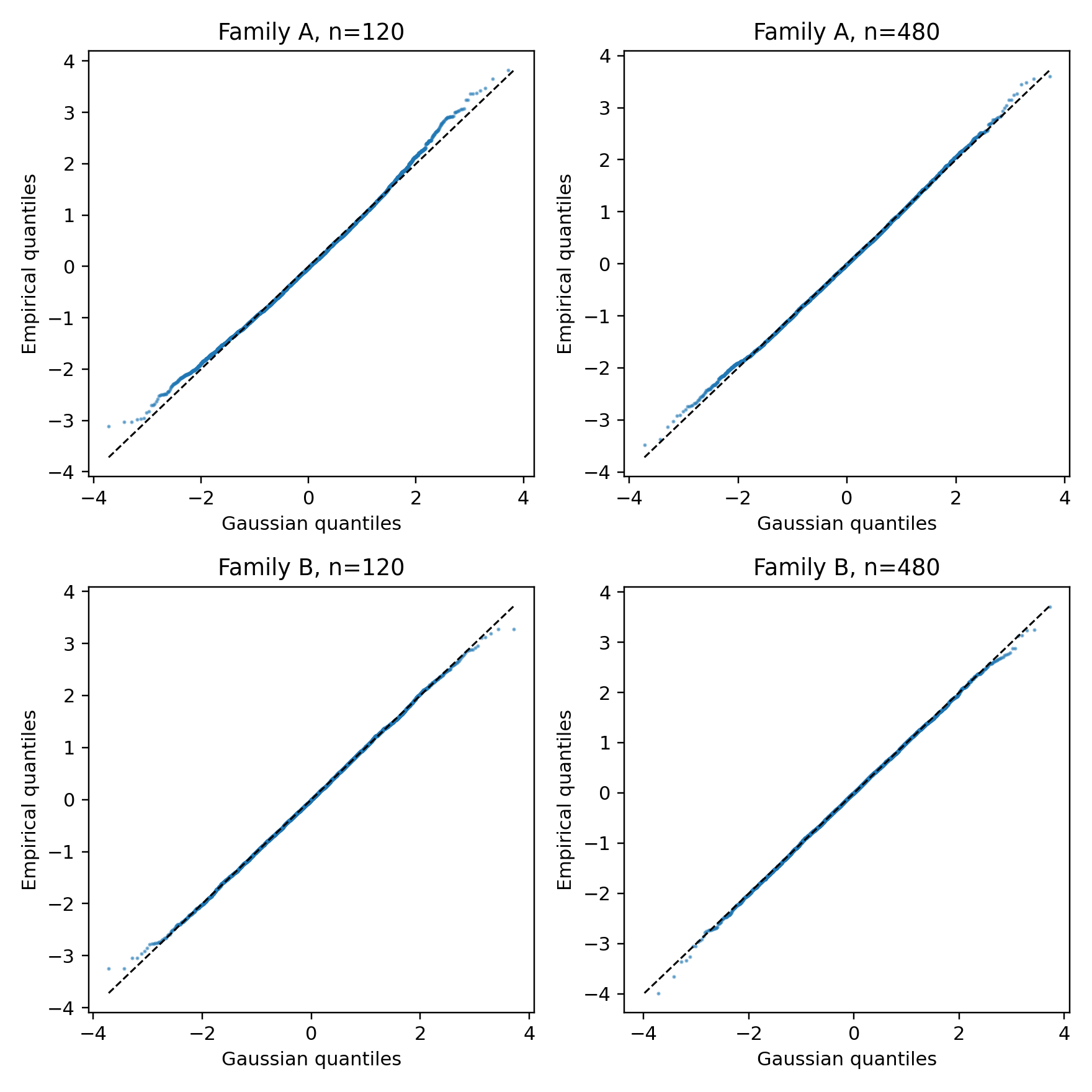}
  \caption[QQ-plots for empirical CLT diagnostics]%
  {\protect\parbox[t]{.93\textwidth}{%
  QQ-plots of the standardized samples $Z_n^{(b)}$ from~\eqref{eq:empirical_standardization} against $\mathcal N(0,1)$ quintiles, with $M=5000$ trajectories per panel. Top row: Family~A. Bottom row: Family~B. The same empirical center $\hat m_n$ and scaled variance $\hat v_n$ from~\eqref{eq:empirical_estimators} are used for the standardization.}}
  \label{fig:qq_plots}
\end{figure}

\medskip
\noindent\textbf{LDP.}
Finally,~\Cref{fig:ldp_tails} investigates fixed-threshold tail events. For several \emph{absolute} thresholds $r$, chosen once from runs and then held fixed for all $n$, we plot
\[
\log \hat p_n(r)
\quad\text{against}\quad n.
\]
In both Family~A and Family~B the resulting curves are close to linear over the displayed range, which is the qualitative behavior one would expect from a large-deviation principle. Since no explicit rate function is known in this setting, the dashed lines in the figure are simple least-squares fits.

\begin{figure}[ht!]
  \centering
  \includegraphics[width=\linewidth]{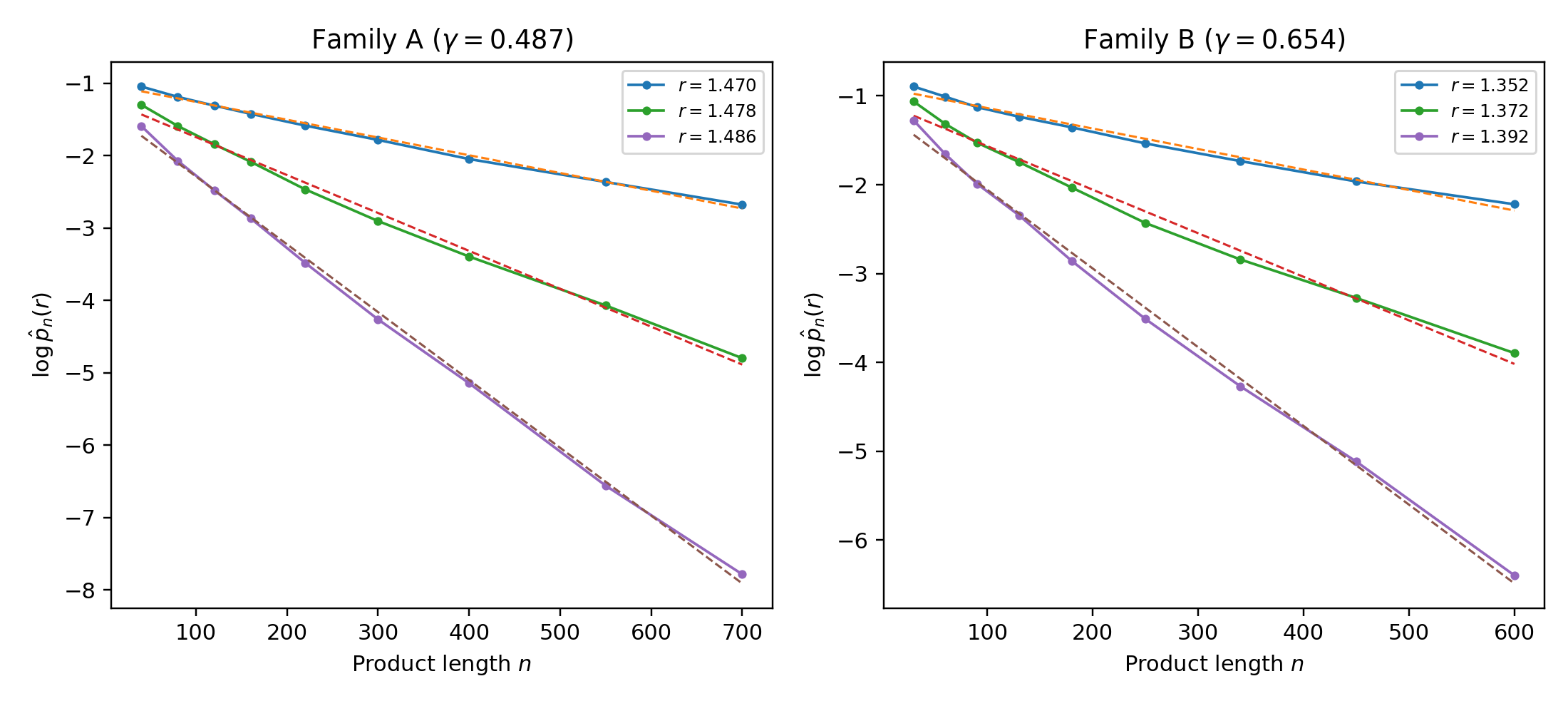}
  \caption[Empirical fixed-threshold tails for generic dense families]%
  {\protect\parbox[t]{.93\textwidth}{%
  Empirical fixed-threshold tails $\log \hat p_n(r)$ for Families~A and~B, based on $M=12000$ trajectories at each value of $n$. The thresholds are absolute constants chosen once from pilot runs and then kept fixed throughout the plot: $r\in\{1.470,1.478,1.486\}$ for Family~A and $r\in\{1.352,1.372,1.392\}$ for Family~B. The dashed lines are least-squares fits and are included only to highlight the approximately linear decay.}}
  \label{fig:ldp_tails}
\end{figure}

\section{Conclusion} 

In this article, we introduced the \emph{random spectral radius}, a new joint spectral characteristic associated with a bounded family of matrices. We established rigorous asymptotic results, including a law of large numbers, central limit theorems, finite-size asymptotic expansions, and a large deviation principle, for commuting and triangular families, together with perturbative stability results for nearby families.

The numerical experiments for generic dense families suggest that several of these asymptotic phenomena may persist well beyond the structured setting treated theoretically here; however, providing a rigorous theory for fully noncommuting families remains an open and challenging problem.

Another important class of matrix families is that of nonnegative matrices. For a nonnegative matrix, Perron--Frobenius theory guarantees that the spectral radius is a nonnegative eigenvalue; under stronger irreducibility or primitivity assumptions, this leading eigenvalue is simple and dominant. It would therefore be natural to investigate the random spectral radius in this setting as well. For random products of nonnegative transition matrices, the logarithmic growth rate is closely related to the top Lyapunov exponent and hence to typical growth or decay under randomness.

A natural extension of the present framework concerns \emph{Markovian switching}, where the index sequence $\sigma = (\sigma_1, \sigma_2, \ldots)$ is generated by a Markov chain with transition matrix $P$ rather than by i.i.d.\ sampling. While the approach of~\cite{CMS21} addresses asymptotic growth rates in this setting, the question of whether the CLT and, more ambitiously, the Edgeworth-type refinements developed here extend remains open. The loss of independence among the $Z_k$ random variables introduces correlations that may affect both the asymptotic variance and the higher-order correction terms.

For diagonal and triangular families, the large deviation principle established in~\cref{subsec:large_deviations} quantifies the probability of rare events, for instance the likelihood that $\rho_n$ approaches the JSR or the LSR, thereby directly addressing the motivating questions posed in the introduction. Extending the LDP to general noncommuting families remains an open problem. Any progress in these directions would represent an important step toward a comprehensive understanding of the random spectral radius for general matrix families.

\subsection*{Acknowledgements}

We thank Kalle Koskinen for the careful reading of the manuscript and valuable comments. 

Nicola Guglielmi and Francesco Paolo Maiale acknowledge that their research was supported by funds from the Italian 
MUR (Ministero dell'Universit\`a e della Ricerca) within the 
PRIN 2022 Project ``Advanced numerical methods for time dependent parametric partial differential equations with applications'' and the PRIN-PNRR Project ``FIN4GEO''.
Nicola Guglielmi is also affiliated to the INdAM-GNCS (Gruppo Nazionale di Calcolo Scientifico).

\begin{appendices}

\makeatletter
\renewcommand{\appendixname}{}
\renewcommand{\thesection}{\Alph{section}}
\renewcommand{\@seccntformat}[1]{%
  \csname the#1\endcsname\quad
}
\makeatother

\section{Additional numerical experiments}\label{app:numerics}

This appendix collects further empirical evidences for the directly sampled dense families introduced in~\cref{sec:numerics}. No diagonal benchmark and no perturbative reference model is used anywhere in this appendix; all quantities are computed directly from Monte Carlo samples of the random spectral radius.

\subsection{Approximate Gaussian profiles}

\Cref{fig:generic_histograms} shows histograms of $\rho_n$ for Families~A and~B at the two representative lengths $n = 80$ and $n=400$. In each panel we superimpose the Gaussian density with the same empirical mean and variance as the histogram. The fit is already good at $n=80$ and becomes visibly tighter by $n=400$, which complements~\cref{fig:qq_plots}.

\begin{figure}[ht!]
  \centering
  \includegraphics[width=.85\linewidth]{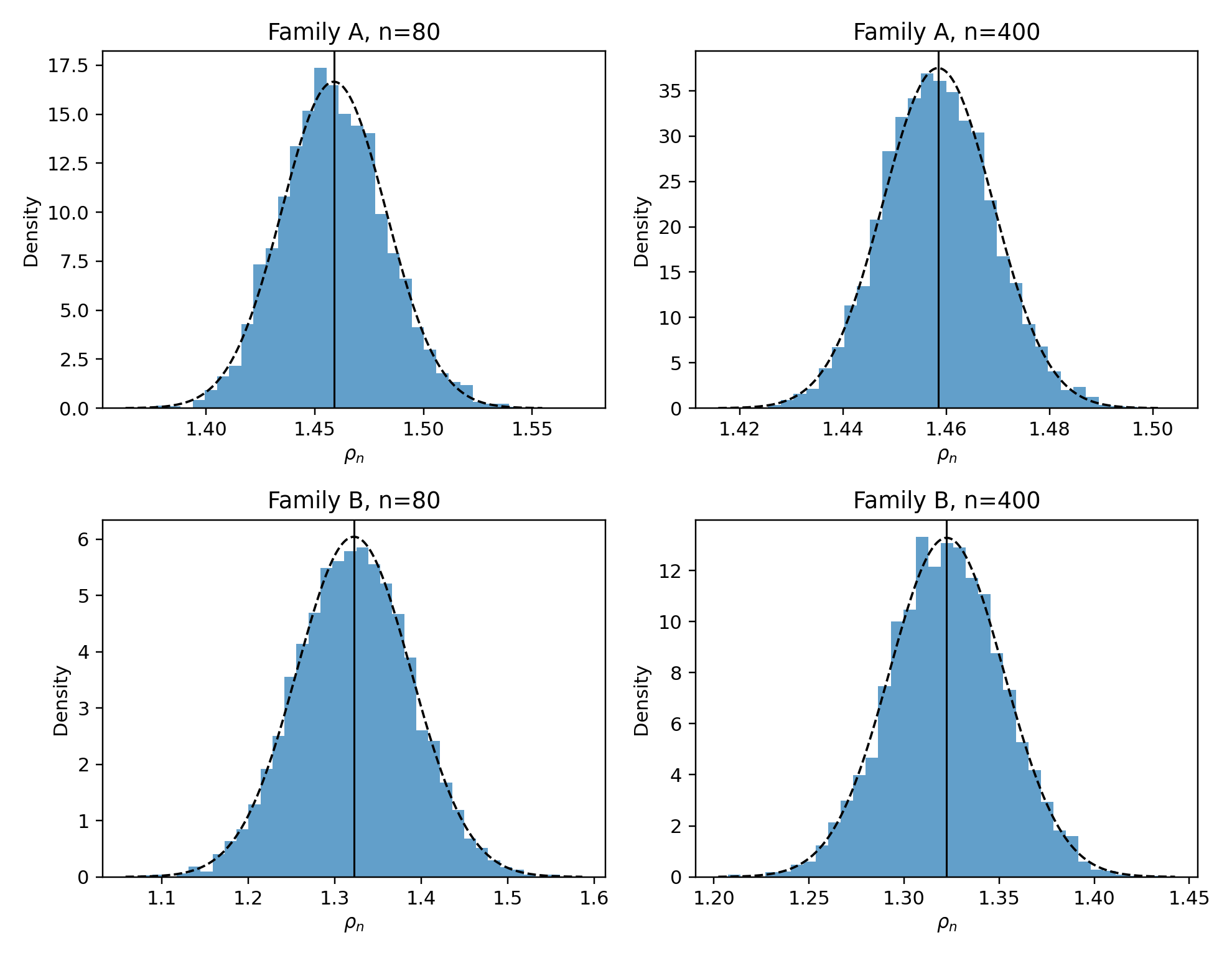}
  \caption[Histograms for Fam.~A and~B]%
  {\protect\parbox[t]{.93\textwidth}{%
  Histograms of $\rho_n$ for Fam.~A and~B at $n=80$ and $n=400$, each based on $M=6000$ trajectories. The dashed curve in each panel is the Gaussian density with the same empirical mean and variance as the corresponding histogram.}}
  \label{fig:generic_histograms}
\end{figure}

\subsection{Stabilization of the empirical center and variance}

The quantities $\hat m_n$ and $\hat v_n$ defined in~\eqref{eq:empirical_estimators} provide a useful summary of the empirical LLN/CLT behavior.~\Cref{fig:mean_variance_generic} shows that in both Families~A and~B the center $\hat m_n$ stabilizes quickly, while the scaled variance $n\,\widehat{\operatorname{Var}}(\rho_n)$ approaches an approximately constant level. This is exactly the pattern one would expect if
\[
\rho_n \approx \rho_\star + \frac{\sigma_\star}{\sqrt n}\,G
\]
for a suitable effective center $\rho_\star$ and fluctuation scale $\sigma_\star$.

\begin{figure}[ht!]
  \centering
  \includegraphics[width=.95\linewidth]{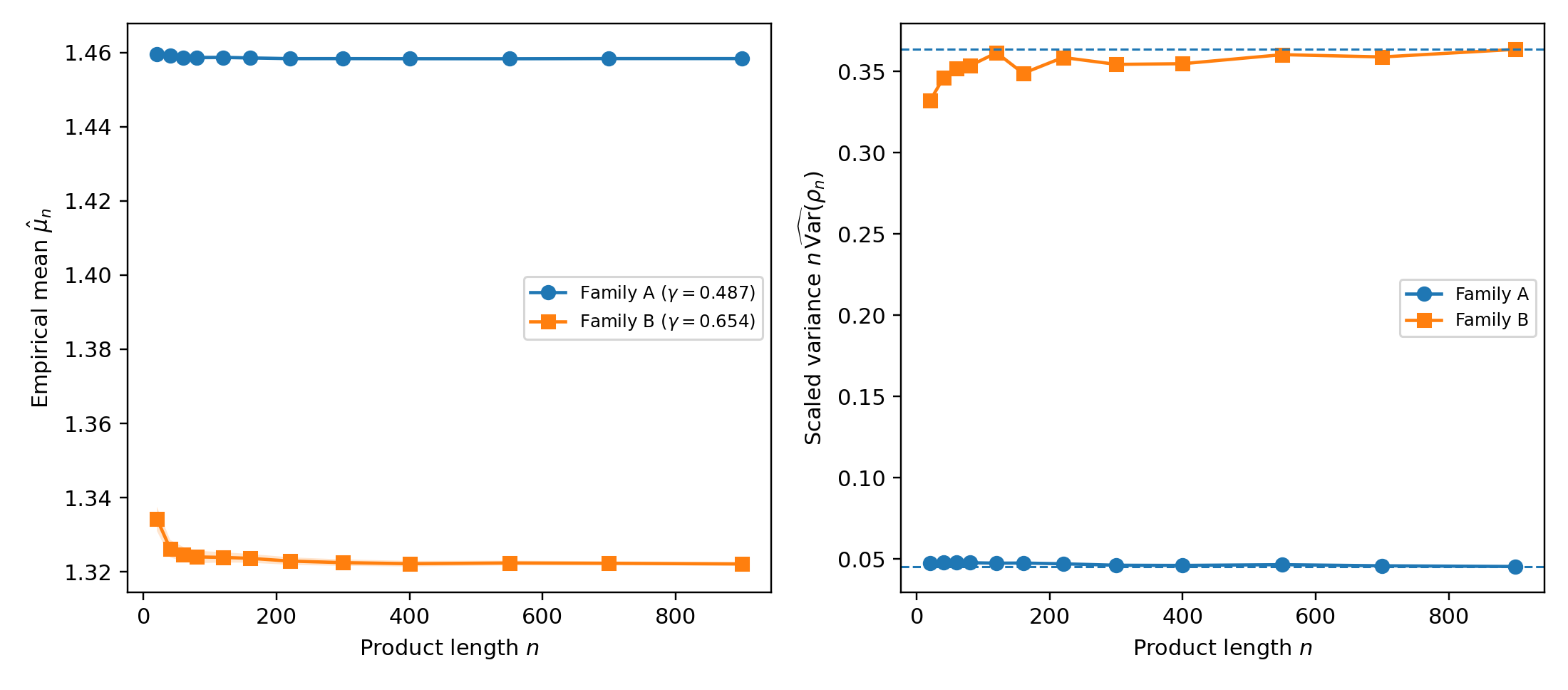}
  \caption[Empirical center and scaled variance]%
  {\protect\parbox[t]{.93\textwidth}{%
  \textbf{Left}: empirical mean $\hat m_n$ for Fam.~A and~B; the shaded bands are $95\%$ MC confidence bands for the mean. \textbf{Right}: scaled variance $\hat v_n=n\,\widehat{\operatorname{Var}}(\rho_n)$. In both the center stabilizes rapidly and the scaled variance flattens to an approximately constant level.}}
  \label{fig:mean_variance_generic}
\end{figure}

\subsection{A higher-dimensional family}

The next experiment shows that the same qualitative behavior persists in larger dimensions. We sample a third family directly from a dense lognormal ensemble:

\smallskip
\noindent
\emph{Family C.} Five $5\times5$ matrices with independent lognormal entries, rescaled to a moderate spectral-radius range, with
\[
\mathbf p=(0.15,0.20,0.25,0.20,0.20),
\qquad
\gamma(\cF_C)\approx 0.389.
\]

\Cref{fig:large_matrices} shows $15$ sample paths together with a QQ-plot at $n=500$. The concentration of the paths and the near-linear QQ-plot are very similar to what we observed in the lower-dimensional families.

\begin{figure}[ht!]
  \centering
  \includegraphics[width=.95\linewidth]{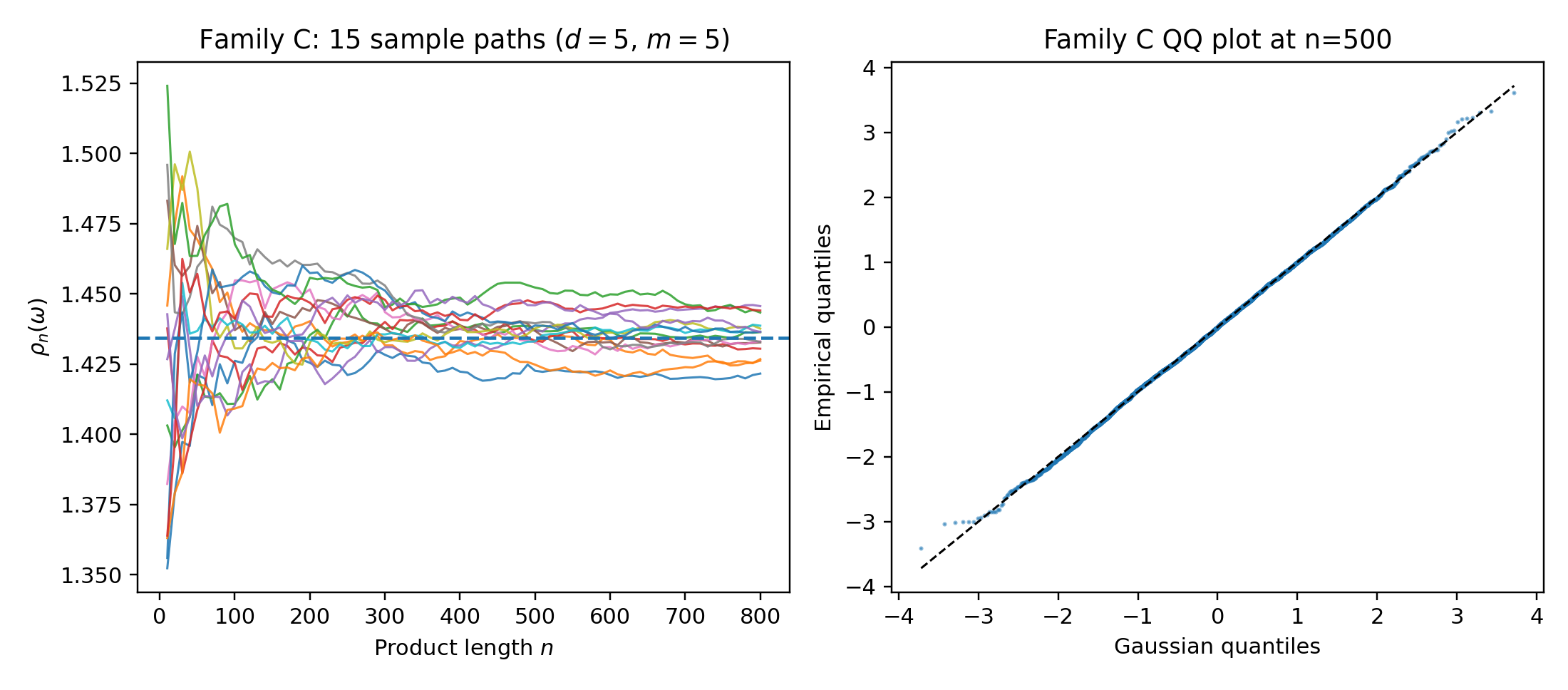}
  \caption[Higher-dimensional generic family]%
  {\protect\parbox[t]{.93\textwidth}{%
  \textbf{Left}: $15$ sample paths of $\rho_n(\omega)$, with dashed line $\hat m_{800}\approx1.4344$. \textbf{Right}: QQ-plot of the standardized variable $Z_{500}^{(b)}$ from~\eqref{eq:empirical_standardization}, based on $M=5000$ trajectories.}}
  \label{fig:large_matrices}
\end{figure}

\subsection{Dependence on the sampling law}

For $\cF$ fixed, the law of the switching sequence enters only through the probability $\mathbf p\in\Delta_m$. To visualize this dependence in a generic noncommuting example,~\cref{fig:simplex} plots the empirical landscape
\[
\mathbf p\longmapsto \hat m_{250}(\mathbf p)
\]
over the simplex $\Delta_3$ for Family~A. Each grid value is estimated independently by Monte Carlo, using $M=1200$ trajectories on a uniform simplex mesh of size $0.1$. The resulting landscape is smooth and strongly anisotropic; on the displayed grid the smallest and largest values occur at the vertices $(0,1,0)$ and $(1,0,0)$, respectively. We emphasize, however, that this figure is purely empirical and does \textbf{not} assert convexity, variational formulas, or any exact relation to the JSR/LSR in the general setting.

\begin{figure}[ht!]
  \centering
  \includegraphics[width=.60\linewidth]{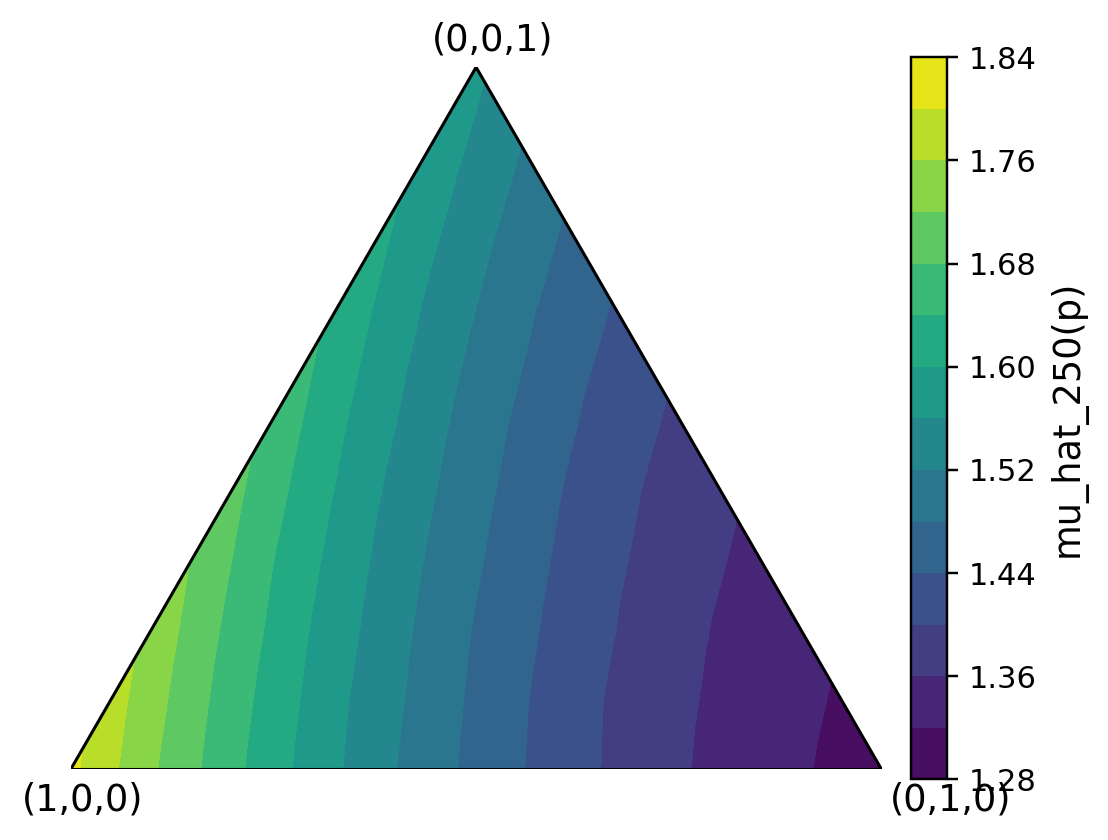}
  \caption[Empirical dependence on \(\mathbf p\) over the simplex]%
  {\protect\parbox[t]{.93\textwidth}{%
  Empirical landscape $\mathbf p\mapsto \hat m_{250}(\mathbf p)$ for Family~A over the probability simplex $\Delta_3$.}}
  \label{fig:simplex}
\end{figure}

\end{appendices}

\bibliographystyle{siam} 
\bibliography{references.bib}

\end{document}